\documentclass[11pt,twoside]{article}

\usepackage{amsmath, amsthm,amsfonts,amssymb,mathtools,mathdots,scalerel,mathrsfs,stmaryrd,cmll,bbold}
\usepackage[utf8]{inputenc}
\usepackage[T1]{fontenc}
\usepackage{mlmodern, tikz-cd, ahquiver}
\usepackage[many]{tcolorbox}

\usepackage{enumitem}
\usepackage[hidelinks]{hyperref}
\usepackage{tikz, fancyhdr, xparse, xcolor, tocloft}
\usepackage[british]{babel}
\usepackage[a4paper, top=4.5cm, bottom=4.5cm, left=4cm, right=4cm, asymmetric]{geometry}

\usepackage{crossreftools}
\usepackage[textwidth=3cm, textsize=small, colorinlistoftodos]{todonotes}

\usepackage{ahtitle, titlesec}
\usepackage{etoolbox}
\makeatletter
\patchcmd{\ttlh@hang}{\parindent\z@}{\parindent\z@\leavevmode}{}{}
\patchcmd{\ttlh@hang}{\noindent}{}{}{}
\makeatother

\newcommand\eqdef{\coloneqq}
\newcommand\nbd{\nobreakdash-\hspace{0pt}}
\newcommand\idd[1]{\mathrm{id}_{#1}}
\newcommand\Idd[1]{\mathrm{Id}_{#1}}
\newcommand\invrs[1]{#1^{-1}}

\newcommand\incl{\hookrightarrow}
\newcommand\iso{\overset{\sim}{\rightarrow}}

\newcommand\surj{\twoheadrightarrow}
\newcommand\sd{\looparrowright}

\newcommand\restr[2]{{#1}{\raisebox{0pt}{$|_{#2}$}}}

\newcommand\set[1]{\left\{ {#1} \right\}}
\newcommand\order[2]{#2^{(#1)}}

\newcommand{\overbar}[1]{{\mkern 1.5mu\overline{\mkern-1.5mu#1\mkern-1.5mu}\mkern 1.5mu}}

\newcommand{\pout}[1]{\amalg_{#1}}
\newcommand{\pback}[1]{\times_{#1}}

\newcommand\Ninfty{\mathbb{N} \cup \set{\infty}}

\newcommand\slice[2]{{#1}/{\raisebox{-2pt}{$#2$}}}
\newcommand\join{\,{\star}\,}

\newcommand\opp[1]{#1^\mathrm{op}}

\newcommand\cat[1]{\underline{\mathrm{#1}}}
\newcommand\core{\mathsf{core}}

\newcommand\fun[1]{\mathsf{#1}}
\DeclareMathOperator*{\colim}{colim}
\DeclareMathOperator{\Ob}{Ob}
\DeclareMathOperator{\Hom}{Hom}

\DeclareMathOperator{\PSh}{PSh}
\DeclareMathOperator{\Lan}{Lan}

\newcommand{\init}{\mathbb{0}}
\newcommand{\term}{\mathbb{1}}

\newcommand\ogPos{\cat{ogPos}}

\newcommand\Pos{\cat{Pos}}
\newcommand{\Cat}{\cat{Cat}}
\newcommand{\Set}{\cat{Set}}

\newcommand{\ssSet}{\cat{ssSet}}
\newcommand{\omegaCat}{\omega\cat{Cat}}

\newcommand\ssimcat{\Delta_i}
\newcommand\simplex[1]{\vec{\Delta}^{#1}}

\newcommand{\dCpx}{\cat{dCpx}}
\newcommand{\rdCpx}{\dCpx^\mathsf{reg}}
\newcommand{\frdCpx}{\dCpx^\mathsf{reg,fin}}

\newcommand{\mdCpx}{\dCpx^\m}
\newcommand{\mssSet}{\ssSet^\m}

\newcommand{\ICpx}{\Infl\cat{Cpx}}
\newcommand{\mICpx}{\Infl\cat{Cpx}^\m}
\newcommand{\MCpx}{\Merg\cat{Cpx}}
\newcommand{\mMCpx}{\Merg\cat{Cpx}^\m}
\newcommand{\MICpx}{\Merg\!\Infl\cat{Cpx}}
\newcommand{\mMICpx}{\Merg\!\Infl\cat{Cpx}^\m}

\newcommand{\UInfl}{\fun{U}_{\!\Infl}}
\newcommand{\FInfl}{\fun{F}_{\!\Infl}}
\newcommand{\UmInfl}{\fun{U}^\m_{\!\Infl}}

\newcommand{\UMerg}{\fun{U}_{\!\Merg}}
\newcommand{\UmMerg}{\fun{U}^\m_{\!\Merg}}
\newcommand{\FMerg}{\fun{F}_{\!\Merg}}
\newcommand{\FmMerg}{\fun{F}^\m_{\!\Merg}}

\newcommand{\Fdelta}{i_{\Delta}}

\newcommand{\atom}{{\scalebox{1.3}{\( \odot \)}}}

\DeclareMathOperator{\clos}{cl}
\newcommand\clset[1]{\mathrm{cl}\set{#1}}

\newcommand\ordcpx[1]{{#1}^\Delta}

\newcommand\conv[1]{#1_\mathrm{c}}

\newcommand\augm[1]{{{#1}_\bot}}
\newcommand\dimin[1]{{#1}_{\not\bot}} 

\newcommand\augmo[1]{{#1}^+}

\newcommand\maxel[1]{\mathscr{M}\!\mathit{ax}\,#1}
\newcommand\gr[2]{#2_{#1}}

\newcommand\bd[2]{\partial_{#1}^{#2}}
\newcommand{\bdmap}{\partial}

\newcommand\faces[2]{\Delta_{#1}^{#2}}
\newcommand\cofaces[2]{\nabla_{#1}^{#2}}
\DeclareMathOperator{\inter}{int}

\newcommand\cp[1]{\,{\scriptstyle\#}_{#1}\,}

\newcommand\cpsub[1]{\triangleright_{#1}}
\newcommand\subcp[1]{\prescript{}{#1\!}{\triangleleft}\;}

\newcommand\subs[3]{#1[#2/#3]}

\newcommand\mrg[1]{\left\langle#1\right\rangle}

\newcommand\celto{\Rightarrow}
\newcommand\rdto{\celto^+}

\newcommand\submol{\sqsubseteq}

\newcommand\gray{\otimes}

\newcommand\molec{\mathit{Mol}}
\newcommand\molecin[1]{\slice{\molec}{#1}}

\newcommand{\cyl}{\curvearrowright} 

\newcommand\inl[1]{{#1} \join }
\newcommand\inr[1]{\join {#1} }

\newcommand{\m}{\mathsf{m}}
\newcommand{\mrk}[1]{#1^{\m}}
\newcommand{\flatm}[1]{#1^{\flat}}
\newcommand{\sharpm}[1]{#1^{\sharp}}

\newcommand{\Um}{\fun{U}_\m}

\newcommand{\Gammacon}{\Gamma^\mathsf{c}}

\newcommand\pcyl[1]{\ltimes_{#1}}

\newcommand\meqv{\simeq_\m}

\newcommand\mcelto{\celto_\m}
\newcommand\mrdto{\rdto_\m}

\newcommand\I{\mathcal{I}}

\newcommand\satur[1]{#1_{\mathsf{inv}}}
\newcommand\clcom[1]{\overbar{#1}}

\DeclareMathOperator{\Pd}{Pd}
\DeclareMathOperator{\Rd}{Rd}

\DeclareMathOperator{\cell}{cell}
\DeclareMathOperator{\dgn}{dgn}

\newcommand{\un}{\varepsilon}
\newcommand{\lun}[1]{\lambda_{#1}}
\newcommand{\run}[1]{\rho_{#1}}

\newcommand{\Infl}{\mathscr{I}}
\newcommand{\mInfl}{\Infl_\m}
\newcommand{\Merg}{\mathscr{M}}

\newcommand\glcom[1]{\llparenthesis{#1}\rrparenthesis}
\makeatletter
\newsavebox{\@brx}
\newcommand{\llangle}[1][]{\savebox{\@brx}{\(\m@th{#1\langle}\)}%
  \mathopen{\copy\@brx\kern-0.5\wd\@brx\usebox{\@brx}}}
\newcommand{\rrangle}[1][]{\savebox{\@brx}{\(\m@th{#1\rangle}\)}%
  \mathclose{\copy\@brx\kern-0.5\wd\@brx\usebox{\@brx}}}
\makeatother
\newcommand\mrgcom[1]{\llangle{#1}\rrangle}

\newcommand{\Rnd}{\mathscr{R}}

\newcommand{\adCpx}{\augmo{\dCpx}}

\newcommand{\Cof}{\mathcal{C}\!\mathit{of}}
\newcommand{\ACof}{\mathcal{AC}\!\mathit{of}}
\newcommand{\Fib}{\mathcal{F}\!\mathit{ib}}
\newcommand{\AFib}{\mathcal{AF}\!\mathit{ib}}

\newcommand{\ICof}[1]{#1\text{-}\Cof}
\newcommand{\JFib}[1]{#1\text{-}\Fib}

\newcommand{\ppnat}{\mathbin{\Box}}

\DeclareMathOperator{\rlp}{r}
\DeclareMathOperator{\llp}{l}

\DeclareMathOperator{\Ho}{Ho}

\newcommand\acof{\overset{\sim}{\incl}}
\newcommand\afib{\overset{\sim}{\surj}}

\newcommand\rcodiag[1]{\nabla_{#1}}
\newcommand\rdiag[1]{\Delta_{#1}}

\newcommand\cufs[1]{#1^{\mathsf{c}\cup\mathsf{f}}}
\newcommand\fibs[1]{#1^{\mathsf{fib}}}
\newcommand\cofs[1]{#1^{\mathsf{cof}}}
\newcommand\bifs[1]{#1^{\mathsf{bif}}}

\newcommand\homo{\approx}

\newcommand\Ibd{I_\partial}
\newcommand\Imrg{I_{\langle\rangle}}
\newcommand\Imrk{I_\m}

\newcommand\Jhorn{J_\mathsf{horn}}
\newcommand\Jsat{J_\mathsf{sat}}
\newcommand\Jn{J_n}

\newcommand\Mw{\mathfrak{M}}
\newcommand\Nw{\mathfrak{N}}

\newcommand\Mwn{\Mw_n}

\newcommand\MMwn{\Mw_{\!\Merg, n}}

\newcommand{\MwKan}[1]{\mathfrak{K}_{#1}}

\newcommand\arr{\vec{I}}
\newcommand\marr{\mrk{\arr}}
\newcommand\globe[1]{O^{#1}}

\renewcommand{\a}{\alpha}
\renewcommand{\b}{\beta}

\renewcommand{\S}{\mathcal{S}}

\newcommand{\C}{\mathcal{C}}
\renewcommand{\L}{\mathcal{L}}
\newcommand{\E}{\mathcal{E}}
\newcommand{\R}{\mathcal{R}}

\renewcommand{\th}{\vartheta}

\newcommand\cls[1]{\mathcal{#1}}

\makeatletter
\newcommand{\oset}[3][0ex]{%
  \mathrel{\mathop{#3}\limits^{
    \vbox to#1{\kern-1.5\ex@
    \hbox{$\scriptstyle#2$}\vss}}}}
\makeatother
\newcommand\qeq{\oset{?}{=}}

\newcommand\ntext{\texorpdfstring{$n$}{n}}
\newcommand\inftyn{\texorpdfstring{$(\infty, n)$}{(infty, n)}}
\newcommand\inftyinf{\texorpdfstring{$(\infty, \infty)$}{(infty, infty)}}

\newtheoremstyle{ittheorem}
  {\topsep}   
  {\topsep}   
  {\itshape}  
  {0pt}       
  {\sffamily \itshape \bfseries} 
  { ---}         
  {5pt plus 1pt minus 1pt} 
  {}          

\newtheoremstyle{itdfn}
  {\topsep}   
  {\topsep}   
  {}  
  {0pt}       
  {\sffamily \itshape \bfseries} 
  {}         
  {5pt plus 1pt minus 1pt} 
  {\thmnumber{#2}{\thmnote{\normalfont\ \ %
  {\sffamily(#3)}.}}}          

\newtheoremstyle{itrmk}
  {0.5\topsep}   
  {0.5\topsep}   
  {\normalfont}  
  {0pt}       
  {\sffamily \itshape} 
  { ---}         
  {5pt plus 1pt minus 1pt} 
  {}          
  
\newtheoremstyle{itexm}
  {0.5\topsep}      
  {0.5\topsep}      
  {\normalfont}     
  {0pt}             
  {\sffamily \itshape \bfseries \color{\mycolor}}        
  {\\}               
  {5pt plus 1pt minus 1pt} 
  {\thmname{#1} \thmnumber{#2}{\thmnote{\normalfont\ \ %
  {\sffamily(#3)}.}}}           
  
\makeatletter
  \renewcommand\@upn{\textit}
\makeatother

\theoremstyle{ittheorem}
\newtheorem{thm}{Theorem}[section]
\newtheorem*{thm*}{Theorem}
\newtheorem{prop}[thm]{Proposition}
\newtheorem*{prop*}{Proposition}
\newtheorem{cor}[thm]{Corollary}
\newtheorem{lem}[thm]{Lemma}
\newtheorem{conj}[thm]{Conjecture}
\theoremstyle{itdfn}
\newtheorem{dfn}[thm]{}
\theoremstyle{itrmk}
\newtheorem{rmk}[thm]{Remark}
\newtheorem{comm}[thm]{Comment}
\newtheorem{exm}[thm]{Example}

\setlist{leftmargin=20pt,itemsep=0pt,parsep=0pt,topsep=1ex}

\titleformat{\section}
 {\large\scshape}{\thesection.}{1em}{}

\titleformat{\subsection}
 {\normalsize\itshape}{\thesubsection.}{1em}{}
\titlespacing*{\subsection}
{0pt}{1.5ex plus 1ex minus .2ex}{1.5ex plus .2ex}

\renewcommand{\cftsubsecpagefont}{\mdseries}

\makeatletter \renewcommand{\cftsubsecfillnum}[1]{%
  {\cftsubsecleader}\nobreak
  \makebox[\@pnumwidth][\cftpnumalign]{\cftsubsecpagefont \oldstylenums{#1}}\cftsubsecafterpnum\par
} \makeatother
\makeatletter \renewcommand{\cftsecfillnum}[1]{%
  {\cftsecleader}\nobreak
  \makebox[\@pnumwidth][\cftpnumalign]{}\cftsecafterpnum\par
} \makeatother

\AtBeginDocument{\addtocontents{toc}{\protect\thispagestyle{fancy}}}

\usetikzlibrary{arrows, decorations.markings, shapes.geometric, decorations.pathmorphing, decorations.pathreplacing, intersections, patterns,calc, backgrounds}

\tikzset{
	0c/.style={circle, draw, fill, inner sep=.7pt},
	1c/.style={->, shorten <=2pt, shorten >=2pt},
	u1c/.style={-, shorten <=2pt, shorten >=2pt},
	2c/.style={double, shorten <=6pt, shorten >=8pt, decoration={markings,mark=at position -6pt with {\arrow[scale=1.75]{>}}}, preaction={decorate}},
	follow/.style={->, >=stealth, ultra thick, shorten <=3pt, shorten >=3pt, color=gray!70},
	arlabel/.style={scale=.8}
}

\newcommand\runtitle{semi-strictification of \inftyn-categories}
\newcommand\runauthor{chanavat and hadzihasanovic}

\title{Semi-strictification of \inftyn-categories}

\author{Cl\'emence Chanavat and Amar Hadzihasanovic}

\institution{Tallinn University of Technology}

\begin{document}

{$\quad$}

\vspace{20pt}

\maketitle 

\noindent\makebox[\textwidth][r]{%
	\begin{minipage}[t]{.7\textwidth}
\small \emph{Abstract.}
	We prove the first equivalence between a weak non-algebraic model and a semi-strict algebraic model of $(\infty, n)$-categories.
	This takes the form of a natural semi-strictification, whereby a weak $(\infty, n)$-category is embedded into a semi-strict one through an acyclic cofibration, in such a way that weak functors lift to semi-strict functors; this constitutes the derived unit of a Quillen equivalence between weak model categories whose fibrant objects are, respectively, the weak $(\infty, n)$-categories and (up to an acyclic fibration) the semi-strict ones.
	The semi-strict model has algebraic units and composition of round pasting diagrams, satisfying a strict form of associativity and interchange as in Henry's regular version of Simpson's weak units conjecture; semi-strict functors strictly preserve round composition, but only weakly preserve units.
	Globular composition operations are obtained from a combination of units and round composition.
	Since the models satisfy the homotopy hypothesis in the case $n = 0$, this result also exhibits the first semi-strict model of the classical homotopy types that has algebraic units and composition.
	The constructions are based on the combinatorics of regular directed complexes and are entirely explicit and combinatorial, in the spirit of Mac Lane's strictification of bicategories.
\end{minipage}}

\vspace{20pt}

\makeaftertitle

\normalsize \thispagestyle{empty} 

\clearpage 

{$\quad$}

\vspace{25pt}

\noindent\makebox[\textwidth][c]{%
\begin{minipage}[t]{.85\textwidth}
\setcounter{tocdepth}{2}
\tableofcontents
\end{minipage}}

\clearpage

\section*{Introduction}
\addcontentsline{toc}{subsection}{Introduction}

\noindent
Since the early days of higher category theory, the notion of $n$\nbd category, and of \inftyn\nbd category, has undergone a kind of de-concretisation: first of all, it is accepted, and with good reason, that there is no uniquely defined mathematical structure (in the traditional sense, that is, unless one accepts foundations that are intrinsically higher-categorical) which captures it; secondly, it seems to also be gaining acceptance that there will be no ``standard model'' that serves as a reference for others---at least, not for essential, rather than contingent reasons---the way that the usual ``categories'' are a standard model for 1-categories (and not, say, simplicial sets satisfying the Segal condition).
Rather, the expectation seems to be that of a web of equivalences of models, each of them convenient in some aspects and lacking in others, which is ``self-reinforcing'' in that each link strengthens the idea that there is, after all, an essentially unique notion of \inftyn\nbd category.
A common vision is that, ultimately, we will all work in a model-independent language, invariant under equivalence of \inftyn\nbd categories, only using formal properties of what are otherwise concrete constructions; if one of these---say, slices, or enrichment, or the Grothendieck construction, or Gray products---is inaccessible in one model, it suffices to produce it in a different model, prove the necessary formal properties, and, finally, prove the equivalence between the two models.

While this may be a perfectly happy vision for homotopy theorists, or for the most formal-minded category theorists, it is less so for practitioners of diagrammatic algebra, computational algebra, representation theory, higher-dimensional rewriting, applied category theory, and other fields that rely on explicit diagrammatic manipulation, presentations by generators and relations, and rewriting methods which are intrinsically not equivalence-invariant.
With this community in mind, the second-named author has been expanding the combinatorial \cite{hadzihasanovic2024combinatorics} and computational \cite{hadzihasanovic2023data, hadzihasanovic2023higher} foundations of higher-categorical diagram rewriting, and then, jointly with the first-named author, developing models of \inftyn\nbd categories that support an expressive language for diagrammatic reasoning and rewriting, and which may reasonably aim to be ``self-contained'' \cite{chanavat2024diagrammatic, chanavat2024equivalences, chanavat2024model, chanavat2025gray}.

When models are 2\nbd dimensional---that is, they are bicategories or monoidal categories---a result that is fundamental to diagrammatic algebra is Mac Lane's strictification theorem \cite{maclane1963natural} that every bicategory can be embedded into a \emph{strict} 2\nbd category via an equivalence of bicategories: the replacement of a bicategory with a strict 2\nbd category allows one to dramatically reduce the higher-dimensional coherence data that needs to be accounted for in computations.
It has been suggested \cite{schelstraete2025rewriting} that the analogous \emph{semi-strictification} result one dimension up, that every tricategory embeds into a Gray-category \cite{gordon1995coherence}, could play an equally significant role in categorified diagrammatic algebra and representation theory.
As algebra turns ever higher, we can expect that semi-strictification results for \inftyn\nbd categories will be equally valuable, and from the start it has been our aim to achieve such a result.
This is what we present here, not exactly for the models we studied in our previous issues, but a close variant.
To explain what exactly we achieved, we first take some time to discuss what it means for a model to be weak, strict, or semi-strict.

\subsection*{Geometric, algebraic, weak, and semi-strict models}

\noindent 
We can identify a number of components that are expected of any model of \inftyn\nbd categories.
\begin{enumerate}
	\item A model of $k$\nbd dimensional \emph{cells} and their \emph{boundaries} for each $k \in \mathbb{N}$.
	\item A method for producing or recognising \emph{structural homotopies} between cells or diagrams of cells, which includes at least \emph{units}; we will use ``units'' as an umbrella term for such structural homotopies.
	\item The definition of a subclass of cells that are \emph{internal equivalences}, such that the structural homotopies are internal equivalences, and all cells in dimension $>n$ are internal equivalences.
	\item A method for producing or recognising an \emph{inverse} of an equivalence.
	\item A class of \emph{composable diagrams}, that is, arrangements of multiple cells that admit a composite, that is, reduction to a single cell.
	\item A method for producing or recognising a \emph{composite} of such a diagram.
	\item A definition of \emph{functor} between \inftyn\nbd categories, and a characterisation of what functors are \emph{equivalences}.
\end{enumerate}
Applying this scheme to categories, as a toy example of a model of $(\infty, 1)$\nbd categories, we have: there are only 0 and 1\nbd dimensional cells, modelled by points and arrows; the only structural homotopies are identity morphisms; the internal equivalences are isomorphisms, and the inverse of an isomorphism is its unique algebraic inverse; composable diagrams are concatenations of arrows, and their composite is specified algebraically; a functor is a functor and an equivalence is an essentially surjective, full and faithful functor.

Given such a model, there is a number of ``tests'' that it is expected to pass, in order to qualify as a candidate for a general model of \inftyn\nbd categories.
\begin{enumerate}
	\item The \emph{homotopy hypothesis}: $(\infty, 0)$\nbd categories, also known as $\infty$\nbd groupoids, must be a model of all classical homotopy types, in such a way that the Postnikov tower of a homotopy type can be read as the natural tower of $k$\nbd truncations identifying all $k$\nbd cells connected by a $(k+1)$\nbd cell.
	\item Categories should be faithfully represented in the model as 1-truncated $(\infty, 1)$\nbd categories, bicategories as 2-truncated $(\infty, 2)$\nbd categories.
	\item The model of $(\infty, 1)$\nbd categories should be equivalent to quasicategories or to complete Segal spaces or another standard model.
	\item More recently, as most models of $(\infty, 2)$\nbd categories are proved equivalent \cite{gagna2022equivalence}, equivalence with these models is also becoming a standard. 
\end{enumerate}
There is sometimes an expectation that strict $n$\nbd categories should be represented as a subclass of \inftyn\nbd categories, but for reasons that we will later recall, we think it is still up for debate that these should play an essential role in the theory.
There are also degrees of strength in what \emph{equivalence} with other models means, but at least with respect to the homotopy hypothesis, there is an undisputed standard which is to produce some version of a Quillen equivalence with the classical model structure on simplicial sets.

Based on the form that the different components take in a model, we can give a rough classification: say that units, inverses, and composites are
\begin{itemize}
	\item \emph{geometric} if their definition is given in terms of a \emph{recognition principle}, and their existence is merely stated as a \emph{property} of some underlying structure of an \inftyn\nbd category;
	\item \emph{algebraic} if they are produced by actual \emph{operations} which are part of the structure of an \inftyn\nbd category.
\end{itemize}
This gives a classification of models ranging from ``fully geometric'' to ``fully algebraic''.
It must be noted, however, that it is almost never expected that inverses should be provided algebraically, since there is no expectation that functors should, even in principle, preserve some assigned inverses; even 1\nbd groupoids are typically defined as categories with the \emph{property} that every morphism is an isomorphism.
Thus, an \emph{algebraic model} is, typically, one in which only units and composites are specified algebraically.
(We note that geometric models may sometimes be ``algebraicised'' by turning an existence property into a function producing existential witnesses---a typical move in constructive mathematics; this is, for example, the case of algebraic Kan complexes as a model of $\infty$\nbd groupoids \cite{nikolaus2011algebraic}.
We consider these models to be only algebraic in a formal sense; our classification is spiritual.)

In the current landscape of models, geometric and algebraic models form two disconnected clusters.
Geometric models include Segal-type models \cite{rezk2010cartesian, paoli2019simplicial} and ``shaped'' models \cite{verity2008weak, campion2025cubical}; several equivalences between these models have been proved \cite{barwick2020unicity, doherty2023equivalence, loubaton2023theory} at the level of their full homotopy theory.
Algebraic models include Grothendieck--Maltsiniotis models \cite{maltsiniotis2010grothendieck} as well as Batanin--Leinster models \cite{batanin1998weak, leinster2004higher} and the recent type-theoretic models \cite{benjamin2024globular}; between these, several equivalences have also been proved \cite{ara2010groupoides, bourke2020iterated, benjamin2024invertible}, although only at the level of equivalences of categories of strict functors.
So far, no equivalence is known between a geometric and an algebraic model for arbitrary $n$; even more significantly, \emph{no fully algebraic model} is known to satisfy the homotopy hypothesis beyond dimension 3, where algebraic tricategories suffice to present all 3\nbd types.

Within the subclass of algebraic or partially algebraic models, it makes sense to say that a model is \emph{strict} or \emph{semi-strict} if composition, in combination with units, satisfies some non-trivial equations \emph{strictly}, that is, up to \emph{equality} of cells and not up to internal equivalence; and that it is \emph{weak} otherwise.
For geometric models, this classification does not make sense, since there are no equations that could even in principle be satisfied; still, it is common to also refer to these models as weak.
Equations are typically sorted into three classes, matching the axioms of strict $n$\nbd categories.
\begin{itemize}
	\item \emph{Associativity}-type equations: the order of composition \emph{along the same direction} is irrelevant.
	\item \emph{Interchange}-type equations: the order of composition \emph{along different directions} is irrelevant. 
	\item \emph{Unitality}-type equations: composition with a unit is trivial.
\end{itemize}
Although there may be other models that satisfy some equations of each type, it is assumed that a fully strict model is one that is equivalent to strict $n$\nbd categories, either globular or cubical \cite{al2002multiple}.

\subsection*{Simpson's weak units conjecture}

\noindent
In \cite{kapranov1991infty}, Kapranov and Voevodsky claimed a proof of the homotopy hypothesis for strict $\omega$\nbd categories in which every cell is weakly invertible.
This would have made the entire point of weak models questionable, but the proof was incorrect, and the result false, as proved conclusively by Simpson \cite[Theorem 4.4.2]{simpson2009homotopy}; in fact, already strict 3\nbd categories are unable to model all homotopy 3\nbd types in the sense of the homotopy hypothesis.
Specific mistakes in Kapranov and Voevodsky's proof have since been pinpointed by Henry \cite{henry2019non} and the second-named author \cite{hadzihasanovic2020diagrammatic}.

Having established that no general model of the classical homotopy types---\emph{a fortiori}, no general model of \inftyn\nbd categories---can be fully strict, the question remains: to what degree can a general model be semi-strict?
At least in the first non-trivial case, of tricategories and homotopy 3-types, it was known, via the equivalence with Gray-categories, that only weakening the interchange equation, and keeping associativity and unitality strict, was enough to ensure full generality.
Simpson, on the other hand, gave an informal conceptual argument that only weakening unitality, while keeping associativity and interchange strict, should be enough.
This has come to be known as \emph{Simpson's weak units conjecture}.
Not being in the form of a precise mathematical statement, but rather a programmatic ``vision'', attempts at a proof of the weak units conjecture have focussed on a variety of models.

A precise formulation of an $n$\nbd categorical weak units conjecture was given by J.~Kock \cite{kock2006weak} as a conjectural equivalence between two Segal-type models, one of which satisfies certain Segal conditions up to isomorphism instead of weak equivalence.
A weak version of this conjecture was proved for $n = 3$ jointly by Joyal and Kock \cite{joyal2007weak}.
To the best of our knowledge, no evident further progress has been made on the $n$\nbd categorical version of the conjecture.

The situation is somewhat better for the case of $\infty$\nbd groupoids.
First, Paoli proved a form of semi-strictification for the subclass of homotopy $n$\nbd types whose classifying space is path-connected, although not quite in the sense of providing a semi-strict algebraic model, but rather of finding semi-strict representatives of objects in a weak model \cite{paoli2009weakly}.
The most significant progress, to our knowledge, is represented by Henry's proof of the homotopy hypothesis for a partially algebraic model of $\infty$\nbd groupoids, endowed with a semi-strict algebraic notion of composition, but non-algebraic units \cite{henry2018regular}.

Henry's approach relies on the algebra of strict $\omega$\nbd categories, not as the actual algebra of composition in the model, but rather, first, for the production of algebraic models of cells and their pasting diagrams; and, secondly, for the definition of a ``restricted'' composition operation where only a subclass of pasting diagrams is considered composable.
In the model that satisfies the homotopy hypothesis, this subclass is characterised by the \emph{spherical boundary} property---what we call \emph{roundness} following Steiner \cite{steiner1998pasting}---postulating, intuitively, that the composable $n$\nbd dimensional diagrams are those whose shape is, topologically, an $n$\nbd dimensional topological ball.
The fact that this ``round'' composition is particularly well-behaved had also been noticed, independently, by the second-named author \cite{hadzihasanovic2020combinatorial}, and can be traced to the fact that it is dual to ``polytopal'' subdivisions of cells, which are highly non-singular in that they do not change the homeomorphism type.
This should be contrasted with the \emph{globular} composition characteristic of most algebraic models, which typically does change the homeomorphism type, for example when producing a 2\nbd ball from a wedge of two 2\nbd balls:
\[\begin{tikzcd}
	\bullet & \bullet & \bullet & \quad & \bullet & \bullet\,.
	\arrow[""{name=0, anchor=center, inner sep=0}, curve={height=-12pt}, from=1-1, to=1-2]
	\arrow[""{name=1, anchor=center, inner sep=0}, curve={height=12pt}, from=1-1, to=1-2]
	\arrow[""{name=2, anchor=center, inner sep=0}, curve={height=-12pt}, from=1-2, to=1-3]
	\arrow[""{name=3, anchor=center, inner sep=0}, curve={height=12pt}, from=1-2, to=1-3]
	\arrow[""{name=4, anchor=center, inner sep=0}, curve={height=-12pt}, from=1-5, to=1-6]
	\arrow[""{name=5, anchor=center, inner sep=0}, curve={height=12pt}, from=1-5, to=1-6]
	\arrow["{\text{compose}}", dashed, shorten <=10pt, shorten >=5pt, maps to, from=1-3, to=1-5]
	\arrow[shorten <=3pt, shorten >=3pt, Rightarrow, from=1, to=0]
	\arrow[shorten <=3pt, shorten >=3pt, Rightarrow, from=3, to=2]
	\arrow[shorten <=3pt, shorten >=3pt, Rightarrow, from=5, to=4]
\end{tikzcd}\]
While Henry's proof is a remarkable step forward, the fact that, in the absence of algebraic units, there is no way to encode an algebraic globular composition limits the ability to compare this model with other algebraic models.

Beyond the weak units conjecture, we mention, briefly, that the development of ``quasistrict'' \cite{bar2017data} and ``associative'' \cite{dorn2018associative} $n$\nbd categories, which serve as a foundation for the \emph{homotopy.io} proof assistant \cite{corbyn2024homotopy}, has been sometimes presented as an approach to semi-strictification in the ``Gray-categorical'' tradition of weakening interchange.
Apart from the original definitions due to Dorn, however, it appears that all subsequent developments have only focussed on freely generated objects, in which composition coincides with pasting of diagrams---that is, the operation of forming diagrams of multiple cells, rather than the reduction of such diagrams to a single cell.
We note that both in Henry's model, and in the ones we will present here, the operation of pasting even satisfies the axioms of \emph{strict} $\omega$\nbd categories, which, of course, does not mean that the models are strict; composition is only \emph{at most} as strict as pasting.
Thus, at present, we do not find that this programme has contributed to this particular question.
Recently, type-theoretic semi-strict models have also been produced \cite{finster2022type}, with no proof of equivalence to other models.

In summary, to the best of our knowledge, the state of the art before our article is the following:
\begin{enumerate}
	\item there is no known model of $\infty$\nbd groupoids with both algebraic units and algebraic composition that satisfies the homotopy hypothesis,
	\item there is no geometric model of \inftyn\nbd categories that has been proved equivalent to an algebraic model for all $n$,
	\item there is no weak model of \inftyn\nbd categories that has been proved equivalent to a semi-strict model for all $n$.
\end{enumerate}
Our article presents a joint solution to all these open problems.

\subsection*{Background and main ideas of the proof}

\noindent
The idea for our proof was first conceived by the second-named author in 2017, from an analysis of the combinatorial structure of Mac Lane's strictification of associativity through the lens of Hermida's ``coherence via universality'' \cite{hermida2001coherent}.
While strictification of monoidal categories is typically presented as an equivalence between two algebraic models, it can fruitfully be seen as the strictification of a non-algebraic model, namely, a variant of Hermida's \emph{representable multicategories} \cite{hermida2000representable}.

Recall that a multicategory has ``multimorphisms'' with many sources and one target, that we can think of as the ``single 0-cell'' case of 2-cells with many 1-cells in its source and a single 1-cell in its target:
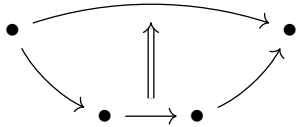
\[\begin{tikzcd}[column sep=scriptsize]
	\bullet &&& \bullet \\
	& \bullet & \bullet
	\arrow[""{name=0, anchor=center, inner sep=0}, curve={height=-12pt}, from=1-1, to=1-4]
	\arrow[curve={height=6pt}, from=1-1, to=2-2]
	\arrow[""{name=1, anchor=center, inner sep=0}, from=2-2, to=2-3]
	\arrow[curve={height=6pt}, from=2-3, to=1-4]
	\arrow[shorten <=6pt, shorten >=6pt, Rightarrow, from=1, to=0]
\end{tikzcd}\]
In a representable multicategory, the tensor product $a \otimes b$ is not given by an operation; instead, it is witnessed by a \emph{universal} multimorphism of type
\begin{equation} \label{eq:universal_cell}
\begin{tikzcd}[column sep=scriptsize]
	\bullet && \bullet \\
	& \bullet
	\arrow[""{name=0, anchor=center, inner sep=0}, "{a \otimes b}", curve={height=-12pt}, from=1-1, to=1-3]
	\arrow["a"', curve={height=6pt}, from=1-1, to=2-2]
	\arrow["b"', curve={height=6pt}, from=2-2, to=1-3]
	\arrow[shorten >=5pt, Rightarrow, from=2-2, to=0]
\end{tikzcd}
\end{equation}
determining a non-algebraic notion of composition, such that the structure of a coherent weakly associative tensor product in a monoidal category is equivalent to the \emph{property} that a multicategory possess enough universal multimorphisms.
While in a multicategory, due to the asymmetry of multiple sources and one target, universality must be stated as a ``unique factorisation'' property, one can formulate variants of this equivalence for more symmetric structures such as coloured pros or planar properads, such that universality can be replaced with \emph{invertibility} as soon as one has a sufficiently expressive algebra of units.

In \cite{hadzihasanovic2019weak}, the second-named author observed that if, in such a structure, we have separately an algebra of units sufficiently expressive to generate units not just on cells, but also more general pasting diagrams
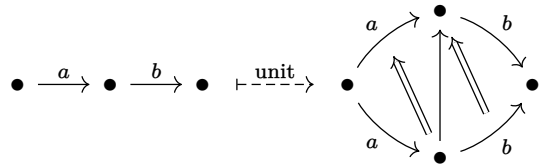
\[\begin{tikzcd}[sep=scriptsize]
	&&&&& \bullet \\
	\bullet & \bullet & \bullet && \bullet && \bullet \\
	&&&&& \bullet
	\arrow["b", curve={height=-6pt}, from=1-6, to=2-7]
	\arrow["a", from=2-1, to=2-2]
	\arrow["b", from=2-2, to=2-3]
	\arrow["{\text{unit}}", shorten <=5pt, shorten >=5pt, dashed, maps to, from=2-3, to=2-5]
	\arrow[""{name=0, anchor=center, inner sep=0}, "a", curve={height=-6pt}, from=2-5, to=1-6]
	\arrow["a"', curve={height=6pt}, from=2-5, to=3-6]
	\arrow[from=3-6, to=1-6]
	\arrow[""{name=1, anchor=center, inner sep=0}, "b"', curve={height=6pt}, from=3-6, to=2-7]
	\arrow[shorten <=3pt, shorten >=7pt, Rightarrow, from=3-6, to=0]
	\arrow[shorten <=7pt, shorten >=3pt, Rightarrow, from=1, to=1-6]
\end{tikzcd}\]
as well as an algebra of round composition
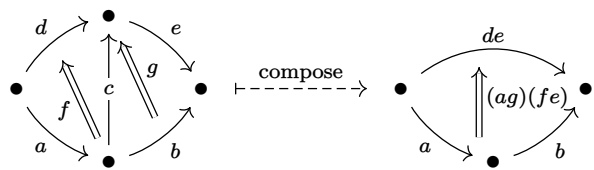
\[\begin{tikzcd}[sep=scriptsize]
	& \bullet \\
	\bullet && \bullet & \quad & \bullet && \bullet \\
	& \bullet &&&& \bullet
	\arrow["e", curve={height=-6pt}, from=1-2, to=2-3]
	\arrow[""{name=0, anchor=center, inner sep=0}, "d", curve={height=-6pt}, from=2-1, to=1-2]
	\arrow["a"', curve={height=6pt}, from=2-1, to=3-2]
	\arrow["{\text{compose}}", shorten <=5pt, shorten >=5pt, dashed, maps to, from=2-3, to=2-5]
	\arrow[""{name=1, anchor=center, inner sep=0}, "de", curve={height=-18pt}, from=2-5, to=2-7]
	\arrow["a"', curve={height=6pt}, from=2-5, to=3-6]
	\arrow["c"{description}, from=3-2, to=1-2]
	\arrow[""{name=2, anchor=center, inner sep=0}, "b"', curve={height=6pt}, from=3-2, to=2-3]
	\arrow["b"', curve={height=6pt}, from=3-6, to=2-7]
	\arrow["f", shorten <=3pt, shorten >=7pt, Rightarrow, from=3-2, to=0]
	\arrow["g"', shorten <=7pt, shorten >=3pt, Rightarrow, from=2, to=1-2]
	\arrow["{(ag)(fe)}"', shift left=2, shorten <=3pt, shorten >=6pt, Rightarrow, from=3-6, to=1]
\end{tikzcd}\]
then such a structure admits canonical witnesses of composition as in (\ref{eq:universal_cell}), obtained by taking a unit on a pasting diagram, then partially composing it:
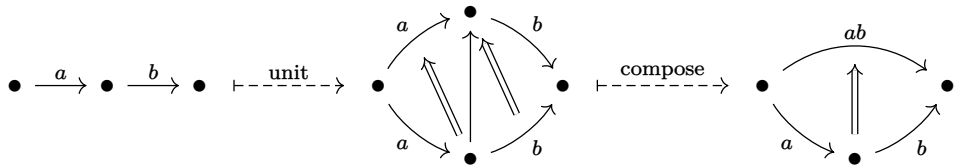
\[\begin{tikzcd}[sep=scriptsize]
	&&&&& \bullet \\
	\bullet & \bullet & \bullet & \; & \bullet && \bullet & \quad & \bullet && \bullet \\
	&&&&& \bullet &&&& \bullet
	\arrow["b", curve={height=-6pt}, from=1-6, to=2-7]
	\arrow["a", from=2-1, to=2-2]
	\arrow["b", from=2-2, to=2-3]
	\arrow["{\text{unit}}", shorten <=5pt, shorten >=5pt, dashed, maps to, from=2-3, to=2-5]
	\arrow[""{name=0, anchor=center, inner sep=0}, "a", curve={height=-6pt}, from=2-5, to=1-6]
	\arrow["a"', curve={height=6pt}, from=2-5, to=3-6]
	\arrow["{\text{compose}}", shorten <=5pt, shorten >=5pt, dashed, maps to, from=2-7, to=2-9]
	\arrow[""{name=1, anchor=center, inner sep=0}, "ab", curve={height=-18pt}, from=2-9, to=2-11]
	\arrow["a"', curve={height=6pt}, from=2-9, to=3-10]
	\arrow[from=3-6, to=1-6]
	\arrow[""{name=2, anchor=center, inner sep=0}, "b"', curve={height=6pt}, from=3-6, to=2-7]
	\arrow["b"', curve={height=6pt}, from=3-10, to=2-11]
	\arrow[shorten <=3pt, shorten >=7pt, Rightarrow, from=3-6, to=0]
	\arrow[shorten <=7pt, shorten >=3pt, Rightarrow, from=2, to=1-6]
	\arrow[shorten <=3pt, shorten >=6pt, Rightarrow, from=3-10, to=1]
\end{tikzcd}\]
Moreover, the algebra of units, by itself, suffices to instantiate the notion of universality-as-invertibility, hence define a non-algebraic model of monoidal categories.
Then, \emph{freely} adding round composites to any instance of this model produces an instance of an \emph{algebraic} model, which, furthermore, satisfies strict associativity; the embedding of the original model into this free composition-algebra is an equivalence, realising its semi-strictification.

The fact that freely adding composites still produces an instance of the weak model---in particular, still admits an algebra of units---relies on a distributive law between the monad governing the algebra of units (called the \emph{inflate} monad) and the monad governing round composition (called the \emph{merge} monad), stating, roughly, that \emph{a unit on a composite is a composite of units}:
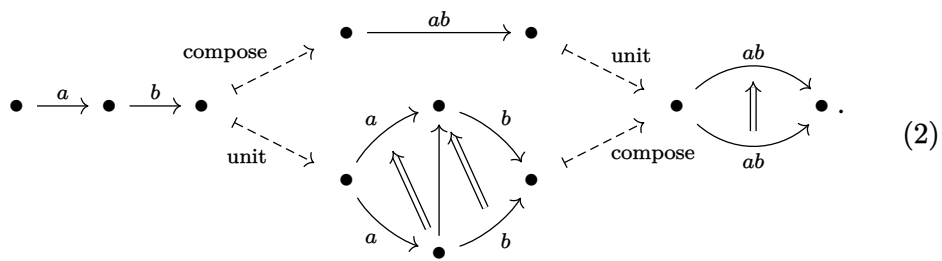
\begin{equation} \label{eq:merge_inflate}
	\begin{tikzcd}[sep=scriptsize]
	&&&& \bullet && \bullet \\
	\bullet & \bullet & \bullet &&& \bullet &&& \bullet && \bullet\;. \\
	&&&& \bullet && \bullet \\
	&&&&& \bullet
	\arrow["ab", from=1-5, to=1-7]
	\arrow["{\text{unit}}", shorten <=5pt, shorten >=5pt, dashed, maps to, from=1-7, to=2-9]
	\arrow["a", from=2-1, to=2-2]
	\arrow["b", from=2-2, to=2-3]
	\arrow["{\text{compose}}", shorten <=5pt, shorten >=5pt, dashed, maps to, from=2-3, to=1-5]
	\arrow["{\text{unit}}"', shorten <=5pt, shorten >=5pt, dashed, maps to, from=2-3, to=3-5]
	\arrow["b", curve={height=-6pt}, from=2-6, to=3-7]
	\arrow[""{name=0, anchor=center, inner sep=0}, "ab", curve={height=-18pt}, from=2-9, to=2-11]
	\arrow[""{name=1, anchor=center, inner sep=0}, "ab"', curve={height=18pt}, from=2-9, to=2-11]
	\arrow[""{name=2, anchor=center, inner sep=0}, "a", curve={height=-6pt}, from=3-5, to=2-6]
	\arrow["a"', curve={height=6pt}, from=3-5, to=4-6]
	\arrow["{\text{compose}}"', shorten <=5pt, shorten >=5pt, dashed, maps to, from=3-7, to=2-9]
	\arrow[from=4-6, to=2-6]
	\arrow[""{name=3, anchor=center, inner sep=0}, "b"', curve={height=6pt}, from=4-6, to=3-7]
	\arrow[shorten <=5pt, shorten >=5pt, Rightarrow, from=1, to=0]
	\arrow[shorten <=3pt, shorten >=7pt, Rightarrow, from=4-6, to=2]
	\arrow[shorten <=7pt, shorten >=3pt, Rightarrow, from=3, to=2-6]
\end{tikzcd}
\end{equation}
Even though such a model does not have a primitive algebraic operation for globular composition, one can still \emph{derive} one from the combination of units and round composition:
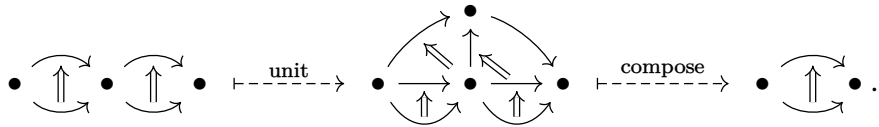
\[\begin{tikzcd}[sep=scriptsize]
	&&&&& \bullet \\
	\bullet & \bullet & \bullet & \; & \bullet & \bullet & \bullet & \quad & \bullet & {\bullet\;.}
	\arrow[curve={height=-6pt}, from=1-6, to=2-7]
	\arrow[""{name=0, anchor=center, inner sep=0}, curve={height=-12pt}, from=2-1, to=2-2]
	\arrow[""{name=1, anchor=center, inner sep=0}, curve={height=12pt}, from=2-1, to=2-2]
	\arrow[""{name=2, anchor=center, inner sep=0}, curve={height=-12pt}, from=2-2, to=2-3]
	\arrow[""{name=3, anchor=center, inner sep=0}, curve={height=12pt}, from=2-2, to=2-3]
	\arrow["{\text{unit}}", shorten <=5pt, shorten >=5pt, dashed, maps to, from=2-3, to=2-5]
	\arrow[""{name=4, anchor=center, inner sep=0}, curve={height=-6pt}, from=2-5, to=1-6]
	\arrow[""{name=5, anchor=center, inner sep=0}, curve={height=18pt}, from=2-5, to=2-6]
	\arrow[""{name=6, anchor=center, inner sep=0}, from=2-5, to=2-6]
	\arrow[""{name=7, anchor=center, inner sep=0}, from=2-6, to=1-6]
	\arrow[""{name=8, anchor=center, inner sep=0}, curve={height=18pt}, from=2-6, to=2-7]
	\arrow[""{name=9, anchor=center, inner sep=0}, from=2-6, to=2-7]
	\arrow["{\text{compose}}", shorten <=5pt, shorten >=5pt, dashed, maps to, from=2-7, to=2-9]
	\arrow[""{name=10, anchor=center, inner sep=0}, curve={height=12pt}, from=2-9, to=2-10]
	\arrow[""{name=11, anchor=center, inner sep=0}, curve={height=-12pt}, from=2-9, to=2-10]
	\arrow[shorten <=3pt, shorten >=3pt, Rightarrow, from=1, to=0]
	\arrow[shorten <=3pt, shorten >=3pt, Rightarrow, from=3, to=2]
	\arrow[shorten <=2pt, shorten >=2pt, Rightarrow, from=5, to=6]
	\arrow[shorten <=3pt, shorten >=3pt, Rightarrow, from=9, to=7]
	\arrow[shorten >=3pt, Rightarrow, from=2-6, to=4]
	\arrow[shorten <=2pt, shorten >=2pt, Rightarrow, from=8, to=9]
	\arrow[shorten <=3pt, shorten >=3pt, Rightarrow, from=10, to=11]
\end{tikzcd}\]
At the end of \cite{hadzihasanovic2019weak}, it was conjectured that the general structure of this semi-strictification proof could, in principle, be replicated in arbitrary dimensions, as long as one could find suitable higher-dimensional analogues of the inflate and merge monads as well as their distributive law.

As a matter of fact, the proof that we present is remarkably similar in its structure to its 2\nbd dimensional proof-of-concept; but while in dimension 2 it was easy to give explicit definitions of the inflate and merge monads, finding a form that would generalise to arbitrarily high dimension has required a development of the theory of higher-categorical pasting diagrams well beyond the previous state of the art.
This development, revisiting and expanding Steiner's groundbreaking work \cite{steiner1993algebra}, has resulted in the monograph \cite{hadzihasanovic2024combinatorics}. 

Our models are based on the combinatorial theory of \emph{regular directed complexes}, order-theoretic objects ideally joining regular CW complexes and strict $\omega$\nbd categories, modelling the shapes of cells and pasting diagrams.
The idea is that different notions of morphism between regular directed complexes should govern the different components of a model: embeddings for faces and boundaries, ``collapsing'' morphisms for units and degeneracies, and subdivisions for composition.
The problem of generalising the merge-inflate distributivity then becomes the problem of finding a class of morphisms that admits a \emph{ternary factorisation system} whose classes are, respectively, subdivisions, collapses, and local embeddings, such that the restriction to collapses and local embeddings is expressive enough to support a weak model of \inftyn\nbd categories.

As shown in \cite[Chapter 6]{hadzihasanovic2024combinatorics}, there are two natural notions of morphism of regular directed complexes, characterised among the order-preserving maps of their underlying posets by the property that they determine functors of strict $\omega$\nbd categories \emph{covariantly} and \emph{contravariantly}.
These are called \emph{maps} and \emph{comaps}, respectively.
It so happens that duals of comaps are remarkably well-adapted to modelling round composition, while maps admit a (collapse, local embedding) orthogonal factorisation system, making it natural to look for a ternary factorisation on a subclass of (comap, map)-spans.

Unfortunately, in general, there are way too many collapsing maps to hope for distributivity against comaps, or even just to have a well-behaved model of weak \inftyn\nbd categories where degeneracies are dual to arbitrary collapsing maps.
The latter problem is solved by restricting to the class of \emph{cartesian maps}, whose underlying order-preserving map is a posetal Grothendieck fibration.
This class is very well-behaved, stable under all sorts of constructions, and produces enough degeneracies to admit a calculus of \emph{coinductively weakly invertible} round diagrams \cite{cheng2007omega}, which we used to develop a model of \inftyn\nbd categories with some remarkable properties, called the diagrammatic model \cite{chanavat2024equivalences, chanavat2024model}.
For finite $n$, this model is naturally a localisation of a model of \inftyinf\nbd categories in the \emph{coinductive} sense, where every coinductively weakly invertible cell is an internal equivalence \cite{loubaton2024inductive}.

The pullback of a comap against a cartesian map satisfies some characteristic properties of regular directed complexes, and our initial hope was that we could realise a merge-inflate distributivity with arbitrary comaps and cartesian maps, semi-strictifying the diagrammatic model.
Unfortunately, this turned out to be too optimistic, as we found some (cartesian map, comap)-pairs that were incompatible with distributivity.

There was a subclass of cartesian maps that we knew to be compatible with distributivity: these are the ones that are generated by \emph{cylindrical collapses}, that is, projections of left Gray-cylinders which are ``partially collapsed'' on the boundary.
Furthermore, the degeneracies generated by these collapses are sufficient to support our calculus of natural equivalences from \cite{chanavat2024equivalences}.
However, there is not enough of them to support a good theory of coinductive weak invertibility; in particular, not all degenerate cells are automatically coinductively weakly invertible.
So, instead, we pivot to defining an \emph{inductive} model of \inftyinf\nbd categories, in the standard fashion of specifying a subclass of \emph{marked cells}, including all the degenerate cells, which are required to be invertible up to marked cells.
(Incidentally, using marked structures is also the approach suggested by Henry \cite[\textsection 6.4.3]{henry2018regular} for extending his proof from $\infty$\nbd groupoids to \inftyn\nbd categories.)

One last problem is that the class of cylindrical collapses is not stable under many constructions, which makes it inconvenient as a ``categorical'' structure that should be strictly respected by morphisms. 
Instead, we work with categories of non-unital structures, which we call \emph{marked directed complexes} (without round composition) and \emph{marked merge-complexes} (with round composition); the inflate-algebra structure is treated as an extra structure on nice objects that does not need to be preserved by morphisms.
For this reason, our semi-strictification of functors produces functors that strictly preserve round composition, but only weakly preserve units.
The classical theory of Quillen model structures is not suited to defining a homotopy theory on such non-unital structures, but its elegant generalisation also due to Henry \cite{henry2020weak}, the theory of \emph{weak model categories}, is perfectly equipped for the task.

With this setup, an amalgamation of the second-named author's proof for bicategories with Henry's proof of $\infty$\nbd groupoids goes through, resulting in a semi-strictification proof for \inftyn\nbd categories.

\subsection*{Structure of the article}

\noindent
Part \ref{part:combinatorial} is dedicated to the underlying combinatorics of our models.
After some recollections on the theory of regular directed complexes (Section \ref{sec:recollections}), in Section \ref{sec:subdivisions} we define the notions of \emph{subdivision} and of \emph{cylindrical collapse} that will determine our algebra of round composition and units.
We prove (Proposition \ref{prop:freeness_of_collapses}) that the algebra of cylindrical collapses is particularly simple, in that it is \emph{freely generated} by codimension-1 collapses.
In Section \ref{sec:composing}, we construct a factorisation of subdivisions against local collapses which determines our ternary factorisation system (Proposition \ref{prop:ternary_factorisation_systems}), and show that Gray products determine monoidal structures on the wide subcategories on local embeddings and local subdivisions (Proposition \ref{prop:gray_product_of_regular_functors}).
Finally, in Section \ref{sec:colimits}, we study some notable colimits in categories of finite regular directed complexes, in preparation for the definition of our models in categories of presheaves that are continuous with respect to these colimits.

Part \ref{part:homotopy} sets up the homotopy theory of \inftyn\nbd categories relative to which we will prove our result.
Section \ref{sec:weakmodel} is a summary of the relevant parts of Henry's theory of weak model categories.
In Section \ref{sec:directed}, we give a couple of equivalent presentations of the category of directed complexes, and import the terminology and some results relative to cells, pasting diagrams, round diagrams, and contexts that we developed for diagrammatic sets.
We also provide a technical lemma (Lemma \ref{lem:day_for_gamma_continuous_presheaves}) on the extension of monoidal structures to categories of continuous presheaves, that we use to define the Gray product of directed complexes (Proposition \ref{prop:gray_product_of_directed_complexes}).
In Section \ref{sec:marked}, we define the category of marked directed complexes, as well as their Gray product.
Then, in Section \ref{sec:weakonmarked}, we put (Theorem \ref{thm:weak_model_structures}) a family of weak model structures $\Mwn$ on this category, indexed by $n \in \Ninfty$.
We do this \emph{\`a la} Cisinski--Olschok, defining a functorial cylinder (by left Gray product with the ``marked arrow'') and sets of generating cofibrations and anodyne extensions.

Part \ref{part:weak} is dedicated to our weak model of \inftyn\nbd categories.
In Section \ref{sec:inflate}, we define \emph{inflate-complexes} either as presheaves on the category of atoms and local collapses, or as algebras for an inflate monad, as well as their marked version.
We define notions of \emph{marked-equivalence} of round diagrams (when there exist marked round diagrams pointing from one to the other, and vice versa) as well as \emph{marked-invertibility} (when a round diagram is invertible up to marked-equivalence).
We then define an \inftyinf\nbd category to be a marked inflate-complex satisfying two axioms: every round diagram is marked-equivalent to a single cell, its \emph{weak composite}; and a cell is marked-invertible if and only if it is marked (a form of Rezk-completeness).
An \inftyinf\nbd category is an \inftyn\nbd category if, furthermore, all cells in dimension $> n$ are marked.
Section \ref{sec:equivalences} is dedicated to the proof of Theorem \ref{thm:characterisation_of_fibrants}, that fibrant objects in $\Mwn$ are precisely the underlying marked directed complexes of \inftyn\nbd categories; in particular, every fibrant admits an inflate-algebra structure (Theorem 
\ref{thm:fibrants_admit_inflate_algebra_structure}).
In Section \ref{sec:equivalencesof}, we define a functor of \inftyn\nbd categories as a morphism of the underlying marked directed complexes, and characterise the equivalences of fibrant objects in $\Mwn$ as being precisely the functors of \inftyn\nbd categories that are surjective on cells of every dimension up to marked-equivalence (Theorem \ref{thm:characterisation_of_equivalences}).
Finally, in Section \ref{sec:homohyp}, we prove the homotopy hypothesis for our weak model: $\Mw_0$ is Quillen-equivalent to a weak model structure on semi-simplicial sets presenting the classical homotopy types (Theorem \ref{thm:homotopy_hypothesis}).

Part \ref{part:semistrict} is semi-strictification. 
In Section \ref{sec:merge}, we define merge-complexes, describe the free merge-complex on a directed complex, define the Gray product of merge-complexes, and prove that the ``free merge-complex'' functor is strong monoidal (Theorem \ref{thm:gray_product_of_merge-complexes}).
Then, we do it all over again for marked merge-complexes.
In Section \ref{sec:weakonmerge}, we put a family of weak model structures $\MMwn$ on marked merge-complexes and show that the free-forgetful adjunction between marked directed complexes and marked merge-complexes determines a Quillen adjunction between $\Mwn$ and $\MMwn$ (Theorem 
\ref{thm:model_structures_on_marked_merge}).
In Section \ref{sec:mergeinflate}, we define \emph{merge-inflate-complexes}, which have compatible structures of merge-complex and inflate-complex in the sense of merge-inflate distributivity (\ref{eq:merge_inflate}).
We then define a \emph{merge-$n$\nbd category} to be a complete marked merge-inflate-complex, in the sense that its marked cells coincide with marked-invertible cells, and a \emph{semi-strict functor} of merge-$n$\nbd categories to be a morphism of the underlying marked merge-complexes.
Finally, in Section \ref{sec:simpson} we prove our main theorem (Theorem \ref{thm:main_thm}), whose full statement we reproduce here.

\begin{thm*}
	Let $f\colon (X, A) \to (Y, B)$ be a functor of \inftyn\nbd categories.
	Then there exists a square
	\[\begin{tikzcd}[column sep=large]
	{(X, A)} & {(Y, B)} \\
	{\UmMerg(\FMerg X, \satur{\clcom{A}})} & {\UmMerg(\FMerg Y, \satur{\clcom{B}}).}
	\arrow["f", from=1-1, to=1-2]
	\arrow[hook, "{\sigma_{(X, A)}}", from=1-1, to=2-1]
	\arrow[hook, "{\sigma_{(Y, B)}}", from=1-2, to=2-2]
	\arrow["{\UMerg\FMerg f}", from=2-1, to=2-2]
\end{tikzcd}\]
	where
	\begin{enumerate}
		\item $\FMerg f\colon (\FMerg X, \satur{\clcom{A}}) \to (\FMerg Y, \satur{\clcom{B}})$ is a semi-strict functor of merge-$n$\nbd categories,
		\item $\sigma_{(X, A)}$ and $\sigma_{(Y, B)}$ are equivalences of \inftyn\nbd categories, in particular acyclic cofibrations.
	\end{enumerate}
	Moreover, the free-forgetful adjunction between marked directed complexes and marked merge-complexes determines a Quillen equivalence between $\Mwn$ and $\MMwn$ such that
	\begin{enumerate}
		\item the category of \inftyn\nbd categories and functors is equivalent to the category of fibrant objects in $\Mwn$,
		\item every morphism of fibrant objects in $\MMwn$ is a semi-strict functor of merge-$n$\nbd categories up to acyclic fibrations over its domain and codomain.
	\end{enumerate}
	In particular, a semi-strict functor of merge-$n$\nbd categories is an equivalence if and only if its underlying functor of \inftyn\nbd categories is an equivalence.
\end{thm*}

\subsection*{Open questions and directions}

\noindent
Our hope is to have breached a wall between non-algebraic and algebraic models, and between weak and semi-strict models, and that this breach can now be exploited to further connect the web of models of \inftyn\nbd categories. 
First of all, having proved the homotopy hypothesis, we know that our models are equivalent to other geometric models of $(\infty, 0)$\nbd categories, but the case $n > 0$ still needs to be proved.
In Section \ref{sec:homohyp}, we set up a comparison with the complicial model, but it seems that the strategy used to prove the equivalence for $n = 0$, based on a left Quillen simplicial subdivision functor, may still work at most for $n = 1$; jointly with Loubaton, we have found combinatorial obstructions to the existence of a ``marked simplicial subdivision'' already for $n = 2$.
A more promising strategy may be to work inductively along the lines of \cite{bergner2013comparison}, proving an equivalence between $\Mw_{n+1}$ and Segal precategories weakly enriched in $\Mwn$, since such a proof would involve a comparison between ``$(n+1)$\nbd dimensional pasting diagrams'' and ``$1$\nbd dimensional pasting diagrams enriched in $n$\nbd dimensional pasting diagrams'', for which our explicit combinatorial models may be advantageous.

Overall, we find it likely that our models are equivalent to the other geometric models, and we are entirely confident of it up to $n = 3$; it should be noted that, until now, every inequivalence of models has already been visible at $n = 0$.
If there remains some doubt about $n > 3$, it is tied to the ``mismatch'' between regular directed complexes and polygraphs for strict $n$\nbd categories that appears at $n = 4$, as discussed in \cite[Section 6.2]{chanavat2024model}: in this dimension, the combinatorial pasting of diagrams satisfies some extra relations that are not provable in the algebra of strict 4\nbd categories.
These relations are topologically sound, satisfied by the pasting of topological cells, and they appear to be forms of ``interchange in higher codimension'' that are not reducible to the codimension-1 case.
For these reasons, we are inclined to consider them as evidence of a form of incompleteness of the algebra of strict 4\nbd categories, and to consider our combinatorial pasting as a better model of the concept that they are both trying to capture.
We cannot exclude, for the moment, that this mismatch may carry over from strict $n$\nbd categories to Segal-type weak models that are also based on ``iterated enrichment''.
On the other hand, we equally cannot see how it \emph{could} carry over to the complicial and cubical models, where it seems intuitively true that pasting should satisfy all relations that are topologically sound; so if the proof of equivalence between complicial and Segal-type models holds, perhaps there is no substantial problem.
In any case, if an inequivalence was found at $n = 4$ between models of \inftyn\nbd categories that are equivalent up to $n = 3$, we would mainly consider it as evidence that claims of ``unicity'' were premature; whereas a proof of equivalence would be further strong evidence for unicity, and at the same time evidence of the incompleteness of strict $n$\nbd categories for pasting diagrams. 

In Section \ref{sec:mergeinflate}, we also show how, in a merge-$n$\nbd category, one can define algebraic globular composition operations, which can be seen as a first step towards a comparison with Batanin--Leinster models based on monads on $\omega$\nbd graphs, or other equivalent models in the algebraic cluster.
Of course, our questions about the mismatch with strict 4\nbd categories seem even more relevant to the comparison with these models, which are directly based on a weakening of the algebra of strict $n$\nbd categories, so \emph{a fortiori} cannot satisfy the extra relations for pasting.
On the other hand, despite some recent progress \cite{fujii2024omega}, the homotopy theory of these algebraic models is still underdeveloped, so there is more leeway for a change of course, should one be necessary.

Although we have focussed here on the potential role of these models in connecting other models, we must underscore that, ultimately, we also believe in their inherent value as ``convenient models''.
While in this article, we only define the Gray product of marked directed complexes and of marked merge-complexes, using the results of \cite[Chapter 7]{hadzihasanovic2024combinatorics}, it is easy to also produce definitions of the join, suspension, and direction-reversing duals, as well as ``pseudo'' version of the Gray product and the join, all of which are fully explicit and combinatorial.
Furthermore, the intrinsic diagrammatic language and ``pasting theorem'' satisfied by our models are well beyond those currently supported by any other models of \inftyn\nbd categories.
We believe that this makes them very suitable as a framework for higher and homotopical diagrammatic algebra.
Of course, much remains to be done in developing the proper ``higher category theory'' of these models, starting from the Grothendieck construction and notions of higher limits and colimits, that we leave to future work.

\subsection*{Acknowledgements}

The second-named author was supported by Estonian Research Council grant PSG764, and is grateful to Nathanael Arkor, Simon Henry, Fosco Loregian, F\'elix Loubaton, and Viktoriya Ozornova for helpful conversations on various points of the article.

\section{Combinatorial background} \label{part:combinatorial}

\subsection{Recollections on regular directed complexes} \label{sec:recollections}

\noindent
What follows is a brief recollection of fundamental notions from \cite{hadzihasanovic2024combinatorics}; we refer to the book for the full details.
We consider each poset $P$ to be equipped with the Alexandrov topology whose closed sets are the lower sets (that is, the downwards closed sets).
We write $\clos$ for the closure operator of this topology.
Given $x \in P$, we let $\faces{}{}x \eqdef \set{y \in P \mid \text{$x$ covers $y$}}$, the set of \emph{faces of $x$}, and, dually, $\cofaces{}{}x \eqdef \set{y \in P \mid x \in \faces{}{}y}$, the set of \emph{cofaces of $x$}.
We let $\bd{}{}x \eqdef \clset{\faces{}{}x}$.

Recall that $P$ is \emph{graded} if, for all $x \in P$, all maximal chains in $\clset{x}$ have the same finite size.
In this case, we write $\dim x$ for the size of a maximal chain in $\clset{x}$ and call it the \emph{dimension} of $x$.
For each $n \in \mathbb{N}$ and $A \subseteq P$, we let $\gr{n}{A} \eqdef \set{x \in A \mid \dim x = n}$, as well as $\gr{>n}{A} \eqdef \bigcup_{k > n} \gr{k}{A}$ and, similarly, $\gr{<n}{A} \eqdef \bigcup_{k < n} \gr{k}{A}$.
We let the \emph{dimension} $\dim P$ of $P$ be equal to $\max \set{\dim x \mid x \in P}$ if it exists, and $\infty$ otherwise.

\begin{dfn}[Oriented graded poset]
	Let $P$ be a graded poset.
	An \emph{orientation} on $P$ is, for each $x \in P$, a bipartition of the set $\faces{}{}x$ into a set $\faces{}{-}x$ of \emph{input faces} and a set $\faces{}{+}x$ of \emph{output faces}.
	An \emph{oriented graded poset} is a graded poset equipped with an orientation.
	A \emph{morphism} $f\colon P \to Q$ of oriented graded posets is a function of their underlying sets that respects the grading and, for each $x \in P$ and $\a \in \set{-, +}$, induces a bijection between $\faces{}{\a}x$ and $\faces{}{\a}f(x)$.
\end{dfn}

\noindent
We let $\ogPos$ denote the category of oriented graded posets and morphisms.
A morphism of oriented graded posets has an underlying closed order-preserving map of posets, which determines a forgetful functor $\ogPos \to \Pos$.

We call an injective morphism an \emph{embedding}; its underlying map is a closed embedding of posets. The category $\ogPos$ has pushouts of embeddings along embeddings, and embeddings are stable under these; furthermore, a pushout square of embeddings is also a pullback square.
All these pushouts are computed in $\Pos$, preserved and reflected by the forgetful functor.

Every closed subset $U \subseteq P$ inherits a unique orientation such that its inclusion is an embedding of oriented graded posets.
For each $n \in \mathbb{N}$ and $\a \in \set{-, +}$, we let $\faces{n}{\a}U \eqdef \set{x \in \gr{n}{U} \mid \cofaces{}{-\a}x \cap U = \varnothing}$.
Then $\faces{n}{-}U \cap \faces{n}{+}U$ is equal to the set $\gr{n}{(\maxel{U})}$ of maximal $n$\nbd dimensional elements of $U$.

\begin{dfn}[Input and output boundaries]
	Let $U$ be a closed subset of an oriented graded poset.
	For each $n \in \mathbb{N}$ and $\a \in \set{-, +}$, we let
	\[
		\bd{n}{\a}U \eqdef \clos (\faces{n}{\a}U) \cup \bigcup_{k < n} \clos \gr{k}{(\maxel P)}.
	\]
	We call $\bd{n}{-}U$ the \emph{input $n$\nbd boundary} and $\bd{n}{+}U$ the \emph{output $n$\nbd boundary} of $U$.
	We let $\bd{n}{}U \eqdef \bd{n}{-}U \cup \bd{n}{+}U$ be the \emph{$n$\nbd boundary of $U$}, and finally $\bd{}{}U \eqdef \bigcup_{k < \dim U} \bd{n}{}U$ be the \emph{boundary of $U$}.
	We let $\inter{U} \eqdef U \setminus \bd{}{}U$ be the \emph{interior of $U$}.
\end{dfn}

\noindent
We adopt the convention of omitting $n$ in $\bd{n}{\a}U$ when $n = \dim U - 1$, and of writing $\bd{n}{\a}x$ for $\bd{n}{\a}\clset{x}$. 

\begin{dfn}[Globular and round subsets]
	Let $U$ be a closed subset of an oriented graded poset.
	We say that $U$ is \emph{globular} if, for all $n \in \mathbb{N}$, $k < n$, and $\a, \b \in \set{-, +}$, it satisfies $\bd{k}{\a}\bd{n}{\b}U = \bd{k}{\a}U$.
	We say that $U$ is \emph{round} if it is globular and, in addition, $\bd{n}{-}U \cap \bd{n}{+}U = \bd{n-1}{}U$ for all $n < \dim U$.
\end{dfn}

\noindent
Let $1$ denote the oriented graded poset with a single element.

\begin{dfn}[Molecule]
	The class of \emph{molecules} is the subclass of oriented graded posets closed under isomorphisms and inductively generated by the following clauses.
	\begin{enumerate}
		\item (\emph{Point}).
			The \emph{point} $1$ is a molecule.
		\item (\emph{Paste}).
			If $U$, $V$ are molecules, $k \in \mathbb{N}$, and $\varphi\colon \bd{k}{+}U \iso \bd{k}{-}V$ is an isomorphism, then the \emph{pasting} $U \cp{k} V$ obtained as the pushout
\[\begin{tikzcd}
	{\bd{k}{+}U} & {\bd{k}{-}V} & V \\
	U && {U \cp{k} V}
	\arrow["\varphi", from=1-1, to=1-2]
	\arrow[hook, from=1-1, to=2-1]
	\arrow[hook, from=1-2, to=1-3]
	\arrow[hook, from=1-3, to=2-3]
	\arrow[hook, from=2-1, to=2-3]
	\arrow["\lrcorner"{anchor=center, pos=0.125, rotate=180}, draw=none, from=2-3, to=1-1]
\end{tikzcd}\]
			is a molecule.
		\item (\emph{Atom}).
			If $U$, $V$ are round molecules, $n \eqdef \dim U = \dim V$, and $\varphi$ is an isomorphism $\bd{}{}U \iso \bd{}{}V$ restricting to $\varphi^\a\colon \bd{}{\a}U \iso \bd{}{\a}V$ for each $\a \in \set{-, +}$, then the \emph{atom} $U \celto V$ obtained by extending the pushout
\[\begin{tikzcd}
	{\bd{}{}U} & {\bd{}{}V} & V \\
	U && {\bd{}{}(U \celto V)}
	\arrow["\varphi", from=1-1, to=1-2]
	\arrow[hook, from=1-1, to=2-1]
	\arrow[hook, from=1-2, to=1-3]
	\arrow[hook, from=1-3, to=2-3]
	\arrow[hook, from=2-1, to=2-3]
	\arrow["\lrcorner"{anchor=center, pos=0.125, rotate=180}, draw=none, from=2-3, to=1-1]
\end{tikzcd}\]
			with a top element $\top$ such that $\faces{}{-}\top \eqdef \gr{n}{U}$ and $\faces{}{+}\top \eqdef \gr{n}{V}$ is a molecule.
	\end{enumerate}
\end{dfn}

\noindent
We also call \emph{atom} any molecule with a greatest element; every atom is either the point or is of the form $U \celto V$ for some round molecules $U$, $V$.
The simplest, smallest molecules are the \emph{globes}, which are in fact atoms.

\begin{dfn}[Globes]
	The class of \emph{globes} is the subclass of molecules generated using only the (\emph{Point}) and (\emph{Atom}) clauses. 
\end{dfn}

\noindent
Up to isomorphism, there is a unique globe $\globe{n}$ for each dimension $n \in \mathbb{N}$, with a single $n$\nbd dimensional element (that we denote by $n$) and two $k$\nbd dimensional elements $k^-, k^+$ such that $\faces{k}{\a}\globe{n} = \set{k^\a}$ for all $\a \in \set{-, +}$ and $k < n$.

\begin{dfn}[Arrow]
	The \emph{arrow} $\arr$ is the 1\nbd globe $\globe{1}$, that is, the atom $1 \celto 1$.
\end{dfn}

\noindent
Molecules satisfy many nice properties.
They are \emph{rigid} in the sense that they have no non-trivial automorphisms in $\ogPos$; this also implies that the isomorphisms $\varphi$ in the definition of pastings or atoms are unique when they exist, which justifies omitting them in the notation.
All input and output boundaries of molecules are molecules.
Moreover, all molecules are globular, and all atoms are round, but not all molecules are round.
Finally, molecules are ``locally atoms'', in the sense that if $U$ is a molecule and $x \in U$, then $\clset{x}$ is an atom.
Turning this into a definition determines the class of \emph{regular directed complexes}.

\begin{dfn}[Regular directed complex]
	A \emph{regular directed complex} is an oriented graded poset $P$ such that, for all $x \in P$, $\clset{x}$ is an atom.
\end{dfn}

\noindent
Regular directed complexes are the \emph{trait d'union} between regular CW complexes and strict $\omega$\nbd categories: given a regular directed complex $P$,
\begin{itemize}
	\item the geometric realisation of the order complex $\ordcpx{P}$ of the underlying poset of $P$ admits a structure of regular CW complex with one cell for each $x \in P$,
	\item the set $\molecin{P}$ of isomorphism classes of morphisms $f\colon U \to P$, with $U$ ranging over molecules, admits a structure of strict $\omega$\nbd category with a minimal set of generators in bijection with the elements of $P$.
\end{itemize}
With this view in mind, when $P$ is a regular directed complex, we will call its elements \emph{cells}.

In particular, if $U$ is a molecule, $\ordcpx{U}$ is contractible, and when $U$ is a round $n$\nbd dimensional molecule, $\ordcpx{U}$ is a PL $n$\nbd ball, while $\ordcpx{\bd{}{}U}$ is a PL $(n-1)$\nbd sphere.
At the same time, if $U$ is an $n$\nbd dimensional molecule, $\molecin{U}$ has a ``greatest'' $n$\nbd cell such that every other cell is a factor, namely, the one represented by $\idd{U}\colon U \to U$.
Then, given a strict $\omega$\nbd category $X$, a functor $d\colon \molecin{U} \to X$ can be seen as a \emph{pasting diagram} in $X$ whose shape is encoded by $U$, and whose composite is the image of the greatest cell.
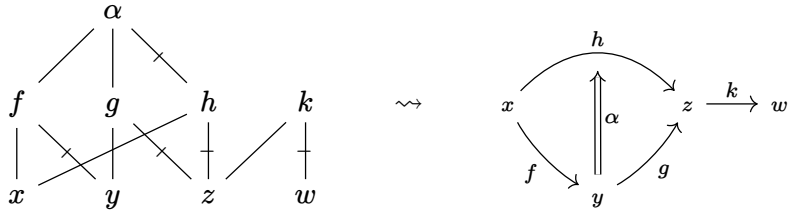
\[\begin{tikzcd}[column sep=scriptsize]
	& \alpha \\
	f & g & h & k & \leadsto & {{\scriptstyle x}} && {{\scriptstyle z}} & {{\scriptstyle w}} \\
	x & y & z & w &&& {{\scriptstyle y}}
	\arrow[no head, from=1-2, to=2-1]
	\arrow[no head, from=1-2, to=2-2]
	\arrow["\shortmid"{marking}, no head, from=1-2, to=2-3]
	\arrow[no head, from=2-1, to=3-1]
	\arrow["\shortmid"{marking}, no head, from=2-1, to=3-2]
	\arrow[no head, from=2-2, to=3-2]
	\arrow["\shortmid"{marking}, no head, from=2-2, to=3-3]
	\arrow[no head, from=2-3, to=3-1]
	\arrow["\shortmid"{marking}, no head, from=2-3, to=3-3]
	\arrow[no head, from=2-4, to=3-3]
	\arrow["\shortmid"{marking}, no head, from=2-4, to=3-4]
	\arrow[""{name=0, anchor=center, inner sep=0}, "h", curve={height=-24pt}, from=2-6, to=2-8]
	\arrow["f"', curve={height=5pt}, from=2-6, to=3-7]
	\arrow["k", from=2-8, to=2-9]
	\arrow["g"', curve={height=5pt}, from=3-7, to=2-8]
	\arrow["\alpha"', shorten <=3pt, shorten >=6pt, Rightarrow, from=3-7, to=0]
\end{tikzcd}\]
Thus, a round molecule may be seen as encoding \emph{a pasting diagram shape which is topologically a ball}.
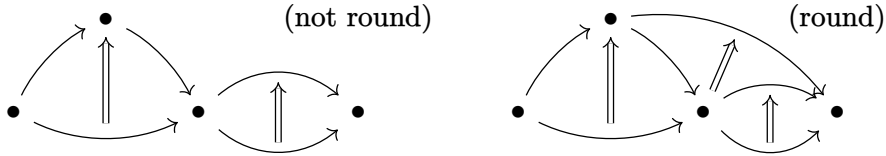
\[\begin{tikzcd}[column sep=scriptsize]
	& \bullet && {\text{(not round)}} && \bullet && {\text{(round)}} \\
	\bullet && \bullet & \bullet & \bullet && \bullet & \bullet
	\arrow[curve={height=-6pt}, from=1-2, to=2-3]
	\arrow[curve={height=-6pt}, from=1-6, to=2-7]
	\arrow[""{name=0, anchor=center, inner sep=0}, curve={height=-18pt}, from=1-6, to=2-8]
	\arrow[curve={height=-6pt}, from=2-1, to=1-2]
	\arrow[""{name=1, anchor=center, inner sep=0}, curve={height=12pt}, from=2-1, to=2-3]
	\arrow[""{name=2, anchor=center, inner sep=0}, curve={height=18pt}, from=2-3, to=2-4]
	\arrow[""{name=3, anchor=center, inner sep=0}, curve={height=-18pt}, from=2-3, to=2-4]
	\arrow[curve={height=-6pt}, from=2-5, to=1-6]
	\arrow[""{name=4, anchor=center, inner sep=0}, curve={height=12pt}, from=2-5, to=2-7]
	\arrow[""{name=5, anchor=center, inner sep=0}, curve={height=18pt}, from=2-7, to=2-8]
	\arrow[""{name=6, anchor=center, inner sep=0}, curve={height=-12pt}, from=2-7, to=2-8]
	\arrow[shorten <=5pt, Rightarrow, from=1, to=1-2]
	\arrow[shorten <=3pt, shorten >=3pt, Rightarrow, from=2, to=3]
	\arrow[shorten <=5pt, Rightarrow, from=4, to=1-6]
	\arrow[shorten <=3pt, shorten >=3pt, Rightarrow, from=5, to=6]
	\arrow[shorten <=2pt, shorten >=3pt, Rightarrow, from=2-7, to=0]
\end{tikzcd}\]
\noindent
Given a round molecule, we can produce an atom with the same boundaries.

\begin{dfn}[Merger of a round molecule]
	Let $U$ be a round molecule.
	The \emph{merger of $U$} is the atom $\mrg{U} \eqdef \bd{}{-}U \celto \bd{}{+}U$.
\end{dfn}

\noindent
The notion of ``factor in a pasting decomposition'' is captured combinatorially by the class of \emph{submolecule inclusions}.

\begin{dfn}[Submolecule inclusion]
	Let $U$, $V$ be molecules.
	The class of \emph{submolecule inclusions} $\iota\colon U \submol V$ is the inductive subclass of embeddings $\iota\colon U \incl V$ generated by the following clauses.
	\begin{enumerate}
		\item (\emph{Pasting factor}).
			If $V = U \cp{k} W$ or $V = W \cp{k} U$ for some molecule $W$ and $k \in \mathbb{N}$, then the canonical inclusion $U \incl V$ is a submolecule inclusion.
		\item (\emph{Isomorphism}).
			If $\varphi\colon U \iso V$ is an isomorphism of molecules, then $\varphi$ is a submolecule inclusion.
		\item (\emph{Composition}).
			If $\iota\colon U \incl W$ and $j\colon W \incl V$ are submolecule inclusions, then $j\iota\colon U \incl V$ is a submolecule inclusion.
	\end{enumerate}
\end{dfn}

\noindent
Given a molecule $U$, $x \in U$, $n \in \mathbb{N}$, and $\a \in \set{-, +}$, we have that the closed subset inclusions $\clset{x} \submol U$ and $\bd{n}{\a}U \submol U$ are all submolecule inclusions.
We can generalise the notion of pasting from pasting along the \emph{entire} output and input $k$\nbd boundary, to pasting along a \emph{submolecule} of one of the two.

\begin{dfn}[Pasting at a submolecule]
	Let $U$, $V$ be molecules, $k \in \mathbb{N}$, and let $\iota\colon \bd{k}{+}U \submol \bd{k}{-}V$ be a submolecule inclusion.
	The \emph{pasting of $U$ at the submolecule $\iota$} is the oriented graded poset $U \cpsub{k, \iota} V$ obtained as the pushout
\[\begin{tikzcd}
	{\bd{k}{+}U} & {\bd{k}{-}V} & V \\
	U && {U \cpsub{k, \iota} V}.
	\arrow["\iota", hook, from=1-1, to=1-2]
	\arrow[hook, from=1-1, to=2-1]
	\arrow[hook, from=1-2, to=1-3]
	\arrow[hook, from=1-3, to=2-3]
	\arrow[hook, from=2-1, to=2-3]
	\arrow["\lrcorner"{anchor=center, pos=0.125, rotate=180}, draw=none, from=2-3, to=1-1]
\end{tikzcd}\]
	Dually, if $\iota\colon \bd{k}{-}U \submol \bd{k}{+}V$, the pasting of $U$ at the submolecule $\iota$ is obtained as the pushout
	\[\begin{tikzcd}
	{\bd{k}{-}U} & {\bd{k}{+}V} & V \\
	U && {V \subcp{k,\iota} U}.
	\arrow["\iota", hook, from=1-1, to=1-2]
	\arrow[hook, from=1-1, to=2-1]
	\arrow[hook, from=1-2, to=1-3]
	\arrow[hook, from=1-3, to=2-3]
	\arrow[hook, from=2-1, to=2-3]
	\arrow["\lrcorner"{anchor=center, pos=0.125, rotate=180}, draw=none, from=2-3, to=1-1]
\end{tikzcd}\]
\end{dfn}

\noindent 
It can be proved that pasting at a submolecule always produces a molecule containing $U$ and $V$ as submolecules.
By convention, both in pasting and pasting at a submolecule, we omit $k$ when it is equal to $\min \set{\dim U, \dim V} - 1$; pastings of this form suffice to generate all molecules.

Pasting diagrams play a central role in the theory of higher-dimensional rewriting, where an $(n+1)$\nbd dimensional cell $U$ is interpreted as the shape of a rewrite rule on $n$\nbd dimensional pasting diagrams, allowing one to substitute a diagram of shape $\bd{}{+}U$ for a diagram of shape $\bd{}{-}U$.
This action by substitution is captured by the following definition.

\begin{dfn}[Substitution at a rewritable submolecule]
	Let $\iota\colon V \submol U$ be a submolecule inclusion.
	We say that $\iota$ is \emph{rewritable} if $\dim V = \dim U$ and $V$ is round.
	If $\iota\colon V \submol U$ is a rewritable submolecule, $n \eqdef \dim V$, and $W$ is a round molecule such that $V \celto W$ is defined, we let $\subs{U}{W}{\iota(V)} \eqdef \bd{}{+}(U \subcp{n,\iota} (V \celto W))$ and call it the \emph{substitution of $W$ for $\iota\colon V \submol U$}.
\end{dfn}

\noindent
Intuitively, $\subs{U}{W}{\iota(V)}$ is $U$ with the submolecule $V$ replaced with the molecule $W$, while preserving the boundary; this always produces a molecule containing $W$ as a submolecule. 

Morphisms $f\colon P \to Q$ between regular directed complexes in $\ogPos$ are very rigid: they are precisely the \emph{local embeddings}, that is, functions that restrict to embeddings on $\clset{x}$ for each $x \in P$.
We let $\rdCpx_\L$ denote the category of regular directed complexes and local embeddings, or, equivalently, the full subcategory of $\ogPos$ on regular directed complexes.

The rigidity of morphisms between regular directed complexes stems from a property called \emph{oriented thinness}: any interval $[x, z]$ of length 2 in $P$ is ``diamond-shaped'', that is, of the form $x < y_1, y_2 < z$ for exactly two elements $y_1, y_2$, and furthermore, letting $\a_i$, $\b_i$ be the unique signs such that $y_i \in \cofaces{}{\a_i}x \cap \faces{}{\b_i}z$ for each $i \in \set{1, 2}$, we have $\a_1\b_1 = -\a_2\b_2$; finally, every 1\nbd dimensional element has exactly one input face and one output face.
Thinness is tied to the topology of combinatorial manifolds \cite{bjorner1995topological}, and oriented thinness---also known as being equipped with a \emph{balanced colouring} \cite{chandler2022broken}---implies that a regular directed complex has a naturally associated augmented chain complex of free abelian groups.

In \cite[Chapter 6]{hadzihasanovic2024combinatorics}, we took two generalised notions of morphism into consideration, determined by the following conditions: an order-preserving map $f\colon P \to Q$ is
\begin{itemize}
	\item a \emph{map} if it \emph{covariantly} determines a functor $\molecin{P} \to \molecin{Q}$,
	\item a \emph{comap} if it \emph{contravariantly} determines a functor $\molecin{Q} \to \molecin{P}$.
\end{itemize}
This comes down to the following explicit definitions.

\begin{dfn}[Map of regular directed complexes]
	Let $P$, $Q$ be regular directed complexes.
	An order-preserving map $f\colon P \to Q$ is a \emph{map} if, for all $n \in \mathbb{N}$, $\a \in \set{-, +}$, and $x \in P$,
	\begin{enumerate}
		\item $f(\bd{n}{\a}x) = \bd{n}{\a}f(x)$, and
		\item for all $y, y' \in \bd{n}{\a}x$, if $f(y) = f(y')$, there is a zig-zag $y \leq y_1 \geq \ldots \leq y_m \geq y'$ in $\bd{n}{\a}x$ such that $f(y) \leq f(y_i)$ for all $i \in \set{1, \ldots, m}$.
	\end{enumerate}
\end{dfn}

\noindent
We say that a map $f\colon P \to Q$ is \emph{final} if, for all $x, x' \in P$, if $f(x) = f(x')$, then there exists a zig-zag $x \leq x_1 \geq \ldots \leq x_m \geq x'$ in $P$ such that $f(x) \leq f(x_i)$ for all $i \in \set{1, \ldots, m}$.
By definition, thus, the restriction of a map to a boundary $\bd{n}{\a}x$ is final onto its image.
The terminology reflects the fact that a map is final if and only if its underlying order-preserving map of posets is final when seen as a functor of posetal categories.
Final maps and local embeddings form an orthogonal factorisation system, lifted from the comprehensive factorisation system on categories and functors \cite{street1973comprehensive}.
In \cite{chanavat2024diagrammatic} and further articles, we also considered a restricted class of maps---the \emph{cartesian maps}---which are those whose underlying map is a Grothendieck fibration of posetal categories.
This notion of map, which has some particularly nice homotopy-theoretic properties, is foundational for the theory of diagrammatic sets.

\begin{dfn}[Comap of regular directed complexes]
	Let $P$, $Q$ be regular directed complexes.
	An order-preserving map $c\colon P \to Q$ is a \emph{comap} if, for all $n \in \mathbb{N}$, $\a \in \set{-, +}$, and $y \in Q$,
	\begin{enumerate}
		\item $\invrs{c}\clset{y}$ is a molecule,
		\item $\bd{n}{\a}\invrs{c}\clset{y} = \invrs{c}\bd{n}{\a}y$.
	\end{enumerate}
\end{dfn}

\noindent
To conclude the section, we recall the definition of \emph{Gray products} and of \emph{joins}, which determine monoidal structures on regular directed complexes compatibly with both maps and comaps, whose units are, respectively, the point $1$ and the empty regular directed complex \( \varnothing \).

\begin{dfn}[Gray product]
	Let $P$, $Q$ be oriented graded posets.
	The \emph{Gray product} of $P$ and $Q$ is the oriented graded poset $P \gray Q$ whose
	\begin{itemize}
		\item underlying graded poset is the product $P \times Q$ of the underlying posets,
		\item orientation is defined, for all $(x, y) \in P \times Q$ and $\alpha \in \set{+, -}$, by the equation $\faces{}{\a}(x, y) = \faces{}{\a}x \times \set{y} \pout{} \set{x} \times \faces{}{(-)^{\dim{x}}\a}y$.
	\end{itemize}
\end{dfn}

\noindent
We say that an oriented graded poset has a \emph{positive least element} if it has a least element $\bot$ such that $\cofaces{}{}\bot = \cofaces{}{+}\bot$.
The full subcategory $\augmo{\ogPos}$ of $\ogPos$ on oriented graded posets with a positive least element is actually equivalent to $\ogPos$, via the pair of functors
\[
	\augm{-}\colon \ogPos \to \augmo{\ogPos}, \quad \quad \dimin{-}\colon \augmo{\ogPos} \to \ogPos
\]
which, respectively, freely add and delete the positive least element.
Moreover, if $P$ and $Q$ have a positive least element, then so does $P \gray Q$; Gray products thus restrict to a monoidal structure on $\augmo{\ogPos}$.
We obtain the join by transporting this monoidal structure via the equivalence $(\augm{-}, \dimin{-})$.

\begin{dfn}[Join]
	Let $P$, $Q$ be oriented graded posets.
	The \emph{join} of $P$ and $Q$ is the oriented graded poset $P \join Q \eqdef \dimin{(\augm{P} \gray \augm{Q})}$.
	For each $x \in P$ and $y \in Q$, we use the notation $\inl{x} \eqdef (x, \bot)$, $\inr{y} \eqdef (\bot, y)$, $x \join y \eqdef (x, y)$ for elements of $P \join Q$; note that $\dim \inl{x} = \dim x$, $\dim \inr{y} = \dim y$, while $\dim (x \join y) = \dim x + \dim y + 1$.
\end{dfn}

\noindent
The classes of atoms, molecules, round molecules, and regular directed complexes are all closed under Gray products and joins.

\subsection{Subdivisions and local collapses} \label{sec:subdivisions}

\noindent
Let $c\colon P \to Q$ be a comap of regular directed complexes.
The inverse image of $c$, as a function from closed subsets of $Q$ to closed subsets of $P$, maps atoms in $Q$ to round molecules in $P$, while preserving their dimension and their partition into interior and $n$\nbd dimensional input and output boundaries for each $n \in \mathbb{N}$; from this perspective, $c$ may be seen as dual to an ``oriented subdivision'' of $Q$.
We will embrace this point of view, and consider a notion of subdivision of regular directed complexes as formally dual to a comap.

\begin{dfn}[Subdivision of regular directed complexes]
	Let $P$ and $Q$ be regular directed complexes.
	A \emph{subdivision} $s\colon P \sd Q$ is a comap $\conv{s}\colon Q \to P$.
\end{dfn}

\noindent
Regular directed complexes and subdivisions form a category $\rdCpx_\S$, which is simply the opposite of the category of regular directed complexes and comaps.
Given a subset $A \subseteq P$ and a subdivision $s\colon P \sd Q$, we will write $s(A) \eqdef \invrs{\conv{s}}A$.
When $K \subseteq P$ is a closed subset, since $\conv{s}$ is order-preserving, $s(K)$ is a closed subset of $Q$.
When $\conv{s}\colon Q \to P$ is invertible, we will identify $s$ with $\invrs{\conv{s}}\colon P \to Q$; by \cite[Proposition 6.3.13]{hadzihasanovic2024combinatorics}, there is no ambiguity in the notion of invertibility between maps and comaps.

\begin{lem} \label{lem:restriction_of_subdivision}
	Let $s\colon P \sd Q$ be a subdivision and let $U \incl P$ be a closed subset.
	Then $\restr{\conv{s}}{s(U)}\colon s(U) \to U$ determines a subdivision $\restr{s}{U}\colon U \sd s(U)$.
\end{lem}
\begin{proof}
	See \cite[Lemma 6.3.14]{hadzihasanovic2024combinatorics}.
\end{proof}

\begin{dfn}[Co-merger of a round molecule]
	Let $U$ be a round molecule.
	The \emph{co-merger of $U$} is the subdivision $\mrg{}_U\colon \mrg{U} \sd U$ determined by the comap
	\[
		x \mapsto \begin{cases}
			\top_U 
			& \text{if $x \in \inter{U}$}, \\
			x 
			& \text{if $x \in \bd{}{}U$}.
		\end{cases}
	\]
\end{dfn}

\begin{dfn}[Substitution along a subdivision]
	Let $P$ be a regular directed complex, $x \in P$, and let $s\colon \clset{x} \sd V$ be a subdivision.
	The \emph{substitution of $V$ for $x$ in $P$ along $s$} is the oriented graded poset $\subs{P}{V}{x}_s$ whose underlying set is the disjoint union $(P \setminus \clset{x}) \pout{} V$ with the partial order and orientation defined, for each $z \in \subs{P}{V}{x}_s$ and $\a \in \set{-, +}$, by
	\[
		\cofaces{}{\a}z \eqdef
		\begin{cases}
			\cofaces{P}{\a}z 
			& \text{if $z \in P \setminus \clset{x}$}, \\
			\cofaces{V}{\a}z \pout{} \bigcup 
			\set{\cofaces{P}{\a}y \mid y = \conv{s}(z), \dim y = \dim z}
			& \text{if $z \in V$}.
		\end{cases}
	\]
\end{dfn}

\begin{comm}
	The union over $y = \conv{s}(z)$ such that $\dim y = \dim z$ is either empty (if $\dim \conv{s}(z) \neq \dim z$) or over a singleton (if $\dim \conv{s}(z) = \dim z$).
\end{comm}

\begin{lem} \label{lem:substitution_preserves_rdcpx}
	Let $P$ be a regular directed complex, $x \in P$, let $s\colon \clset{x} \sd V$ be a subdivision, and let $\conv{t}\colon \subs{P}{V}{x}_s \to P$ be the function defined by
	\[
		z \mapsto
		\begin{cases}
			z 
			& \text{if $z \in P \setminus \clset{x}$}, \\
			\conv{s}(z)
			& \text{if $z \in V$}.
		\end{cases}
	\]
	Then
	\begin{enumerate}
		\item $\subs{P}{V}{x}_s$ is a regular directed complex,
		\item $\conv{t}$ determines a subdivision $t\colon P \sd \subs{P}{V}{x}_s$.
	\end{enumerate}
\end{lem}
\begin{proof}
	It suffices to consider the case where $P$ is a molecule $U$, in which case $\clset{x} \submol U$.
	By \cite[Proposition 6.3.3, Lemma 6.3.9]{hadzihasanovic2024combinatorics}, $V$ is a round molecule, and furthermore $s$ preserves all boundaries and pasting decompositions of submolecules of $\clset{x}$.
	The result then follows by a straightforward induction on submolecules $\clset{x} \submol U' \submol U$.
\end{proof}

\begin{rmk}
	When $\clset{x} \submol U$ is a top-dimensional atom in a molecule $U$, and $V$ is a round molecule with $\bd{}{\a}V$ isomorphic to $\bd{}{\a}x$ for each $\a \in \set{-, +}$, then the substitution of rewritable submolecules $\subs{U}{V}{\clset{x}}$ is obtained as a substitution along the co-merger $\mrg{}_V\colon \mrg{V} \sd V$.
\end{rmk}

\noindent 
Let $I$ be the underlying poset $\set{0^- < 1 > 0^+}$ of $\arr$, let $P$ be a regular directed complex, and let $K \subseteq P$ be a closed subset.
We recall the notion of \emph{partial Gray cylinder on $P$ relative to $K$} from \cite[Section 1.2]{chanavat2024equivalences}, with a new notation.
This is the regular directed complex $\arr \pcyl{K} P$ whose underlying graded poset is the pushout
\[
\begin{tikzcd}
	{I \times K} & K \\
	{I \times P} & {(I \times P) \pout{I \times K} K}
	\arrow[two heads, from=1-1, to=1-2]
	\arrow[hook, from=1-1, to=2-1]
	\arrow["{(-)}", hook, from=1-2, to=2-2]
	\arrow["q_K", two heads, from=2-1, to=2-2]
	\arrow["\lrcorner"{anchor=center, pos=0.125, rotate=180}, draw=none, from=2-2, to=1-1]
\end{tikzcd}
\]
in $\Pos$, whose elements are either of the form $(x)$ for $x \in K$ or $(i, x)$ for $i \in I$ and $x \in P \setminus K$, and orientation is specified by
\begin{align*}
	\faces{}{\a}(x) & \eqdef \set{(y) \mid y \in \faces{}{\a}x}, \\
	\faces{}{\a}(i, x) & \eqdef \begin{cases}
		\set{(0^\a, x)} \cup \set{(1, y) \mid y \in \faces{}{-\a}x \setminus K} &
		\text{if $i = 1$,} \\
		\set{(i, y) \mid y \in \faces{}{\a}x \setminus K} \cup
		\set{(y) \mid y \in \faces{}{\a}x \cap K} &
		\text{otherwise}.
	\end{cases}
\end{align*}
In particular, $\arr \pcyl{\varnothing} P = \arr \gray P$.
As shown in \cite[Lemma 1.20]{chanavat2024equivalences}, when $U$ is a molecule, $\arr \pcyl{K} U$ is also a molecule, which is round whenever $U$ is round and $K \subseteq \bd{}{}U$.

\begin{dfn}[Cylindrical collapse of atoms]
	Let $U$, $V$ be atoms.
	The class of \emph{cylindrical collapses} $p\colon U \surj V$ is the inductive subclass of maps $p\colon U \surj V$ generated by the following clauses.
	\begin{enumerate}
		\item (\emph{Generating collapse}).
			If $U = \arr \pcyl{K} V$ for some $K \subseteq \bd{}{}V$, then the canonical projection $\tau_K\colon \arr \pcyl{K} V \surj V$ is a cylindrical collapse.
		\item (\emph{Isomorphism}).
			If $\varphi\colon U \iso V$ is an isomorphism of atoms, then $\varphi$ is a cylindrical collapse.
		\item (\emph{Composition}).
			If $p\colon U \surj W$ and $q\colon W \surj V$ are cylindrical collapses, then $qp\colon U \surj V$ is a cylindrical collapse.
	\end{enumerate}
\end{dfn}

\begin{dfn}[Local collapse of regular directed complexes]
	Let $P$, $Q$ be regular directed complexes.
	A \emph{local collapse} $f\colon P \to Q$ is a map $f\colon P \to Q$ with the property that, for all $x \in P$, the restricted map $\restr{f}{\clset{x}}\colon \clset{x} \surj \clset{f(x)}$ is a cylindrical collapse.
	A local collapse is a \emph{collapse} if it is a final map.
\end{dfn}

\noindent 
Since, by construction, cylindrical collapses are closed under isomorphisms and under composition, so are local collapses.
Thus, regular directed complexes and local collapses form a category $\rdCpx_{\C\L}$.

\begin{rmk}
	Because isomorphisms are local collapses, every local embedding of regular directed complexes is a local collapse.
	Thus, $\rdCpx_{\C\L}$ contains $\rdCpx_\L$ as a subcategory.
\end{rmk}

\begin{lem} \label{lem:local_collapses_of_atoms}
	Let $U$, $V$ be atoms and let $p\colon U \surj V$ be a surjective map.
	The following are equivalent:
	\begin{enumerate}[label=(\alph*)]
		\item $p$ is a collapse;
		\item $p$ is a local collapse;
		\item $p$ is a cylindrical collapse.
	\end{enumerate}
\end{lem}
\begin{proof}
	Every surjective map of atoms is final, so a local collapse of atoms is always a collapse.
	Moreover, since $U$ is surjective and has a greatest element, if $p$ is a local collapse, then it is a cylindrical collapse.
	Because isomorphisms are local isomorphisms, and local collapses are closed under composition, it suffices to show that generating collapses are local collapses.
	Suppose then that $U = \arr \pcyl{K} V$ and $p = \tau_K$ for some $K \subseteq \bd{}{}V$.
	For each $x \in U$, let $K_x \eqdef K \cap \clset{\tau_K(x)}$. Then $\restr{\tau_K}{\clset{x}}$ is, up to isomorphism, the canonical projection $\tau_{K_x}\colon \arr \pcyl{K_x} \clset{\tau_K(x)} \surj \clset{\tau_K(x)}$, which determines either a generating collapse (if $\tau_K(x) \notin K$) or an isomorphism (if $\tau_K(x) \in K$).
\end{proof}

\begin{prop} \label{prop:factorisation_of_local_collapses}
	There is an orthogonal factorisation system $(\C, \L)$ on $\rdCpx_{\C\L}$ whose left class $\C$ is the class of collapses, and right class $\L$ is the class of local embeddings.
\end{prop}
\begin{proof}
	The proof of \cite[Proposition 6.2.30]{hadzihasanovic2024combinatorics} restricts from all maps of regular directed complexes to local collapses, showing the existence of an essentially unique factorisation of $f$ as a final map $p\colon P \surj P'$ followed by a local embedding $j\colon P' \to Q$.
	Since the property of being a local collapse is local, $p$ is a local collapse, so by definition it is a collapse.
\end{proof}

\noindent The factorisation system restricts to an orthogonal factorisation system $(\C, \E)$ on the full subcategory $\atom_{\C\E}$ on the atoms, in which case, by Lemma \ref{lem:local_collapses_of_atoms}, the left class coincides with the class of cylindrical collapses, while the right class is the class of embeddings of atoms.
The category can be assumed to be skeletal, in which case the factorisation system is strict.

\begin{lem} \label{lem:sections_of_collapses}
	Let $p\colon U \surj V$ be a collapse of atoms.
	Then
	\begin{enumerate}
		\item $p$ admits a section,
		\item if $p'\colon U \surj V$ is another collapse with the same set of sections, then $p = p'$.
	\end{enumerate}
\end{lem}
\begin{proof}
	By \cite[Lemma 1.20]{chanavat2024equivalences}, generating collapses are cartesian maps of atoms.
	Since isomorphisms are also cartesian, and cartesian maps are closed under composition, it follows that all collapses are cartesian.
	We conclude by \cite[Proposition 1.17]{chanavat2024diagrammatic}.
\end{proof}

\begin{comm}
	Proposition \ref{prop:factorisation_of_local_collapses} in combination with Lemma \ref{lem:sections_of_collapses} and the fact that atoms have no non-trivial automorphisms implies that, with its natural grading given by dimension, any skeleton of $\atom_{\C\E}$ is an Eilenberg--Zilber category in the sense of \cite[Definition 1.3.1]{cisinski2019higher}.
\end{comm}

\begin{prop} \label{prop:freeness_of_collapses}
	Let $U$, $V$ be atoms, let $p\colon U \surj V$ be a collapse, and let $m \eqdef \dim{U} - \dim{V}$.
	Then there exists a unique triple of
	\begin{enumerate}
		\item a sequence $(\order{i}{V})_{i=0}^m$ of atoms,
		\item a sequence $(\order{i}{K} \subseteq \order{i-1}{V})_{i=1}^m$ of closed subsets,
		\item an isomorphism $\varphi\colon U \iso \order{m}{V}$,
	\end{enumerate}
	such that
	\begin{enumerate}
		\item $\order{0}{V} = V$ and, for each $i \in \set{1, \ldots, m}$, $\order{i}{V} = \arr \pcyl{\order{i}{K}} \order{i-1}{V}$,
		\item $p = \tau_{\order{1}{K}} \ldots \tau_{\order{m}{K}}\varphi$.
	\end{enumerate}
\end{prop}
\begin{proof}
	By Lemma \ref{lem:local_collapses_of_atoms}, $p$ is a cylindrical collapse, so it is a composite of isomorphisms and generating collapses, and the rigidity of atoms allows us to push isomorphisms to the right of generating collapses.
	This implies both the existence of the entire triple, and the uniqueness of the isomorphism $\varphi$, so it suffices to prove uniqueness.
	Observe, first, that the sequence $(\order{i}{K})_{i=1}^m$ determines uniquely the sequence $(\order{i}{V})_{i=0}^m$.
	We proceed by induction on $m$; when $m = 0$, there is nothing more to prove.
	When $m = 1$, observe that $\order{1}{K}$ is uniquely determined as the set of cells in $V$ whose $p$-fibre is a singleton.
	Finally, suppose that $m > 1$, and $p = \tau_{\order{1}{K}}\ldots\tau_{\order{m}{K}}\varphi = \tau_{\order{1}{L}}\ldots\tau_{\order{m}{L}}\psi$.
	Let $\top$ be the greatest element of $V$ and let $\top_U$ be the greatest element of $U$. 
	Then $\varphi$ maps $\top_U$ to $(1, (1, \ldots, (1, \top)\ldots))$.
	Now, $\tau_{\order{m}{K}}\varphi$ has exactly two sections $j^-$ and $j^+$, whose images are the closures of the faces $x^\a \in \faces{}{\a}\top_U$ that $\varphi$ maps to $(0^\a, (1, \ldots, (1, \top)\ldots))$ for each $\a \in \set{-, +}$.
	Using $\varphi$ as an explicit parametrisation, we can characterise $x^\a$ as the unique cell in $\faces{}{\a}\top_U \cap \invrs{p}\top$ such that, for all $x \in \faces{}{\a}\top_U \cap \invrs{p}\top$, either $x^\a = x$ or $\faces{}{}x^\a \cap \faces{}{}x = \faces{}{-\a}x^\a \cap \faces{}{\a}x$.
	Since this characterisation is independent of the factorisation of $p$, the same two faces determine the images of the two sections of $\tau_{\order{m}{L}}\psi$.
	It follows that, for each $\a \in \set{-, +}$, $\chi^\a \eqdef \tau_{\order{m}{L}}\psi j^\a$ is an isomorphism of atoms, so
	\[
		 \tau_{\order{1}{K}}\ldots\tau_{\order{m-1}{K}} = \tau_{\order{1}{L}}\ldots\tau_{\order{m-1}{L}}\chi^\a,
	\]
	which by the inductive hypothesis implies $(\order{i}{K})_{i=1}^{m-1} = (\order{i}{L})_{i=1}^{m-1}$ and $\chi^\a = \idd{}$.
	Then $\tau_{\order{m}{K}}\varphi$ and $\tau_{\order{m}{L}}\psi$ are parallel collapses with the same set of sections, so by Lemma \ref{lem:sections_of_collapses} they are equal, and we conclude by the case $m = 1$.
\end{proof}

\begin{comm}
	Proposition \ref{prop:freeness_of_collapses} can be interpreted as the statement that collapses of atoms are \emph{freely} generated by the generating collapses, in the sense that no non-trivial equations appear between their composites, even up to isomorphism.
\end{comm}

\subsection{Composing subdivisions and local collapses} \label{sec:composing}

\noindent
We have introduced subdivisions and local collapses as two somewhat orthogonal notions of morphisms of regular directed complexes; the aim of this section is to glue them together.
First of all, from \cite[Proposition 6.3.13]{hadzihasanovic2024combinatorics} we know that comaps and maps intersect exactly at isomorphisms, which justifies identifying an invertible subdivision with an invertible local collapse.

\begin{dfn}[Local subdivision-collapse of regular directed complexes]
	Let $P$, $Q$ be regular directed complexes.
	A \emph{local subdivision-collapse} $[f, s]\colon P \to Q$ is an equivalence class of pairs of
	\begin{enumerate}
		\item a local collapse $f\colon P' \to Q$, and
		\item a subdivision $s\colon P \sd P'$,
	\end{enumerate}
	under the equivalence relation $[f, s] = [f\invrs{\varphi}, \varphi s]$ for all isomorphisms of regular directed complexes $\varphi\colon P' \iso P''$.
\end{dfn}

\begin{comm}
	Since subdivisions are formal duals of comaps, the pair $[f, s]$ is really a \emph{span} of a comap and a cartesian map of regular directed complexes; the equivalence relation also arises from the usual truncation identifying isomorphic spans.
	However, we prefer having both subdivisions and local collapses pointing in their natural direction when seen as higher-categorical functors; recall from \cite[Theorem 6.3.17]{hadzihasanovic2024combinatorics} that the category of regular directed complexes and comaps admits a \emph{contravariant} functor to the category $\omega\Cat$ of strict $\omega$\nbd categories and functors.
\end{comm}

\begin{prop} \label{prop:pullback_of_local_collapse_along_subdivision}
	Let $P$, $Q$, $Q'$ be regular directed complexes, let $f\colon P \to Q$ be a local collapse, let $s\colon Q \sd Q'$ be a subdivision, and consider the pullback 
\[\begin{tikzcd}
	{P'} && P \\
	{Q'} && Q
	\arrow["{\conv{(f^*s)}}", from=1-1, to=1-3]
	\arrow["s^*f", from=1-1, to=2-1]
	\arrow["\lrcorner"{anchor=center, pos=0.125}, draw=none, from=1-1, to=2-3]
	\arrow["f", from=1-3, to=2-3]
	\arrow["{\conv{s}}", from=2-1, to=2-3]
\end{tikzcd}\]
	of the underlying order-preserving maps in $\Pos$.
	Then $P'$ is a graded poset and admits a unique orientation such that
	\begin{enumerate}
		\item $s^*f$ is a local collapse of regular directed complexes,
		\item $\conv{(f^*s)}$ is a comap, dual to a subdivision $f^*s\colon P \sd P'$.
	\end{enumerate}
\end{prop}
\begin{proof}
	An element of $P'$ can be explicitly represented as a pair $(x, y)$ with $x \in P$ and $y \in Q'$ such that $f(x) = \conv{s}(y)$.
	Because $\conv{(f^*s)}$ is order-preserving, and $f$ is both closed and order-preserving,
	\[
		\clset{(x, y)} \subseteq \invrs{\conv{(f^*s)}}\clset{x} \subseteq \invrs{\conv{(f^*s)}}\invrs{f}\clset{f(x)} = \invrs{(s^*f)}\invrs{\conv{s}}\clset{f(x)},
	\]
	which---by the local definition of regular directed complex, local collapse, and subdivision---implies that it suffices to consider the case where $f$ is a collapse of marked atoms, in which case, by \cite[Proposition 6.3.3, Lemma 6.3.9]{hadzihasanovic2024combinatorics}, $Q'$ is a round molecule.
	Finally, by the pasting law for pullbacks, and the fact that isomorphisms of molecules are trivial, it suffices to consider the case where $f$ is a generating collapse $\tau_K\colon \arr \pcyl{K} Q \surj Q$.

	In this case, $P'$ is isomorphic to $(I \times Q') \pout{I \times s(K)} s(K)$ and $s^*f$ to $\tau_{s(K)}$.
	Then, the orientation of $\arr \pcyl{s(K)} Q'$ makes the pullback into a round molecule and $\tau_{s(K)}$ into a local collapse, and is evidently unique with this property.
	Moreover, $\conv{(f^*s)}\colon \arr \pcyl{s(K)} Q' \to \arr \pcyl{K} Q$ is defined by $(x) \mapsto (\conv{s}(x))$ for all $x \in s(K)$ and by $(i, x) \mapsto (i, \conv{s}(x))$ for all $x \in Q' \setminus s(K)$ and $i \in I$, and can be checked explicitly to determine a subdivision, using the fact that $\conv{s}$ determines a subdivision, from which we conclude.
\end{proof}

\noindent
Given local subdivision-collapses $[f, s]\colon P \to Q$ and $[g, t]\colon Q \to R$, we define
\[
	[g, t][f,s] \eqdef [g(t^*f), (f^*t)s]\colon P \to R
\]
with the notation of Proposition \ref{prop:pullback_of_local_collapse_along_subdivision}.
The canonicity of the orientation on the pullback ensures associativity of this assignment, and with the units $[\idd{P}, \idd{P}]$ determines a category $\rdCpx_{\S\C\L}$ of regular directed complexes and local subdivision-collapses.
There are evident inclusions
\[\begin{tikzcd}
	{\rdCpx_{\S}} & {\rdCpx_{\S\C\L}} & {\rdCpx_{\C\L}},
	\arrow[hook, from=1-1, to=1-2]
	\arrow[hook', from=1-3, to=1-2]
\end{tikzcd}\]
the first sending $s\colon P \sd Q$ to $[\idd{Q}, s]\colon P \to Q$, and the second sending $f\colon P \to Q$ to $[f, \idd{P}]\colon P \to Q$.
We will identify subdivisions and local collapses with their images through these inclusions, and treat them as subclasses of local subdivision-collapses.

\begin{rmk}
	Given a local subdivision-collapse $[f, s]$, where $s\colon P \sd P'$ and $f\colon P' \to Q$, by \cite[Theorem 6.3.17]{hadzihasanovic2024combinatorics}, the subdivision $s$ induces a strict functor $\conv{s}^*\colon \molecin{P} \to \molecin{P'}$, while the map $f$ induces a strict functor $f_*\colon \molecin{P'} \to \molecin{Q}$ of strict $\omega$\nbd categories.
	Then, the assignment $[f, s] \mapsto f_*\conv{s}^*\colon \molecin{P} \to \molecin{Q}$ is compatible with the factorisation of subdivisions against local collapses, so it determines a functor from $\rdCpx_{\S\C\L}$ to the category $\omegaCat$ of small strict $\omega$\nbd categories. 
	By \cite[Proposition 6.2.37, Proposition 6.3.19]{hadzihasanovic2024combinatorics}, this functor is pseudomonic, and can be seen as a representation of the category of regular directed complexes and local subdivision-collapses as a subcategory of $\omegaCat$.
\end{rmk}

\begin{prop} \label{prop:ternary_factorisation_systems}
	There is a ternary factorisation system $(\S, \C, \L)$ on $\rdCpx_{\S\C\L}$ whose classes are, respectively, the class $\S$ of subdivisions, the class $\C$ of collapses, and the class $\L$ of local embeddings.
\end{prop}
\begin{proof}
	The fact that $(\S, \C\L)$ is an orthogonal factorisation systems holds essentially by construction, and we conclude by Proposition \ref{prop:factorisation_of_local_collapses}.
\end{proof}

\noindent
By general properties of ternary factorisation systems, \emph{any} pair of classes determines a subcategory of morphisms whose factor in the third class is an isomorphism.
The following definition corresponds to the class $\S\L$.

\begin{dfn}[Local subdivision]
	A local subdivision-collapse $[f, s]\colon P \to Q$ is a \emph{local subdivision} if $f$ is a local embedding.
\end{dfn}

\noindent 
Regular directed complexes and local subdivision-collapses form a category $\rdCpx_{\S\L}$.

\begin{rmk}
	By Lemma \ref{lem:restriction_of_subdivision}, given an embedding $\iota\colon U \incl P$ and a subdivision $s\colon P \sd Q$, the square
\[\begin{tikzcd}[column sep=large]
	U & {s(\iota(U))} \\
	P & Q
	\arrow[arloop->, "{\restr{s}{\iota(U)}\iota}", from=1-1, to=1-2]
	\arrow["\iota", hook, from=1-1, to=2-1]
	\arrow[hook, from=1-2, to=2-2]
	\arrow[arloop->, "s", from=2-1, to=2-2]
\end{tikzcd}\]
	is an $(\S, \L)$ factorisation of $s\iota$, where the $\L$\nbd factor is in fact an embedding.
	Thus $\S$ also forms an orthogonal factorisation system with the restricted right class $\E \subseteq \L$ of embeddings.
	Note that this is not the case for the $(\C, \L)$ factorisation system: the $(\C, \L)$ factorisation of a morphism of the form $p\iota$ with $\iota \in \E$ can have an $\L$\nbd factor which is not an embedding.
\end{rmk}

\noindent
We know from \cite[Proposition 7.2.21]{hadzihasanovic2024combinatorics} that Gray products determine a monoidal structure on the category of regular directed complexes and comaps, so they also determine a monoidal structure $(\rdCpx_\S, \gray, 1)$.
We also know from \cite[Corollary 7.2.18]{hadzihasanovic2024combinatorics} that Gray products determine a monoidal structure $(\rdCpx_\L, \gray, 1)$.
Thus, the following is straightforward.

\begin{prop} \label{prop:gray_product_of_regular_functors}
	Let $[f, s]\colon P \to Q$ and $[g, t]\colon P' \to Q'$ be local subdivisions of regular directed complexes.
	Then $f \gray g$ and $s \gray t$ determine a local subdivision
	\[
		[f, s]\gray[g, t] \eqdef [f \gray g, s \gray t]\colon P \gray P' \to Q \gray Q'.
	\]
	This determines a monoidal structure $(\rdCpx_{\S\L}, \gray, 1)$, such that the inclusions
\[\begin{tikzcd}
	{(\rdCpx_\S, \gray, 1)} & {(\rdCpx_{\S\L}, \gray, 1)} & {(\rdCpx_\L, \gray, 1)}
	\arrow[hook, from=1-1, to=1-2]
	\arrow[hook', from=1-3, to=1-2]
\end{tikzcd}\]
	are strong monoidal.
\end{prop}

\begin{comm}
	On the other hand, while Gray products also extend to maps and to cartesian maps of regular directed complexes, this is \emph{not} the case for local collapses; indeed, by non-commutativity of Gray products, the Gray product of two cylindrical collapses is almost never a cylindrical collapse.
\end{comm}

\noindent
We have a similar story for joins, using \cite[Proposition 7.4.22, Proposition 7.4.23]{hadzihasanovic2024combinatorics}.

\begin{prop} \label{prop:join_of_local_sd_collapses}
	Let $[f, s]\colon P \to Q$ and $[g, t]\colon P' \to Q'$ be local subdivisions of regular directed complexes.
	Then $f \join g$ and $s \join t$ determine a local subdivision
	\[
		[f, s]\join[g, t] \eqdef [f \join g, s \join t]\colon P \join P' \to Q \join Q'.
	\]
	This determines a monoidal structure $(\rdCpx_{\S\L}, \join, \varnothing)$, such that the inclusions
\[\begin{tikzcd}
	{(\rdCpx_\S, \join, \varnothing)} & {(\rdCpx_{\S\L}, \join, \varnothing)} & {(\rdCpx_\L, \join, \varnothing)}
	\arrow[hook, from=1-1, to=1-2]
	\arrow[hook', from=1-3, to=1-2]
\end{tikzcd}\]
	are strong monoidal.
\end{prop}

\subsection{Colimits of finite regular directed complexes} \label{sec:colimits}

\noindent 
We conclude this part of the article by studying some colimits in categories of regular directed complexes.
For simplicity, and because this is the only case that we will need, we will focus only on categories of \emph{finite} regular directed complexes; the extension to the infinite case is straightforward but involves some technicalities about transfinite compositions that we wish to avoid.

Restricting to the full subcategories on finite objects, we have a diagram of inclusions of subcategories
\begin{equation} \label{eq:subcategories_of_scl}
\begin{tikzcd}
	{\frdCpx_\L} & {\frdCpx_{\S\L}} & {\frdCpx_\S} \\
	{\frdCpx_{\C\L}} & {\frdCpx_{\S\C\L}} & {\frdCpx_{\S\C}}
	\arrow[hook, from=1-1, to=1-2]
	\arrow[hook, from=1-1, to=2-1]
	\arrow[hook, from=1-2, to=2-2]
	\arrow[hook', from=1-3, to=1-2]
	\arrow[hook, from=1-3, to=2-3]
	\arrow[hook, from=2-1, to=2-2]
	\arrow[hook', from=2-3, to=2-2]
\end{tikzcd}
\end{equation}
Let
\begin{align*}
	\Ibd & \eqdef \set{\bdmap_U\colon \bd{}{}U \incl U \mid \text{$U$ is an atom}}, \\
	\Imrg & \eqdef \set{\mrg{}_U\colon \mrg{U} \sd U \mid \text{$U$ is a round molecule}},
\end{align*}
which we call the set of \emph{boundary inclusions} and the set of \emph{co-mergers}.

\begin{prop} \label{prop:initial_object}
	The empty regular directed complex $\varnothing$ is a strict initial object in $\frdCpx_\L$, preserved by all possible inclusions in diagram (\ref{eq:subcategories_of_scl}).
\end{prop}
\begin{proof}
	The fact that $\varnothing$ is a strict initial object in $\frdCpx_{\C\L}$ follows from the same property in $\Pos$, and the only subdivision with domain or codomain $\varnothing$ is the identity.
\end{proof}

\begin{prop} \label{prop:pushouts_of_embeddings_along_embeddings}
	The category $\frdCpx_\L$ has pushouts of embeddings along embeddings, preserved by all possible inclusions in diagram (\ref{eq:subcategories_of_scl}).
	Moreover, embeddings are stable under these pushouts.
\end{prop}
\begin{proof}
	The fact that $\frdCpx_\L$ has these pushouts and that they are preserved by the inclusion into $\frdCpx_{\C\L}$ is a simple variant of \cite[Proposition 6.2.27]{hadzihasanovic2024combinatorics}, restricted from all maps to local collapses; moreover, the pushouts are constructed as in $\Pos$.
	It then suffices to show that these squares are still universal after inclusion into $\frdCpx_{\S\L}$ and into $\frdCpx_{\S\C\L}$.
	We will consider the latter and note that the proof restricts to the former.
	Consider a cone under a span $(\iota, \iota')$ of embeddings in $\frdCpx_{\S\C\L}$.
	Using the $(\S, \C\L)$ factorisation, and the fact that $(\S, \E)$ form an orthogonal factorisation system, this corresponds to an essentially unique commutative diagram
\[\begin{tikzcd}
	& P & {s(P)} \\
	U & {t(U)} && R \\
	& {P'} & {s'(P')}
	\arrow[arloop->, "s", from=1-2, to=1-3]
	\arrow["{f\in \C\L}", from=1-3, to=2-4]
	\arrow["\iota", hook, from=2-1, to=1-2]
	\arrow[arloop->, "t", from=2-1, to=2-2]
	\arrow["{\iota'}", hook, from=2-1, to=3-2]
	\arrow["j", hook, from=2-2, to=1-3]
	\arrow["{j'}", hook, from=2-2, to=3-3]
	\arrow[arloop->, "{s'}", from=3-2, to=3-3]
	\arrow["{f' \in \C\L}", from=3-3, to=2-4]
\end{tikzcd}\]
	Ignoring $f$ and $f'$, the rest of the diagram has, by construction of the factorisations as in Lemma \ref{lem:restriction_of_subdivision}, an underlying commutative diagram in $\Pos$
	\[\begin{tikzcd}
	& P & {s(P)} \\
	U & {t(U)} \\
	& {P'} & {s'(P')}
	\arrow["{\conv{s}}"', from=1-3, to=1-2]
	\arrow["\iota", hook, from=2-1, to=1-2]
	\arrow["{\iota'}", hook, from=2-1, to=3-2]
	\arrow["j", hook, from=2-2, to=1-3]
	\arrow["{\conv{t}}"', from=2-2, to=2-1]
	\arrow["{j'}", hook, from=2-2, to=3-3]
	\arrow["{\conv{s'}}"', from=3-3, to=3-2]
\end{tikzcd}\]
	which is a morphism of spans, so it induces a unique order-preserving map between their pushouts 
\[\begin{tikzcd}
	& P & {s(P)} \\
	U && {P \pout{U} P'} & {s(P) \pout{t(U)} s'(P')} \\
	& {P'} & {s'(P')}
	\arrow[hook, from=1-2, to=2-3]
	\arrow["{\conv{s}}"', from=1-3, to=1-2]
	\arrow[hook, from=1-3, to=2-4]
	\arrow["\iota", hook, from=2-1, to=1-2]
	\arrow["{\iota'}", hook, from=2-1, to=3-2]
	\arrow["\lrcorner"{anchor=center, pos=0.125, rotate=-135}, draw=none, from=2-3, to=2-1]
	\arrow["{\conv{t'}}"', dashed, from=2-4, to=2-3]
	\arrow[hook, from=3-2, to=2-3]
	\arrow[hook, from=3-3, to=2-4]
	\arrow["{\conv{s'}}"', from=3-3, to=3-2]
\end{tikzcd}\]
	and it is straightforward to verify that $\conv{t'}$ lifts to a (necessarily unique) comap between the tips of the two cones lifted to $\rdCpx_{\S\C\L}$.
	We conclude by the already proven universality of the pushout $s(P) \pout{t(U)} s'(P')$ in $\rdCpx_{\C\L}$.
\end{proof}

\begin{prop} \label{prop:construction_of_embeddings_from_boundary_inclusions}
	Every embedding of finite regular directed complexes can be constructed in $\frdCpx_\L$ as a composite of pushouts of boundary inclusions along embeddings.
\end{prop}
\begin{proof}
	Let $\iota\colon P \incl Q$ be an embedding, let $A \eqdef Q \setminus \iota(P)$, and for each $i \leq n \eqdef \dim Q$, fix an ordering $(\order{i,j}{x})_{j=1}^{m_i}$ of the elements of $\gr{i}{A}$.
	Then, letting $m_{-1} \eqdef 0$, we define, recursively on $i \geq -1$ and $j \in \set{0,\ldots,m_i}$, regular directed complexes $\order{i,j}{P}$ that fit into a factorisation
	\[
		P \incl \order{i,j}{P} \incl Q
	\]
	and such that the embedding of $\order{i, j}{P}$ into $Q$ is surjective on cells of dimension $< i$.
	We let $\order{-1,0}{P} \eqdef P$, and then, for $i \geq 0$, $\order{i,0}{P} \eqdef \order{i-1, m_{i-1}}{P}$ and, for $j \in \set{1, \ldots, m_i}$, we let $\order{i,j}{P}$ be defined by the pushout
\[\begin{tikzcd}
	{\bd{}{}\order{i,j}{x}} & {\clset{\order{i,j}{x}}} \\
	{\order{i,j-1}{P}} & {\order{i,j}{P}}
	\arrow[hook, from=1-1, to=1-2]
	\arrow[hook, from=1-1, to=2-1]
	\arrow[hook, from=1-2, to=2-2]
	\arrow[hook, from=2-1, to=2-2]
	\arrow["\lrcorner"{anchor=center, pos=0.125, rotate=180}, draw=none, from=2-2, to=1-1]
\end{tikzcd}\]
where the leftmost vertical embedding picks the closed subset whose image in $Q$ is equal to $\bd{}{}\order{i,j}{x}$, compatibly with the orientation of $\order{i,j}{x}$.
	Then the evident square induces a unique embedding $\order{i,j}{P} \incl Q$.
	By construction, $\order{n,m_n}{P} \incl Q$ is an isomorphism.
\end{proof}

\begin{lem} \label{lem:pushouts_of_co-mergers}
	The category $\frdCpx_{\S\L}$ has pushouts of co-mergers along embeddings, preserved by the inclusion $\frdCpx_{\S\L} \incl \frdCpx_{\S\C\L}$.
\end{lem}
\begin{proof}
	Consider a span $(\iota, s)$ of an embedding and a co-merger; because the domain of a co-merger is an atom, we may assume that $\iota$ is the inclusion of an atom $\clset{x} \incl P$.
	From Lemma \ref{lem:substitution_preserves_rdcpx}, we have a commutative square
\begin{equation} \label{eq:pushout_of_comerger}
	\begin{tikzcd}
	{\clset{x}} & V \\
	P & {\subs{P}{V}{x}_s}
	\arrow[arloop->, "s", from=1-1, to=1-2]
	\arrow["\iota", hook, from=1-1, to=2-1]
	\arrow["j", hook, from=1-2, to=2-2]
	\arrow[arloop->, "t", from=2-1, to=2-2]
\end{tikzcd}
\end{equation}
	which we claim is a pushout in $\frdCpx_{\S\L}$ and in $\frdCpx_{\S\C\L}$.
	We consider the latter category, and note that the proof restricts to the former.
	Consider a cone under $(\iota, s)$ in $\frdCpx_{\S\C\L}$.
	Using the $(\S, \C\L)$ factorisation, and the fact that $(\S, \E)$ form an orthogonal factorisation system, this corresponds to an essentially unique commutative diagram
	\[\begin{tikzcd}[row sep=scriptsize]
	& V & {t'(V)} \\
	{\clset{x}} &&&& Q. \\
	& P && {s'(P)}
	\arrow[arloop->, "{t'}", from=1-2, to=1-3]
	\arrow["{f \in \C\L}", curve={height=-18pt}, from=1-3, to=2-5]
	\arrow["{\iota'}", hook, from=1-3, to=3-4]
	\arrow[arloop->, "s", from=2-1, to=1-2]
	\arrow[arloop->, "t's", curve={height=12pt}, from=2-1, to=1-3]
	\arrow["\iota", hook, from=2-1, to=3-2]
	\arrow[arloop->, "{s'}", from=3-2, to=3-4]
	\arrow["{f' \in \C\L}", from=3-4, to=2-5]
\end{tikzcd}\]
	Now, observe that, by construction, the commutative squares of underlying order-preserving maps
\begin{equation} \label{eq:two_pullback_squares}
\begin{tikzcd}
	{\clset{x}} & V \\
	P & {\subs{P}{V}{x}_s,}
	\arrow["\iota", hook, from=1-1, to=2-1]
	\arrow["{\conv{s}}"', from=1-2, to=1-1]
	\arrow["\lrcorner"{anchor=center, pos=0.125, rotate=-90}, draw=none, from=1-2, to=2-1]
	\arrow["j", hook, from=1-2, to=2-2]
	\arrow["{\conv{t}}"', from=2-2, to=2-1]
\end{tikzcd}
\begin{tikzcd}
	{\clset{x}} & V & {t'(V)} \\
	P && {s'(P)}
	\arrow["\iota", hook, from=1-1, to=2-1]
	\arrow["{\conv{s}}"', from=1-2, to=1-1]
	\arrow["{\conv{t'}}"', from=1-3, to=1-2]
	\arrow["\lrcorner"{anchor=center, pos=0.125, rotate=-90}, draw=none, from=1-3, to=2-1]
	\arrow["{\iota'}", hook, from=1-3, to=2-3]
	\arrow["{\conv{s'}}"', from=2-3, to=2-1]
\end{tikzcd}
\end{equation}
	are pullback squares in $\Pos$, so the assignment
	\[
		y \mapsto \begin{cases}
			\conv{s'}(y) 
			& \text{if $\conv{s'}(y) \in P \setminus \clset{x}$}, \\
			\conv{t'}(y')
			& \text{if $\conv{s'}(y) \in \clset{x}$, $y' \in t'(V)$, $y = \iota'(y')$}
		\end{cases}
	\]
	determines an order-preserving map $\conv{u}\colon s'(P) \to \subs{P}{V}{x}_s$, which is easily determined to be a comap; moreover, this is evidently unique with the property that the rightmost square in (\ref{eq:two_pullback_squares}) factors through the leftmost as
	\[\begin{tikzcd}
	{\clset{x}} & V & {t'(V)} \\
	P & {\subs{P}{V}{x}_s} & {s'(P)}.
	\arrow["\iota", hook, from=1-1, to=2-1]
	\arrow["{\conv{s}}"', from=1-2, to=1-1]
	\arrow["j", hook, from=1-2, to=2-2]
	\arrow["{\conv{t'}}"', from=1-3, to=1-2]
	\arrow["\lrcorner"{anchor=center, pos=0.125, rotate=-90}, draw=none, from=1-3, to=2-1]
	\arrow["{\iota'}", hook, from=1-3, to=2-3]
	\arrow["{\conv{t}}"', from=2-2, to=2-1]
	\arrow["{\conv{u}}"', dashed, from=2-3, to=2-2]
\end{tikzcd}\]
	This proves the universality of the square (\ref{eq:pushout_of_comerger}).
\end{proof}

\begin{prop} \label{prop:construction_of_subdivisions_from_comergers}
	Every subdivision of finite regular directed complexes can be constructed in $\frdCpx_{\S\L}$ as a composition of pushouts of co-mergers along embeddings.
\end{prop}
\begin{proof}
	Let $s\colon P \sd Q$ be a subdivision, and for each $i \leq n \eqdef \dim{P} = \dim{Q}$, fix an ordering $(\order{i,j}{x})_{j=1}^{m_i}$ of the elements of $\gr{i}{P}$.
	Then, letting $m_{-1} \eqdef 0$, we define, recursively on $i \geq -1$ and $j \in \set{0, \ldots, m_i}$, regular directed complexes $\order{i,j}{P}$ that fit into a factorisation
	\[
		P \sd \order{i,j}{P} \sd Q
	\]
	and such that the restriction of $\order{i,j}{P} \sd Q$ to cells of dimension $< i$ is an isomorphism, while $P \sd \order{i,j}{P}$ is bijective on the cells $\order{i',j'}{x}$ with $i' > i$ or with $i' = i$ and $j' > j$.
	We let $\order{-1, 0}{P} \eqdef P$, and then, for $i \geq 0$, $\order{i,0}{P} \eqdef \order{i-1,m_{i-1}}{P}$ and, for $j \in \set{1, \ldots, m_i}$, we let $\order{i,j}{P}$ be defined by the pushout
	\[\begin{tikzcd}
	{\mrg{s\left(\clset{\order{i,j}{x}}\right)}} & {s\left(\clset{\order{i,j}{x}}\right)} \\
	{\order{i,j-1}{P}} & {\order{i,j}{P}}
	\arrow[arloop->, "{\mrg{}}", from=1-1, to=1-2]
	\arrow[hook, from=1-1, to=2-1]
	\arrow[hook, from=1-2, to=2-2]
	\arrow[arloop->, from=2-1, to=2-2]
	\arrow["\lrcorner"{anchor=center, pos=0.125, rotate=180}, draw=none, from=2-2, to=1-1]
\end{tikzcd}\]
	where the leftmost embedding picks the unique cell in the image of $\order{i,j}{x}$ through $P \sd \order{i,j-1}{P}$.
	The evident square induces a unique subdivision $\order{i,j}{P} \sd Q$.
	By construction, $\order{n,m_n}{P} \sd Q$ is an isomorphism.
\end{proof}

\begin{prop} \label{prop:pushouts_of_subdivisions_along_embeddings}
	The category $\frdCpx_{\S\L}$ has pushouts of subdivisions along embeddings, preserved by the inclusion into $\frdCpx_{\S\C\L}$.
	Moreover, both subdivisions and embeddings are stable under these pushouts.
\end{prop}
\begin{proof}
	Follows from Lemma \ref{lem:pushouts_of_co-mergers} combined with Proposition \ref{prop:construction_of_subdivisions_from_comergers} by repeated application of the pasting law for pushouts.
\end{proof}

\begin{comm}
	An explicit construction of the pushout of $\iota\colon U \incl P$ along a subdivision $s\colon U \sd V$ is the oriented graded poset $\subs{P}{V}{\iota(U)}_s$ whose underlying set is $(P \setminus \iota(U)) \pout{} V$ with the partial order and orientation defined, for each $z \in \subs{P}{V}{\iota(U)}_s$ and $\a \in \set{-, +}$, by
	\[
		\cofaces{}{\a}z \eqdef
		\begin{cases}
			\cofaces{P}{\a}z 
			& \text{if $z \in P \setminus \iota(U)$}, \\
			\cofaces{V}{\a}z + \bigcup 
			\set{\cofaces{P}{\a}y \mid y = \iota(\conv{s}(z)), \dim y = \dim z}
			& \text{if $z \in V$},
		\end{cases}
	\]
	an evident generalisation of substitution.
\end{comm}

\begin{dfn}[The classes of colimits $\Gamma$ and $\Gamma_\S$]
	We let $\Gamma$ denote the class of colimit cones in $\frdCpx_\L$ consisting of
	\begin{enumerate}
		\item the initial object,
		\item pushouts of boundary inclusions along embeddings.
	\end{enumerate}
	We let $\Gamma_\S$ denote the class of colimit cones in $\frdCpx_{\S\L}$ which additionally contains pushouts of co-mergers along embeddings.
\end{dfn}

\begin{rmk}
	There are countably many isomorphism classes of finite regular directed complexes, and each of them has finitely many embeddings and subdivisions into it, so the classes $\Gamma$ and $\Gamma_\S$ are essentially small, that is, they can be treated as sets.
\end{rmk}

\begin{prop} \label{prop:gamma_continuous_functors}
	Let $\fun{F}$ be a functor defined on one of the subcategories of $\frdCpx_{\S\C\L}$ in diagram (\ref{eq:subcategories_of_scl}).
	Then
	\begin{enumerate}
		\item if $\fun{F}$ is defined at least on $\frdCpx_{\L}$ and preserves $\Gamma$\nbd colimits, then it preserves finite coproducts and all pushouts of embeddings along embeddings,
		\item if $\fun{F}$ is defined at least on $\frdCpx_{\S\L}$ and preserves $\Gamma_\S$\nbd colimits, then it also preserves all pushouts of subdivisions along embeddings.
	\end{enumerate}
\end{prop}
\begin{proof}
	Follows from Proposition \ref{prop:construction_of_embeddings_from_boundary_inclusions} and Proposition \ref{prop:construction_of_subdivisions_from_comergers}.
\end{proof}

\noindent We will refer also to these more general colimits as ``$\Gamma$\nbd colimits'' and ``$\Gamma_\S$\nbd colimits'', respectively.

\begin{prop} \label{prop:gray_preservation_of_colimits}
	Let $P$ be a regular directed complex.
	The endofunctors $P \gray -$ and $- \gray P$
	\begin{enumerate}
		\item preserve all $\Gamma$-colimits when defined on $\frdCpx_\L$,
		\item preserve all $\Gamma_\S$-colimits when defined on $\frdCpx_{\S\L}$.
	\end{enumerate}
\end{prop}
\begin{proof}
	Since $\Gamma$\nbd colimits are preserved and reflected by the forgetful functor to $\Pos$, the first result follows from the fact that $P \times -$ and $- \times P$ preserves both embeddings and all colimits, since $\Pos$ is cartesian closed.
	As for $\Gamma_\S$-colimits, consider a pushout square of an embedding $\iota\colon \clset{x} \incl Q$ along a co-merger $s\colon \clset{x} \sd V$.
	Then the natural bijection
	\[
		P \times ((Q \setminus \clset{x}) \pout{} V) \simeq
		(P \times (Q \setminus \clset{x})) \pout{} (P \times V)
	\]
	given by distributivity of products over disjoint unions lifts to an isomorphism between $P \gray \subs{Q}{V}{x}_s$ and $\subs{(P \gray Q)}{P \gray V}{P \gray \clset{x}}_{\idd{P} \gray s}$.
\end{proof}

\begin{prop} \label{prop:join_preservation_of_connected_colimits}
	Let $P$ be a regular directed complex.
	The endofunctors $P \join -$ and $- \join P$
	\begin{enumerate}
		\item preserve all connected $\Gamma$-colimits when defined on $\frdCpx_\L$,
		\item preserve all connected $\Gamma_\S$-colimits when defined on $\frdCpx_{\S\L}$.
	\end{enumerate}
\end{prop}
\begin{proof}
	Simple variant of Proposition \ref{prop:gray_preservation_of_colimits} using \cite[Lemma 1.3.23]{hadzihasanovic2024combinatorics}.
\end{proof}

\begin{rmk}
	The only non-connected $\Gamma_\S$-colimit is the initial object.
\end{rmk}

\section{Homotopy-theoretic background} \label{part:homotopy}

\subsection{Recollections on weak model categories} \label{sec:weakmodel}

\noindent
We recall some definitions and results about weak model structures from \cite{henry2020weak}.

\begin{dfn}[Class of cofibrations]
	Let $\cat{C}$ be a category with an initial object $\init$.
	A \emph{class of cofibrations} is a class $\Cof$ of morphisms in $\cat{C}$, called \emph{cofibrations} and notated as $\iota\colon A \incl B$, satisfying the following properties.
	For each $A \in \Ob \cat{C}$, say that $A$ is \emph{cofibrant} if the unique morphism $\init \to A$ is a cofibration.
	Then
	\begin{enumerate}
		\item $\init$ is cofibrant,
		\item if $\iota\colon A \incl B$ is a cofibration, then $A$ is cofibrant,
		\item if $A$ is cofibrant and $\varphi\colon A \iso B$ is an isomorphism, then $\varphi$ is a cofibration,
		\item if $\iota\colon A \incl B$ and $j\colon B \incl C$ are cofibrations, then so is $j\iota\colon A \incl C$,
		\item if $\iota\colon A \incl B$ is a cofibration, $f\colon A \to C$ a morphism, and $C$ is cofibrant, then the pushout
	\[\begin{tikzcd}
	A & C \\
	B & {B \pout{A} C}
	\arrow["f", from=1-1, to=1-2]
	\arrow["\iota", hook, from=1-1, to=2-1]
	\arrow["j", hook, from=1-2, to=2-2]
	\arrow[from=2-1, to=2-2]
	\arrow["\lrcorner"{anchor=center, pos=0.125, rotate=180}, draw=none, from=2-2, to=1-1]
\end{tikzcd}\]
	exists and $j\colon C \incl B \pout{A} C$ is a cofibration.
	\end{enumerate}
\end{dfn}

\noindent The following is dual to the previous definition, but we spell it out explicitly to set some notation.

\begin{dfn}[Class of fibrations]
	Let $\cat{C}$ be a category with a terminal object $\term$.
	A \emph{class of fibrations} is a class $\Fib$ of morphisms in $\cat{C}$, called \emph{fibrations} and notated as $p\colon X \surj Y$, satisfying the following properties.
	For each $X \in \Ob \cat{C}$, say that $X$ is \emph{fibrant} if the unique morphism $X \to \term$ is a fibration.
  	Then
	\begin{enumerate}
		\item $\term$ is fibrant,
		\item if $p\colon X \surj Y$ is a fibration, then $Y$ is fibrant,
		\item if $Y$ is fibrant and $\varphi\colon X \iso Y$ is an isomorphism, then $\varphi$ is a fibration,
		\item if $p\colon X \surj Y$ and $q\colon Y \surj Z$ are fibrations, then so is $qp\colon X \surj Z$,
		\item if $p\colon X \surj Y$ is a fibration, $f\colon Z \to Y$ a morphism, and $Z$ is fibrant, then the pullback
			\[\begin{tikzcd}
	{Z \pback{Y} X} & X \\
	Z & Y
	\arrow[from=1-1, to=1-2]
	\arrow["q", two heads, from=1-1, to=2-1]
	\arrow["\lrcorner"{anchor=center, pos=0.125}, draw=none, from=1-1, to=2-2]
	\arrow["p", two heads, from=1-2, to=2-2]
	\arrow["f", from=2-1, to=2-2]
\end{tikzcd}\]
	exists and $q\colon Z \pback{Y} X \surj Z$ is a fibration.
	\end{enumerate}
\end{dfn}

\begin{comm}
	The second condition---that domains of cofibrations must be cofibrant, and codomains of fibrations fibrant---is not part of the original definition, but as discussed in \cite[Remark 2.1.4]{henry2020weak}, it is not restrictive to assume it; we do so to avoid constantly specifying that a cofibration is between cofibrant objects and a fibration is between fibrant objects.
\end{comm}

\noindent
Given a class of morphisms $\cls{A}$, we let $\rlp (\cls{A})$ denote the class of morphisms with the right lifting property against $\cls{A}$, and, dually, $\llp (\cls{A})$ denote the class of morphisms with the left lifting property against $\cls{A}$.

\begin{dfn}[Acyclic cofibrations and acyclic fibrations]
	Let $(\cat{C}, \Cof, \Fib)$ be a category with a class of cofibrations and a class of fibrations.
	The class of \emph{acyclic cofibrations} in $\cat{C}$ is $\ACof \eqdef \Cof \cap \llp(\Fib)$.
	Dually, the class of \emph{acyclic fibrations} in $\cat{C}$ is $\AFib \eqdef \Fib \cap \rlp(\Cof)$.
	We notate an acyclic cofibration as $\iota\colon A \acof B$ and an acyclic fibration as $p\colon X \afib Y$.
\end{dfn}

\begin{dfn}[Relative cylinder]
	Let $(\cat{C}, \Cof, \Fib)$ be a category with a class of cofibrations and a class of fibrations, let $j\colon A \incl B$ be a cofibration, and consider the commutative diagram
\[\begin{tikzcd}
	A & B \\
	B & {B \pout{A} B} & B.
	\arrow["j", hook, from=1-1, to=1-2]
	\arrow["j", hook, from=1-1, to=2-1]
	\arrow["{\iota_2}", hook, from=1-2, to=2-2]
	\arrow[curve={height=-12pt}, equals, from=1-2, to=2-3]
	\arrow["{\iota_1}", hook, from=2-1, to=2-2]
	\arrow[curve={height=30pt}, equals, from=2-1, to=2-3]
	\arrow["\lrcorner"{anchor=center, pos=0.125, rotate=180}, draw=none, from=2-2, to=1-1]
	\arrow["{\rcodiag{j}}", dashed, from=2-2, to=2-3]
\end{tikzcd}\]
	A \emph{relative cylinder for $j$} is a cofibration $(\ell^-, \ell^+)\colon B \pout{A} B \incl I_A B$ such that
	\begin{enumerate}
		\item there exist an acyclic cofibration $d\colon B \acof D_A B$ and a factorisation of $d\rcodiag{j}$ through $(\ell^-, \ell^+)$,
		\item $\ell^-\colon B \acof I_A B$ is an acyclic cofibration.
	\end{enumerate}
	A relative cylinder for $j$ is \emph{strong} if the acyclic cofibration $d$ can be taken to be $\idd{B}$.
	If $A$ is cofibrant, then a \emph{cylinder for $A$} is a relative cylinder for the unique morphism $\init \incl A$.
\end{dfn}

\begin{dfn}[Relative path object]
	Let $(\cat{C}, \Cof, \Fib)$ be a category with a class of cofibrations and a class of fibrations, let $p\colon X \surj Y$ be a fibration, and consider the commutative diagram
\[\begin{tikzcd}
	X & {X \pback{Y} X} & X \\
	& X & Y.
	\arrow["{\rdiag{p}}", dashed, from=1-1, to=1-2]
	\arrow[curve={height=-30pt}, equals, from=1-1, to=1-3]
	\arrow[curve={height=12pt}, equals, from=1-1, to=2-2]
	\arrow["{\pi_2}", two heads, from=1-2, to=1-3]
	\arrow["{\pi_1}", two heads, from=1-2, to=2-2]
	\arrow["\lrcorner"{anchor=center, pos=0.125}, draw=none, from=1-2, to=2-3]
	\arrow["p", two heads, from=1-3, to=2-3]
	\arrow["p", two heads, from=2-2, to=2-3]
\end{tikzcd}\]
	A \emph{relative path object for $p$} is a fibration $(q^-, q^+)\colon P_Y X \surj X \pback{Y} X$ such that
	\begin{enumerate}
		\item there exist an acyclic fibration $t\colon T_Y X \afib X$ and a factorisation of $\rdiag{p}t$ through $(q^-, q^+)$,
		\item $q^- \colon P_Y X \afib X$ is an acyclic fibration.
	\end{enumerate}
	A relative path object for $p$ is \emph{strong} if the acyclic fibration $t$ can be taken to be $\idd{X}$.
	If $X$ is fibrant, then a \emph{path object for $X$} is a relative path object for the unique morphism $X \surj \term$.
\end{dfn}

\begin{dfn}[Weak model category]
	A \emph{weak model category} $\Mw = (\cat{C}, \Cof, \Fib)$ is a category endowed with a class of cofibrations and a class of fibrations satisfying the following axioms.
	\begin{enumerate}
		\item (\emph{Factorisation}). Every morphism $A \to X$ with $A$ cofibrant and $X$ fibrant factors both as $A \incl X' \afib X$ and as $A \acof A' \surj X$.
		\item (\emph{Cylinder}). Every cofibration $A \incl X$ with $X$ fibrant admits a relative strong cylinder object.
		\item (\emph{Path object}). Every fibration $A \surj X$ with $A$ cofibrant admits a relative strong path object.
	\end{enumerate}
\end{dfn}

\noindent
The factorisation axiom, applied to the unique morphisms $A \to \term$ and $\init \to X$, produces fibrant replacements of cofibrant objects and cofibrant replacements of fibrant objects.

\begin{dfn}[Fibrant and cofibrant replacement]
	Let $\Mw$ be a weak model category, $A$ a cofibrant object, and $X$ a fibrant object.
	A \emph{fibrant replacement of $A$} is an acyclic cofibration $A \acof \fibs{A}$ such that $\fibs{A}$ is fibrant.
	A \emph{cofibrant replacement of $X$} is an acyclic fibration $\cofs{X} \afib X$ such that $\cofs{X}$ is cofibrant.
\end{dfn}

\noindent
The structure of weak model category is sufficient for a ``nice'' construction of the homotopy category, as follows. 
If $\Mw = (\cat{C}, \Cof, \Fib)$ is a weak model category, say an object is \emph{bifibrant} if it is both fibrant and cofibrant, and let
\begin{itemize}
	\item $\cufs{\Mw}$ denote the full subcategory of $\cat{C}$ on objects that are either cofibrant or fibrant,
	\item $\cofs{\Mw}$ denote the full subcategory of $\cat{C}$ on cofibrant objects,
	\item $\fibs{\Mw}$ denote the full subcategory of $\cat{C}$ on fibrant objects,
	\item $\bifs{\Mw}$ denote the full subcategory of $\cat{C}$ on bifibrant objects;
\end{itemize}
there is an evident square of inclusions
\begin{equation} \label{eq:inclusions_of_cofs_fibs}
\begin{tikzcd}
	{\bifs{\Mw}} & {\fibs{\Mw}} \\
	{\cofs{\Mw}} & {\cufs{\Mw}}.
	\arrow[hook, from=1-1, to=1-2]
	\arrow[hook, from=1-1, to=2-1]
	\arrow[hook, from=1-2, to=2-2]
	\arrow[hook, from=2-1, to=2-2]
\end{tikzcd}
\end{equation}

\begin{lem} \label{lem:homotopy_relation}
	Let $\Mw$ be a weak model category and $f, g\colon A \to X$ be two morphisms with $A$ cofibrant and $X$ fibrant.
	The following are equivalent:
	\begin{enumerate}[label=(\alph*)]
		\item there exist a cylinder $(j^-, j^+)\colon A \pout{} A \incl IA$ and $h\colon IA \to X$ such that $f = hj^-$ and $g = hj^+$;
		\item for every cylinder $(j^-, j^+)\colon A \pout{} A \incl IA$, there exists $h\colon IA \to X$ such that $f = hj^-$ and $g = hj^+$;
		\item there exist a path object $(p^-, p^+)\colon PX \surj X \times X$ and $h\colon A \to PX$ such that $f = p^-h$ and $g = p^+h$;
		\item for every path object $(p^-, p^+)\colon PX \surj X \times X$, there exists $h\colon A \to PX$ such that $f = p^-h$ and $g = p^+h$.
	\end{enumerate}
\end{lem}
\begin{proof}
	See \cite[Proposition 2.1.16]{henry2020weak}.
\end{proof}

\begin{dfn}[Homotopy relation]
	Let $\Mw$ be a weak model category and $f, g\colon A \to X$ be two morphisms with $A$ cofibrant and $X$ fibrant.
	We say that $f$, $g$ are \emph{homotopic}, and write $f \homo g$, if any of the equivalent conditions of Lemma \ref{lem:homotopy_relation} holds.
\end{dfn}

\begin{lem} \label{lem:homotopy_relation_is_equivalence}
	Let $\Mw$ be a weak model category.
	Then the homotopy relation $\homo$ on morphisms with cofibrant domain and fibrant codomain is an equivalence relation.
\end{lem}
\begin{proof}
	See \cite[Proposition 2.1.17]{henry2020weak}.
\end{proof}

\begin{dfn}[Homotopy category]
	Let $\Mw$ be a weak model category.
	The \emph{homotopy category} of $\Mw$ is the category $\Ho(\Mw)$ whose objects are the bifibrant objects in $\Mw$, and morphisms are equivalence classes of morphisms under the homotopy relation.
\end{dfn}

\begin{prop} \label{prop:characterisation_of_homotopy_category}
	Let $\Mw$ be a weak model category.
	The following categories all exist and are equivalent:
	\begin{enumerate}[label=(\alph*)]
		\item the homotopy category of $\Mw$;
		\item the localisation of $\bifs{\Mw}$ at the acyclic fibrations;
		\item the localisation of $\bifs{\Mw}$ at the acyclic cofibrations;
		\item the localisation of $\fibs{\Mw}$ at the acyclic fibrations;
		\item the localisation of $\cofs{\Mw}$ at the acyclic cofibrations;
		\item the localisation of $\cufs{\Mw}$ at acyclic cofibrations and acyclic fibrations.
	\end{enumerate}
	The equivalences are induced by the square of inclusions (\ref{eq:inclusions_of_cofs_fibs}).
\end{prop}
\begin{proof}
	See \cite[Theorem 2.2.6]{henry2020weak}.
\end{proof}

\noindent
Proposition \ref{prop:characterisation_of_homotopy_category} allows us to treat all of these as alternative presentations of $\Ho(\Mw)$.

\begin{dfn}[Equivalence in a weak model category]
	Let $\Mw$ be a weak model category.
	A morphism in $\cufs{\Mw}$ is an \emph{equivalence} if it is invertible in $\Ho(\Mw)$.
\end{dfn}

\begin{rmk}
	In particular, by Proposition \ref{prop:characterisation_of_homotopy_category}, a morphism $f\colon X \to Y$ between bifibrant objects is an equivalence if and only if it is a \emph{homotopy equivalence}, that is, there exists $g\colon Y \to X$ such that $gf \homo \idd{X}$ and $fg \homo \idd{Y}$.
\end{rmk}

\noindent
The following property of model categories still holds in weak model categories.
\begin{lem} \label{lem:acyclic_cofib_is_cofib_equivalence}
	Let $\Mw$ be a weak model category.
	Then
	\begin{enumerate}
		\item a cofibration is an acyclic cofibration if and only if it is an equivalence,
		\item a fibration is an acyclic fibration if and only if it is an equivalence.
	\end{enumerate}
\end{lem}
\begin{proof}
	See \cite[Proposition 2.2.10]{henry2020weak}; note that, since we assumed that domains of cofibrations are cofibrant, every cofibration is a cofibration between cofibrants, and dually for fibrations.
\end{proof}

\noindent
We also still have the following ``cube lemma''.

\begin{lem} \label{lem:cube_lemma}
	Let $\Mw$ be a weak model category and consider a commutative diagram in $\cofs{\Mw}$
\[\begin{tikzcd}
	B & A & C \\
	{B'} & {A'} & {C'}
	\arrow["\wr", from=1-1, to=2-1]
	\arrow[hook', from=1-2, to=1-1]
	\arrow[hook, from=1-2, to=1-3]
	\arrow["\wr", from=1-2, to=2-2]
	\arrow["\wr", from=1-3, to=2-3]
	\arrow[hook', from=2-2, to=2-1]
	\arrow[hook, from=2-2, to=2-3]
\end{tikzcd}\]
	where the horizontal morphisms are cofibrations and the vertical morphisms are equivalences.
	Then the morphism $B \pout{A} C \to B' \pout{A'} C'$ universally determined between pushouts is an equivalence.
\end{lem}
\begin{proof}
	The proof of \cite[Proposition 2.2.12]{henry2016algebraic}, specialised to the case where all horizontal morphisms are cofibrations, so there is no need for the factorisation step, only uses that cofibrations are stable under pushout, that equivalences satisfy 2-out-of-3, and that equivalences of cofibrant objects are stable under pushout along cofibrations, which holds by \cite[Corollary 2.4.4]{henry2020weak}.
\end{proof}

\noindent
As in non-weak model categories, there are notions of \emph{Quillen adjunction} and of \emph{Quillen equivalence} of weak model categories.

\begin{dfn}[Quillen adjunction of weak model categories]
	Let $\Mw$ and $\Nw$ be weak model categories.
	A \emph{Quillen adjunction} between $\Mw$ and $\Nw$ is a pair of functors $\fun{L}\colon \cofs{\Mw} \to \cofs{\Nw}$ and $\fun{R}\colon \fibs{\Nw} \to \fibs{\Mw}$ such that
	\begin{enumerate}
		\item there are isomorphisms $\Hom_{\cufs{\Nw}}(\fun{L}A, X) \simeq \Hom_{\cufs{\Mw}}(A, \fun{R}X)$ natural in $A \in \Ob(\cofs{\Mw})$ and $X \in \Ob(\fibs{\Nw})$,
		\item $\fun{L}$ sends cofibrations to cofibrations,
		\item $\fun{R}$ sends fibrations to fibrations.
	\end{enumerate}
	We say that a functor $\fun{L}\colon \cofs{\Mw} \to \cofs{\Nw}$ is \emph{left Quillen}, and that a functor $\fun{R}\colon \fibs{\Nw} \to \fibs{\Mw}$ is \emph{right Quillen}, when they are part of a Quillen adjunction.
\end{dfn}

\begin{prop} \label{prop:quillen_adjunction_descends_to_homotopy_category}
	Let $(\fun{L}, \fun{R})$ be a Quillen adjunction between weak model categories $\Mw$ and $\Nw$.
	Then $\fun{L}$ and $\fun{R}$ induce an adjunction $\Ho(\fun{L}) \dashv \Ho(\fun{R})$ between $\Ho(\Mw)$ and $\Ho(\Nw)$.
\end{prop}
\begin{proof}
	See \cite[Proposition 2.4.3]{henry2020weak}.
\end{proof}

\begin{dfn}[Quillen equivalence of weak model categories]
	A Quillen adjunction between weak model categories is a \emph{Quillen equivalence} if it determines an adjoint equivalence between their homotopy categories.
\end{dfn}

\noindent
We have the following criterion for deciding that a Quillen adjunction is a Quillen equivalence.

\begin{prop} \label{prop:quillen_equivalence_criterion}
	Let $(\fun{L}, \fun{R})$ be a Quillen adjunction between weak model categories $\Mw$ and $\Nw$.
	The following are equivalent:
	\begin{enumerate}[label=(\alph*)]
		\item for each $A \in \Ob \cofs{\Mw}$ there is a fibrant replacement $\fun{L}A \acof \fibs{(\fun{L}A)}$ such that the transpose morphism $A \to \fun{R}\fibs{(\fun{L}A)}$ is an equivalence in $\cufs{\Mw}$, and for each $X \in \Ob \fibs{\Nw}$ there is a cofibrant replacement $\cofs{(\fun{R}X)} \afib \fun{R}X$ such that the transpose morphism $\fun{L}\cofs{(\fun{R}X)} \to X$	is an equivalence in $\cufs{\Nw}$;
		\item $(\fun{L}, \fun{R})$ is a Quillen equivalence between $\Mw$ and $\Nw$.
	\end{enumerate}
\end{prop}
\begin{proof}
	See \cite[Proposition 2.4.5]{henry2020weak}.
\end{proof}

\begin{rmk}
	By a standard argument, if the first condition in Proposition \ref{prop:quillen_equivalence_criterion} holds, then it holds for all fibrant and cofibrant replacements.
\end{rmk}

\noindent 
To construct weak model structures, we will use the ``Cisinski--Olschok'' methods from \cite[Section 3]{henry2020weak}.

\begin{dfn}[Pushout-product]
	Let $\cat{C}$ be a category with pushouts, let $\fun{F}, \fun{G}$ be two endofunctors of $\cat{C}$, and let $\beta\colon \fun{F} \Rightarrow \fun{G}$ be a natural transformation.
	Given a morphism $f\colon C \to D$ in $\cat{C}$, the \emph{pushout-product of $\beta$ and $f$} is the morphism $\beta \ppnat f$ obtained universally in the diagram
\[\begin{tikzcd}
	{\fun{F}X} & {\fun{G}X} \\
	{\fun{F}Y} & {\fun{F}Y \pout{\fun{F}X} \fun{G}X} & {\fun{G}Y}.
	\arrow["{\beta_X}", from=1-1, to=1-2]
	\arrow["{\fun{F}f}", from=1-1, to=2-1]
	\arrow[from=1-2, to=2-2]
	\arrow["{\fun{G}f}", curve={height=-12pt}, from=1-2, to=2-3]
	\arrow[from=2-1, to=2-2]
	\arrow["{\beta_Y}", curve={height=30pt}, from=2-1, to=2-3]
	\arrow["\lrcorner"{anchor=center, pos=0.125, rotate=180}, draw=none, from=2-2, to=1-1]
	\arrow["{\beta \ppnat f}", dashed, from=2-2, to=2-3]
\end{tikzcd}\]
\end{dfn}

\noindent 
In a category $\cat{C}$ with product and coproducts, we let $\rcodiag{\cat{C}}$ denote the codiagonal endofunctor $A \mapsto A \pout{} A$ and $\rdiag{\cat{C}}$ the diagonal endofunctor $X \mapsto X \times X$.

\begin{dfn}[Functorial cylinder]
	Let $\cat{C}$ be a category with products and coproducts.
	A \emph{functorial cylinder on $\cat{C}$} is a left adjoint endofunctor $\fun{I}$ together with a pair of natural transformations $\iota^-, \iota^+\colon \Idd{\cat{C}} \to \fun{I}$.
	We let $(\iota^-, \iota^+)\colon \rcodiag{\cat{C}} \to \fun{I}$ denote the induced natural transformation with components $(\iota^-_A, \iota^+_A)\colon A \pout{} A \to \fun{I}A$.
\end{dfn}

\noindent
A functorial cylinder determines a \emph{functorial path object} $(\fun{P}, (\pi^-, \pi^+))$, where $\fun{P}$ is the right adjoint of $\fun{I}$, and $(\pi^-, \pi^+)\colon \fun{P} \to \rdiag{\cat{C}}$ is the transpose of $(\iota^-, \iota^+)$.

\begin{dfn}[$I$-fibrations and $I$-cofibrations]
	Let $\cat{C}$ be a category with an initial object $\init$ and a terminal object $\term$, let $I$ be a set of morphisms in $\cat{C}$, and let $f\colon A \to X$ be a morphism in $\cat{C}$.
	We say that
	\begin{itemize}
		\item $X$ is \emph{$I$-fibrant} if the unique morphism $X \to \term$ is in $\rlp(I)$,
		\item $f$ is an \emph{$I$-fibration} if $X$ is $I$-fibrant and $f \in \rlp(I)$,
		\item $A$ is \emph{$I$-cofibrant} if the unique morphism $\init \to A$ is in $\llp(\rlp(I))$,
		\item $f$ is an \emph{$I$-cofibration} if $A$ is $I$-cofibrant and $f \in \llp(\rlp(I))$.
	\end{itemize}
	We write $\JFib{I}$ for the class of $I$-fibrations, and $\ICof{I}$ for the class of $I$-cofibrations.
\end{dfn}

\begin{prop} \label{prop:weak_cisinski_olschok}
	Let $\cat{C}$ be a locally presentable category, let $(\fun{I}, (\iota^-, \iota^+))$ be a functorial cylinder on $\cat{C}$ with dual functorial path object $(\fun{P}, (\pi^-, \pi^+))$, let $I$, $J$ be sets of morphisms in $\cat{C}$, and suppose that
	\begin{enumerate}
		\item $J \subseteq \ICof{I}$,
		\item for all $i \in I$, the pushout-product $(\iota^-, \iota^+) \ppnat i$ is in $\llp(\rlp(I))$,
		\item for all $i \in I$ and $\a \in \set{-, +}$, the pushout-product $\iota^\a \ppnat i$ is in $\llp(\JFib{J})$,
		\item for all $j \in J$, the pushout-product $(\iota^-, \iota^+) \ppnat j$ is in $\llp(\JFib{J})$.
	\end{enumerate}
	Then $(\cat{C}, \ICof{I}, \JFib{J})$ is a weak model category.
	Moreover, for each cofibrant object $A$ and fibrant object $X$,
	\begin{enumerate}
		\item $(\iota^-_A, \iota^+_A)\colon A \pout{} A \incl \fun{I}A$ is a cylinder for $A$,
		\item $(\pi^-_X, \pi^+_X)\colon \fun{P}X \surj X \times X$ is a path object for $X$.
	\end{enumerate}
\end{prop}
\begin{proof}
	A special case of \cite[Theorem 3.0.5]{henry2020weak}, where we assume that $\cat{C}$ is locally presentable to ensure that the small object argument applies.
\end{proof}

\noindent
In the conditions of Proposition \ref{prop:weak_cisinski_olschok}, we say that a morphism in $I$ is a \emph{generating cofibration}, and that a morphism in $J$ is a \emph{generating anodyne extension}.

\subsection{Directed complexes} \label{sec:directed}

\noindent 
Let $\cat{C}$ be a category and let $\Xi$ be a class of colimit cones in $\cat{C}$.
A presheaf on $\cat{C}$ is said to be \emph{$\Xi$-continuous} if it sends colimit cones in $\Xi$ to limit cones in $\Set$.
The full subcategory $\PSh_\Xi(\cat{C})$ of the category of presheaves on the $\Xi$-continuous presheaves enjoys many good properties---for instance, it is always a reflective subcategory; the classical reference is \cite{freyd1972categories}.

\begin{dfn}[Directed complex]
	A \emph{directed complex} is a $\Gamma$-continuous presheaf on the category $\frdCpx_{\L}$ of finite regular directed complexes and local embeddings.
\end{dfn}

\begin{comm}
	Note that this is not the same notion as Steiner's notion of directed complex \cite{steiner1993algebra}, which is the combinatorial forerunner of our regular directed complexes.
\end{comm}

\noindent
We let $\dCpx$ denote the category $\PSh_\Gamma(\rdCpx_\L)$ of directed complexes.
Because representable presheaves are continuous with respect to all colimits, the Yoneda embedding factors through
\[
	\frdCpx_{\L} \incl \dCpx
\]
and we will identify each finite regular directed complex with its image through this embedding.
This presentation of $\dCpx$ is particulary convenient for obtaining this identification, but there is a simpler, equivalent one.
Let $\atom_\E$ be a skeleton of the full subcategory of $\frdCpx_\L$ on the atoms.

\begin{prop} \label{prop:guises_of_directed_complexes}
	The following categories are equivalent:
	\begin{enumerate}[label=(\alph*)]
		\item the category of directed complexes;
		\item the category of presheaves on $\atom_\E$.
	\end{enumerate}
\end{prop}
\begin{proof}
	There is an evident restriction functor $\dCpx \to \PSh(\atom_\E)$.
	Every regular directed complex $P$ has a presentation as the colimit of the diagram $\fun{F}_P\colon P \to \rdCpx_\L$ sending $x \in P$ to the atom $\clset{x}$, which has image in $\atom_\E$.
	When $P$ is finite, this colimit can be constructed with a sequence of coproducts and pushouts of embeddings, which are $\Gamma$\nbd colimits.
	This implies that the restriction functor is faithful.
	Moreover, a local embedding $f\colon P \to Q$ is uniquely determined by the restrictions $\restr{f}{\clset{x}}\colon \clset{x} \to \clset{f(x)}$, which are components of a natural transformation from $\fun{F}_P$ to $\fun{F}_Q f$; this implies that the restriction functor is full.
	Finally, there are no non-trivial $\Gamma$-colimits in $\atom_\E$, so its essential image is the entire category of presheaves.
\end{proof}

\noindent
Directed complexes are essentially ``diagrammatic sets without degeneracies''; the fact that $\frdCpx_{\L}$ embeds into them can be compared with \cite[Lemma 2.5]{chanavat2024diagrammatic}.
We can thus import all the terminology from the theory of diagrammatic sets that does not refer to degeneracies.

\begin{dfn}[Diagram in a directed complex]
	Let $U$ be a finite regular directed complex and $X$ a directed complex.
	A \emph{diagram of shape $U$ in $X$} is a morphism $u\colon U \to X$.
	A diagram is a \emph{pasting diagram} if $U$ is a molecule, a \emph{round diagram} if $U$ is round, and a \emph{cell} if $U$ is an atom.
\end{dfn}

\noindent
Given a diagram $u\colon U \to X$ and a morphism $f\colon X \to Y$ of directed complexes, we will sometimes write $f(u) \eqdef fu\colon U \to Y$.
We let $\cell X$ denote the set of cells in a directed complex, which is the same as the set of its elements as a presheaf on $\atom_\E$.
This set is graded by dimension.
By entirely standard arguments, based on the properties of $\atom_\E$, every directed complex $X$ is the colimit of the sequence
\[
	\varnothing \equiv \gr{\le -1}{X} \incl \gr{\le 0}{X} \incl \dots \incl \gr{\le n}{X} \incl \dots,
\]
of inclusions of its restrictions $\gr{\le n}{X}$ to cells of dimension $\le n$, and, furthermore, each inclusion fits into a pushout diagram
\[\begin{tikzcd}
	{\coprod_{u \in \gr{n}{\cell X}} \bd{}{}U_u} & {\coprod_{u \in \gr{n}{\cell X}} U_u} \\
	{\gr{\leq n-1}{X}} & {\gr{\leq n}{X},}
	\arrow[hook, from=1-1, to=1-2]
	\arrow["{(\bd{}{}u)_{u \in \gr{n}{\cell X}}}", from=1-1, to=2-1]
	\arrow["{(u)_{u \in \gr{n}{\cell X}}}", from=1-2, to=2-2]
	\arrow[hook, from=2-1, to=2-2]
	\arrow["\lrcorner"{anchor=center, pos=0.125, rotate=180}, draw=none, from=2-2, to=1-1]
\end{tikzcd}\]
where $U_u$ is the shape of the cell $u$.

What follows is mainly a recap of \cite[Section 1.1]{chanavat2024equivalences}, restricted from diagrammatic sets to directed complexes.
Recall that, for all $k \in \mathbb{N}$, an \emph{$\omega$\nbd graph} (or globular set) in degree $\geq k$ is a graded set $G = \coprod_{n \geq k} \gr{n}{G}$ together with boundary functions $\bd{}{-}, \bd{}{+}\colon \gr{n+1}{G} \to \gr{n}{G}$ for all $n \geq k$, satisfying $\bd{}{\a}\bd{}{-} = \bd{}{\a}\bd{}{+}$ for all $\a \in \set{-, +}$.
For all $n \geq m \geq k$, this relation allows us to define recursively $\bd{m}{\a}\colon \gr{n}{G} \to \gr{m}{G}$ by $\bd{m}{\a} \eqdef \idd{\gr{n}{G}}$ if $n = m$ and $\bd{m}{\a} \eqdef \bd{}{\a}\bd{m+1}{\a}$ if $n > m$.
If $G$ and $H$ are $\omega$\nbd graphs in degree $\geq k$, then a morphism $f\colon G \to H$ is a grade-preserving, boundary-preserving function.

Given $n > k$ and $a \in \gr{n}{G}$, we write $a\colon a^- \celto a^+$ if $\bd{}{\a}a = a^\a$ for each $\a \in \set{-, +}$, and say that $a$ is of \emph{type} $a^- \celto a^+$.
For all $n \geq k$, we say that $a, b \in \gr{n}{G}$ are \emph{parallel} if either $n = 0$, or $n > 0$ and $a$ and $b$ have the same type.
Given parallel $a, b \in \gr{n}{G}$, the graded set
\[
	G(a, b) \eqdef \set{ c \in \gr{> n}{G} \mid \bd{n}{-}c = a, \bd{n}{+}c = b }
\]
inherits by restriction a structure of $\omega$\nbd graph in degree $> n$.

\begin{dfn}[The $\omega$-graph of pasting diagrams]
	Let $u\colon U \to X$ be a pasting diagram in a directed complex, $n \in \mathbb{N}$, and $\a \in \set{-, +}$.
	We let $\bd{n}{\a}u \eqdef \restr{u}{\bd{n}{\a}{U}}\colon \bd{n}{\a}U \to X$; we may omit the index $n$ when $n = \dim{u} - 1$.
	We let $\Pd X$ denote the set of pasting diagrams in $X$ and $\Rd X \subseteq \Pd X$ its subset of round diagrams.
	The set $\Pd X$ is graded by dimension; given a subset $A$ of $\Pd X$ and $n \in \mathbb{N}$, we let $\gr{n}{A} \eqdef \set{u \in A \mid \dim u = n}$.
	Then, $\Pd X$ admits the structure of an $\omega$\nbd graph with the functions $\bd{}{-}, \bd{}{+}\colon \gr{n+1}{\Pd X} \to \gr{n}{\Pd X}$ for each $n \in \mathbb{N}$.
	These restrict along the inclusions $\gr{n}{\Rd X} \subseteq \gr{n}{\Pd X}$, making $\Rd X$ an $\omega$\nbd subgraph of $\Pd X$.
\end{dfn}

\noindent
When specifying the type of a pasting diagram, we will distinguish notationally between cells and other pasting diagrams, by writing $u\colon a \celto b$ for cells and $u\colon a \rdto b$ for more general pasting diagrams.

\begin{comm}
	The $+$ in $\rdto$ should be read as the Kleene $+$ in formal language theory, rather than a reference to orientation.
	It is preferrable to the Kleene $*$ because the relation ``there exists a pasting diagram of type $a \rdto b$'' is transitive but not, in general, reflexive.
\end{comm}

\begin{dfn}[Subdiagram]
	Let $u\colon U \to X$ be a pasting diagram.
	A \emph{subdiagram of $u$} is a pair of a pasting diagram $v\colon V \to X$ and a submolecule inclusion $\iota\colon V \submol U$ such that $v = u\iota$.
	A subdiagram is \emph{rewritable} when $\iota$ is a rewritable submolecule inclusion.
	We write $\iota\colon v \submol u$ for the data of a subdiagram of $u$.
\end{dfn}

\noindent We will simply write $v \submol u$ when $\iota$ is irrelevant or evident from the context.

\begin{dfn}[Pasting of pasting diagrams]
	Let $u\colon U \to X$ and $v\colon V \to X$ be pasting diagrams such that $\bd{k}{+}u = \bd{k}{-}v$.
	We let $u \cp{k} v\colon U \cp{k} V \to X$ be the pasting diagram determined by the universal property of the pasting $U \cp{k} V$.
	More generally, if $\iota\colon \bd{k}{+}u \submol \bd{k}{-}v$, we let $u \cpsub{k,\iota} v\colon U \cpsub{k,\iota} V \to X$ be the pasting diagram determined by the universal property of $U \cpsub{k,\iota} V$ as a pasting of $U$ at a submolecule of $\bd{k}{-} V$.
	Dually, if $\iota\colon \bd{k}{-}v \submol \bd{k}{+}u$, we let $u \subcp{k,\iota} v$ be the universally determined pasting diagram of shape $U \subcp{k,\iota} V$.
\end{dfn}

\begin{rmk}
	There are evident subdiagrams $u, v \submol u \cpsub{k,\iota} v$ and $u, v \submol u \subcp{k,\iota} v$ whenever the pastings are defined.
\end{rmk}

\begin{rmk}
	All pastings are evidently preserved by morphisms $f\colon X \to Y$ of directed complexes, in the sense that $f(u \cp{k} v) = f(u) \cp{k} f(v)$.
\end{rmk}

\noindent We may omit the index $k$ when it is equal to $\min \set{\dim{u}, \dim{v}} - 1$, and omit $\iota$ when it is irrelevant or evident from the context.
It follows from \cite[Chapter 5]{hadzihasanovic2024combinatorics} that pasting satisfies all the axioms of composition in strict $\omega$\nbd categories.
In particular, pastings of the form $u \cp{} v$ suffice to generate all pastings of the form $u \cp{k} v$, as well as pastings at a subdiagram $u \cpsub{k, \iota} v$, for all $k \in \mathbb{N}$.

\begin{dfn}[Substitution at a rewritable subdiagram]
	Let $u\colon U \to X$ be a pasting diagram, let $\iota\colon v \submol u$ be a rewritable subdiagram of shape $V$, and let $w$ be a round diagram of shape $W$, parallel to $v$.
	The \emph{substitution of $w$ for $\iota\colon v \submol u$} is the unique pasting diagram $\subs{u}{w}{\iota(v)}$ of shape $\subs{U}{W}{\iota(V)}$ which restricts to $w$ along $W \incl \subs{U}{W}{\iota(V)}$ and to $\restr{u}{U \setminus \inter{\iota(V)}}$ along $U \setminus \inter{\iota(V)} \incl \subs{U}{W}{\iota(V)}$.
\end{dfn}

\noindent 
From \cite[Section 3.1]{chanavat2024equivalences} and \cite[Section 3.1]{chanavat2025gray}, we recall the notion of an $A$\nbd context for pasting diagrams.

\begin{dfn}[$A$-context]
	Let $X$ be a directed complex, $A \subseteq \Rd X$.
	For $k$ ranging over $\mathbb{N}$ and $v, w$ over parallel pairs in $\gr{k}{\Rd X}$, the class of \emph{$A$-contexts on $\Pd X(v, w)$} is the inductive class of morphisms of $\omega$\nbd graphs in degree $> k$ with domain $\Pd X(v, w)$ generated by the following clauses.
	\begin{enumerate}
		\item (\textit{Left pasting}). 
			For all $u \in \gr{k+1}{A}$ and rewritable $\iota\colon \bd{}{+}u \submol v$,
			\[
				u \cpsub{\iota} -\colon \Pd X(v, w) \to \Pd X(\subs{v}{\bd{}{-}u}{\iota(\bd{}{+}u)}, w)
			\]
			is an $A$\nbd context on $\Pd X(v, w)$.
		\item (\textit{Right pasting}).
			For all $u \in \gr{k+1}{A}$ and rewritable $\iota\colon \bd{}{-}u \submol w$,
			\[
				- \subcp{\iota} u\colon \Pd X(v, w) \to \Pd X(v, \subs{w}{\bd{}{+}u}{\iota(\bd{}{-}u)})
			\]
			is an $A$\nbd context on $\Pd X(v, w)$.
		\item (\textit{Identity}).
			The identity $-\colon \Pd X(v, w) \to \Pd X(v, w)$ is an $A$\nbd context on $\Pd X(v, w)$.
		\item (\textit{Composition}).
			If $\fun{F}\colon \Pd X(v, w) \to \Pd X(v', w')$ is an $A$\nbd context and $\fun{G}$ is an $A$\nbd context on $\Pd X(v', w')$, then $\fun{GF}$ is an $A$\nbd context on $\Pd X(v, w)$.
		\item (\textit{Promotion}).
			If $k > 0$ and $\fun{F}$ is an $A$\nbd context on $\Pd X(\bd{}{-}v, \bd{}{+}w)$, then
			\[
				\fun{F}_{v, w} \eqdef \restr{\fun{F}}{\Pd X(v, w)}\colon \Pd X(v, w) \to \Pd X(\fun{F}v, \fun{F}w)
			\]
			is an $A$\nbd context on $\Pd X(v, w)$.
	\end{enumerate}
	We let $\dim \fun{F} \eqdef k+1$ be the \emph{dimension} of any $A$\nbd context $\fun{F}$ on $\Pd X(v, w)$.
	When $A = \Rd X$ or $A$ is irrelevant, we speak simply of a \emph{context} on $\Pd X(v, w)$.
\end{dfn}

\begin{rmk}
	Given a context $\fun{F}$ on $\Pd X(v, w)$ and a cell $a\colon v \celto w$, there is an evident subdiagram $\iota_{\fun{F}}\colon a \submol \fun{F}a$.
	Conversely, given a rewritable subdiagram $\iota\colon a \submol b$ where $a\colon v \rdto w$, there exists a unique context $\fun{F}_\iota$ on $\Pd(v, w)$ such that $b = \fun{F}_\iota a$.
	We call this the context \emph{determined by $\iota\colon a \submol b$}.
\end{rmk}

\begin{lem} \label{lem:context_layering}
	Let $X$ be a directed complex and $\fun{F}$ a context on $\Pd X(v,w)$ with $k \eqdef \dim{\fun{F}}$.
	Then there exist pasting diagrams $(\ell_i, r_i)_{i=1}^{k}$ in $X$ such that
	\begin{enumerate}
		\item $\fun{F} = \ell_k \cp{k-1} (\ell_{k-1} \cp{k-2} \ldots (\ell_1 \cp{0} - \cp{0} r_1) \ldots \cp{k-2} r_{k-1}) \cp{k-1} r_k$,
		\item $\dim{\ell_i}, \dim{r_i} \leq i$ for all $i \in \set{1, \ldots, k}$.
	\end{enumerate}
\end{lem}
\begin{proof}
	Same as \cite[Lemma 3.3]{chanavat2024equivalences}.
\end{proof}

\begin{dfn}[Shape of a context]
	Let $X$ be a directed complex, $v\colon V \to X$ and $w\colon W \to X$ be parallel round diagrams, and $\fun{F}$ be a context on $\Pd X(v, w)$ with $k \eqdef \dim{\fun{F}}$.
	Let $(\ell_i\colon L_i \to X, r_i\colon R_i \to X)_{i=1}^k$ be sequences of pasting diagrams provided for $\fun{F}$ by Lemma \ref{lem:context_layering}.
	The \emph{shape of $\fun{F}$} is the molecule
	\[
		L_k \cp{k-1} (L_{k-1} \cp{k-2} \ldots (L_1 \cp{0} (V \celto W) \cp{0} R_1) \ldots \cp{k-2} R_{k-1}) \cp{k-1} R_k.
	\]
	We say that $\fun{F}$ is \emph{round} if its shape is round.
\end{dfn}

\begin{rmk}
	When there exists a cell $a\colon v \celto w$, then the shape of a context $\fun{F}$ on $\Pd(v, w)$ is precisely the shape of $\fun{F}a$.
\end{rmk}

\begin{dfn}[Context subdiagram]
	Let $X$ be a directed complex, $v, w \in \Pd X$ be parallel, and let $\fun{F}$ be a context on $\Pd X(v, w)$.
	A \emph{context subdiagram} $\iota\colon z \submol \fun{F}$ is a pair of 
	\begin{enumerate}
		\item a decomposition $\fun{F} = v' \cp{} {\fun{F}'-}$ or $\fun{F} = {\fun{F}'-} \cp{} v'$, and
		\item a subdiagram $\iota\colon z \submol v'$.
	\end{enumerate}
	A context subdiagram is \emph{rewritable} if $\dim{v'} = \dim \fun{F}$ and $\iota$ is rewritable.
\end{dfn}

\noindent
By \cite[Lemma 3.18]{chanavat2024equivalences}, given a context $\fun{F}$ on $\Pd(v, w)$ and a context subdiagram $\iota\colon z \submol \fun{F}$, every $a\colon v \rdto w$ determines a subdiagram $\iota_a\colon z \submol \fun{F}a$, which is rewritable if $\iota$ is rewritable.

Since the Gray product restricts to a monoidal structure $(\atom_\E, \gray, 1)$, we obtain a biclosed monoidal structure on $\dCpx$ by the Day convolution construction on categories of presheaves \cite{day1970closed}.
Both because we want to show that the embedding of $(\frdCpx_\L, \gray, 1)$ is strong monoidal, and because it will come useful later, we will use the following, more general version.

\begin{lem} \label{lem:day_for_gamma_continuous_presheaves}
	Let $(\cat{C}, \otimes, I)$ be a monoidal category and let $\Xi$ be a class of colimit cones in $\cat{C}$ such that, for all $C \in \Ob \cat{C}$, the functors $C \otimes -$ and $- \otimes C$ preserve $\Xi$-colimits.
	Then there exists an essentially unique monoidal structure $(\PSh_\Xi(\cat{C}), \otimes, I)$ such that
	\begin{enumerate}
		\item for all $\Xi$-continuous presheaves $X$, the functors $X \otimes -$ and $- \otimes X$ preserve all small colimits,
		\item the embedding $(\cat{C}, \otimes, I) \incl (\PSh_\Xi(\cat{C}), \otimes, I)$ is strong monoidal.
	\end{enumerate}
	This monoidal structure is biclosed, and for all $\Xi$\nbd continuous presheaves $X$, $Y$, the monoidal product $X \otimes Y$ can be constructed as
	\[
		\fun{r}_\Xi \left(\int^{C, C' \in \Ob \cat{C}} \Hom_{\cat{C}}(-, C \otimes C') \times X(C) \times Y(C') \right),
	\]
	where $\fun{r}_\Xi\colon \PSh(\cat{C}) \to \PSh_\Xi(\cat{C})$ is the reflector.
	Furthermore, let $(\cat{D}, \otimes, I)$ be another monoidal category with a class $\Lambda$ of colimit cones satisfying the same properties, and let $\fun{F}\colon \cat{C} \to \cat{D}$ be a functor such that
	\begin{enumerate}
		\item $\fun{F}\colon (\cat{C}, \otimes, I) \to (\cat{D}, \otimes, I)$ is strong monoidal,
		\item the left Kan extension $\Lan \fun{F}\colon \PSh(\cat{C}) \to \PSh(\cat{D})$ commutes with reflectors, that is, $(\Lan \fun{F}) \fun{r}_\Xi$ is naturally isomorphic to $\fun{r}_\Lambda (\Lan \fun{F})$.
	\end{enumerate}
	Then $\Lan \fun{F}\colon (\PSh_\Xi(\cat{C}), \otimes, I) \to (\PSh_{\Lambda}(\cat{D}), \otimes, I)$ is strong monoidal.
\end{lem}
\begin{proof}
	Let $C \in \Ob (\cat{C})$ and let $X$ be a $\Xi$-continuous presheaf.
	Since $C \otimes -$ and $- \otimes C$ preserve $\Xi$-colimits, the presheaves
	\begin{align*}
		H(C, X)\colon D & \mapsto \Hom_{\PSh_{\Xi}(\cat{C})}(C \otimes D, X), \\
		H'(C, X)\colon D & \mapsto \Hom_{\PSh_{\Xi}(\cat{C})}(D \otimes C, X)
	\end{align*}
	are $\Xi$-continuous, and the assignment is evidently functorial both contravariantly in $C$ and covariantly in $X$, so it determines functors
	\[
		H, H'\colon \opp{\cat{C}} \times \PSh_\Xi(\cat{C}) \to \PSh_\Xi(\cat{C})
	\]
	with natural isomorphisms
	\[
		\Hom(D, H(C, X)) \simeq 
		\Hom(C \otimes D, X) \simeq
		\Hom(C, H'(D, X)).
	\]
	Since $\PSh_\Xi(\cat{C})$ has all small colimits and limits as a reflective subcategory of a presheaf category, and $\cat{C}$ is a dense subcategory, the hypotheses of \cite[Theorem 5.3]{ara2020joint}, attributed to Day \cite{day1970closed} via Street \cite{street2004categorical}, are all met, and guarantee the existence and essential uniqueness of the biclosed monoidal structure $(\PSh_{\Xi}(\cat{C}), \otimes, I)$.
	Moreover, given $\Xi$-continuous presheaves $X$, $Y$, their monoidal product $X \otimes Y$ is given explicitly by
	\[
		\colim_{C \to X}\, \colim_{C' \to Y}\, C \otimes C',
	\]
	which in the case of a reflective subcategory of a presheaf category can be computed in $\PSh(\cat{C})$ by the stated coend, then reflected onto $\PSh_\Xi(\cat{C})$.
	Now, $(\cat{D}, \otimes, I)$ with the class $\Lambda$ satisfies the same hypotheses, so we have an essentially unique biclosed monoidal structure on $\PSh_{\Lambda}(\cat{D})$ extending the one on $\cat{D}$.
	Moreover, $\Lan \fun{F}$ sends a presheaf $X$ on $\cat{C}$ to the presheaf $\Lan \fun{F} X$ on $\cat{D}$ computed by the coend
	\[
	\int^{C \in \Ob(\cat{C})} \Hom_{\cat{D}}(-, \fun{F}C) \times X(C).
	\]
	Under the assumption that $(\Lan \fun{F}) \fun{r}_\Xi$ is naturally isomorphic to $\fun{r}_\Lambda (\Lan \fun{F})$, we then have, given $\Xi$\nbd continuous presheaves $X$, $Y$,
	\begin{align*}
		\Lan \fun{F}(X \otimes Y) & \simeq \fun{r}_\Lambda\left(\int^{C, C'} \Lan \fun{F}\left(\Hom_{\cat{C}}(-, C \otimes C')\right) \times X(C) \times Y(C')\right) \\
					  & \simeq \fun{r}_\Lambda\left(\int^{C, C'} \Hom_{\cat{D}}(-, \fun{F}(C \otimes C')) \times X(C) \times Y(C')\right) \\
					  & \simeq \fun{r}_\Lambda\left(\int^{C, C'} \Hom_{\cat{D}}(-, \fun{F}C \otimes \fun{F}C')) \times X(C) \times Y(C')\right)
	\end{align*}
	by strong monoidality of $\fun{F}$.
	Now, by strong monoidality of the embedding of $\cat{D}$, the representable presheaf on $\fun{F}C \otimes \fun{F}C'$ is naturally isomorphic to the monoidal product of the representables on $\fun{F}C$ and $\fun{F}C'$, so expanding its coend expression, we reduce the above to $\fun{r}_\Lambda$ applied to
	\[
		\int^{C,C',D,D'} \Hom_{\cat{D}}(-, D \otimes D') \times \Hom_{\cat{D}}(D, \fun{F}C) \times \Hom_{\cat{D}}(D', \fun{F}C') \times X(C) \times Y(C')
	\]
	which reduces to
	\[
		\fun{r}_\Lambda\left(\int^{D, D' \in \Ob(\cat{D})} \Hom_{\cat{D}}(-, D \otimes D') \times \Lan \fun{F}X(D) \times \Lan \fun{F}Y(D')\right)
	\]
	which is the expression for $\Lan {\fun{F}}X \otimes \Lan {\fun{F}}Y$.
	This completes the proof.
\end{proof}

\begin{rmk}
	The usual Day convolution is the case $\Xi = \varnothing$ of Lemma \ref{lem:day_for_gamma_continuous_presheaves}.
\end{rmk}

\begin{prop} \label{prop:gray_product_of_directed_complexes}
	There is an essentially unique biclosed monoidal structure $(\dCpx, \gray, 1)$ such that the embedding of $(\frdCpx_\L, \gray, 1)$ is strong monoidal.
\end{prop}
\begin{proof}
	Immediate from Proposition \ref{prop:gray_preservation_of_colimits} and Lemma 
	\ref{lem:day_for_gamma_continuous_presheaves}.
\end{proof}

\noindent
Adapting \cite[Lemma 3.4]{chanavat2024diagrammatic}, since all morphisms in $\atom_\E$ are monomorphisms, we find that the Gray product $X \gray Y$ of two directed complexes $X$ and $Y$ has a particularly simple form: its cells are, up to isomorphism, all of the form $u \gray v$ for some cell $u$ of $X$ and some cell $v$ of $Y$.

Next, let $\Gammacon$ be the subclass of connected $\Gamma$-colimits, that is, the class of pushouts of boundary inclusions of atoms along embeddings of finite regular directed complexes.

\begin{dfn}[Augmented directed complex]
	An \emph{augmented directed complex} is a $\Gammacon$\nbd continuous presheaf on the category $\frdCpx_\L$ of finite regular directed complexes and local embeddings.
\end{dfn}

\begin{comm}
	This is not to be confused with Steiner's notion of augmented directed chain complexes 
	\cite{steiner2004omega}, although the two can be related; see \cite[Chapter 11]{hadzihasanovic2024combinatorics}.
\end{comm}

\noindent
We let $\adCpx$ denote the category $\PSh_{\Gammacon}(\frdCpx_\L)$.

\begin{lem} \label{lem:join_of_augmented_directed_complexes}
	There is an essentially unique biclosed monoidal structure $(\adCpx, \join, \varnothing)$ such that the embedding of $(\frdCpx_\L, \join, \varnothing)$ is strong monoidal.
\end{lem}
\begin{proof}
	Immediate from Proposition \ref{prop:join_preservation_of_connected_colimits} and Lemma 
	\ref{lem:day_for_gamma_continuous_presheaves}.
\end{proof}

\noindent
By a variant of Proposition \ref{prop:guises_of_directed_complexes}, we may also present $\adCpx$ as the category of presheaves on a skeleton of the full subcategory of $\frdCpx_\L$ spanned by the atoms and $\varnothing$.
Thus, we have an extension-restriction adjunction
\[\begin{tikzcd}
	\dCpx && \adCpx,
	\arrow[""{name=0, anchor=center, inner sep=0}, "{\augm{-}}", curve={height=-12pt}, from=1-1, to=1-3]
	\arrow[""{name=1, anchor=center, inner sep=0}, "{\dimin{-}}", curve={height=-12pt}, from=1-3, to=1-1]
	\arrow["\dashv"{anchor=center, rotate=-90}, draw=none, from=0, to=1]
\end{tikzcd}\]
where the essential image of $\augm{-}$ consists of the presheaves $X$ such that $X(\varnothing)$ is a singleton.

\begin{dfn} [Join of directed complexes]
	Let $X$, $Y$ be directed complexes.
	The \emph{join of \( X \) and \( Y \)} is the directed complex \( X \join Y \eqdef \dimin{(\augm{X} \join \augm{Y})} \).
\end{dfn}

\begin{prop} \label{prop:join_of_directed_complexes}
	The triple \( (\dCpx, \join, \varnothing) \) is a monoidal structure on \( \dCpx \) such that
	\begin{enumerate}
		\item the embedding of \( (\frdCpx_\L, \join, \varnothing) \) is strong monoidal,
		\item \( \join \) preserves connected colimits in each variable.
	\end{enumerate}
\end{prop}
\begin{proof}
	Let $X$, $Y$ be directed complexes.
	By the explicit description of $\augm{X} \join \augm{Y}$ given by Lemma 
	\ref{lem:day_for_gamma_continuous_presheaves}, we have
	\begin{equation*}
		\augm{X} \join \augm{Y}(\varnothing) \simeq \augm{X}(\varnothing) \times \augm{Y}(\varnothing) \times \Hom(\varnothing, \varnothing)
	\end{equation*}
	which is a singleton because $\varnothing$ is a strict initial object.
	For the same reason, the essential image of the canonical embedding of $\frdCpx_\L$ into $\adCpx$ is included in the essential image of $\augm{-}$.
	This proves both that the monoidal structure \( (\adCpx, \join, \varnothing) \) restricts to \( (\dCpx, \join, \varnothing) \), and that the embedding \( (\frdCpx_\L, \join, \varnothing) \incl (\dCpx, \join, \varnothing) \) is strong monoidal.
	
	Next, since the monoidal unit \( \varnothing \) is initial in \( \dCpx \), we have canonical inclusions \( X \incl X \join Y \) and \( Y \incl X \join Y \).
	By \cite[\textsection 5.7]{ara2020joint}, to show that \( \join \) preserves connected colimits in each variable, it is enough to show that the functor \( \inr{Y} \colon \dCpx \to \slice{Y}{\dCpx} \) sending \( X \) to \( Y \incl X \join Y \) and the functor \( \inl{Y} \colon \dCpx \to \slice{Y}{\dCpx} \) sending \( X \) to \( Y \incl Y \join X \) both admit right adjoints.
	The construction of these right adjoints is a formal analogue of \cite[Proposition 16, Proposition 17]{hadzihasanovic2020combinatorial} using \cite[Proposition 7.5.29]{hadzihasanovic2024combinatorics} for the relation between duals and joins.
\end{proof}

\begin{lem} \label{lem:join_preserves_mono}
	Let \( i \), \( i' \) be monomorphisms of directed complexes.
	Then \( i \join i' \) is a monomorphism.
\end{lem}
\begin{proof}
	Since the functors \( \augm{-} \) and \( \dimin{-} \) evidently preserves monomorphisms, it is enough to show that the join of augmented directed complexes does.
	For this, the proof of \cite[Lemma 3.5]{chanavat2024model} goes through with \( {\join} \) in place of \( \gray \), using Proposition \ref{prop:join_of_directed_complexes} and \cite[Lemma 7.4.10]{hadzihasanovic2024combinatorics}.
\end{proof}

\noindent
If $X$, $Y$ are directed complexes, \( u \) is a cell of \( X \) and \( v \) is a cell of \( Y \), we write respectively \( \inl{u} \) and \( \inr{v} \) for the image of \( u \) and \( v \) through the canonical inclusions \( X \incl X \join Y \) and \( Y \incl X \join Y \).
Adapting \cite[Lemma 3.4]{chanavat2024diagrammatic}, we find that each cell of $X \join Y$ is of the form $u \join v$, $\inl{u}$, or \( \inr{v} \) for some cell $u$ of $X$ and some cell $v$ of $Y$.

\subsection{Marked directed complexes} \label{sec:marked}

\begin{dfn}[Marked directed complex]
	A \emph{marked directed complex} is a pair $(X, A)$ of a directed complex $X$ and a set $A \subseteq \gr{>0}{\cell X}$ of \emph{marked cells}.
	Given marked directed complexes $(X, A)$, $(Y, B)$, a \emph{morphism} $f\colon (X, A) \to (Y, B)$ is a morphism $f\colon X \to Y$ in $\dCpx$ such that $A \subseteq \invrs{f}B$.
\end{dfn}

\noindent
We let $\mdCpx$ denote the category of marked directed complexes and their morphisms.
We have an evident forgetful functor
\[
	\Um\colon \mdCpx \to \dCpx
\]
with a left adjoint $\flatm{-}\colon \dCpx \to \mdCpx$ sending $X$ to $(X, \varnothing)$ and a right adjoint $\sharpm{-}\colon \dCpx \to \mdCpx$ sending $X$ to $(X, \gr{>0}{\cell X})$.

\begin{dfn}[Marking of marked directed complexes]
	A morphism $f\colon (X, A) \to (Y, B)$ of marked directed complexes is a \emph{marking} if its underlying morphism of directed complexes is an isomorphism.
\end{dfn}

\begin{dfn}[Conservative morphism]
	A morphism $f\colon (X, A) \to (Y, B)$ of marked directed complexes is \emph{conservative} if $A = \invrs{f}B$.
\end{dfn}

\begin{prop} \label{prop:marking_local_embedding_factorisation}
	There is an orthogonal factorisation system on $\mdCpx$ whose left class is the class of markings, and right class is the class of conservative morphisms.
\end{prop}
\begin{proof}
	Both classes evidently contain all isomorphisms and are closed under composition.
	Let $f\colon (X, A) \to (Y, B)$ be a morphism of marked directed complexes.
	Then $A \subseteq \invrs{f}{B}$, so $\idd{X}\colon (X, A) \to (X, \invrs{f}{B})$ is well-defined as a marking, and $f$ factors as this marking followed by the conservative morphism $f\colon (X, \invrs{f}{B}) \to (Y, B)$.
	Essential uniqueness is straightforward.
\end{proof}

\begin{comm}
	This kind of factorisation system on marked presheaves is often called \emph{(entire, regular)}, for example in \cite{verity2008weak, campion2025cubical}.
	However, \emph{regular} is already overloaded, and anticipating the fact that in fibrant objects, marked cells will be exactly the equivalences, our terminology is consistent with the use of ``conservative functor'' for a functor that reflects isomorphisms.
\end{comm}

\noindent
Given a finite regular directed complex $P$, there is an evident bijection between the set $\cell P$ of cells in $P$ and the set of elements of $P$.
Thus, we can identify a set of marked cells in $P$ with a subset $A \subseteq \gr{>0}{P}$.

\begin{dfn}[Marking of top-dimensional elements]
	Let $P$ be a finite regular directed complex, $n \eqdef \dim P$.
	We let $\mrk{P}$ be the regular marked directed complex $(P, \gr{n}{P})$ and $\m_P\colon \flatm{P} \to \mrk{P}$ be the marking determined by the identity on $P$.
\end{dfn}

\noindent
In particular, if $U$ is an atom with greatest element $\top$, we have $\mrk{U} = (U, \set{\top})$.
We let 
\[
	\Imrk \eqdef \set{\m_U\colon \flatm{U} \to \mrk{U} \mid \text{$U$ is an atom}},
\]
the set of \emph{top-markings} of atoms.
If $U$ is of the form $V \celto W$, we will also write $V \mcelto W$ for $\mrk{U}$.

\begin{prop} \label{prop:colimits_of_marked_directed_complexes}
	Let $\fun{F}\colon \cat{J} \to \mdCpx$ be a small diagram.
	Then
	\begin{enumerate}
		\item every colimit cone $\gamma$ under $\Um\fun{F}$ in $\dCpx$ with tip $X$ lifts to a colimit cone under $\fun{F}$ in $\mdCpx$ with tip
			\[
				\left( X, \bigcup_{j \in \Ob \cat{J}} \set{\gamma_j(u) \mid \text{$u$ is marked in $\fun{F}j$}} \right),
			\]
		\item every limit cone $\delta$ over $\Um\fun{F}$ in $\dCpx$ with tip $X$ lifts to a limit cone over $\fun{F}$ in $\mdCpx$ with tip
			\[
				\left( X, \bigcap_{j \in \Ob \cat{J}} \set{u \in \gr{>0}{\cell X} \mid \text{$\delta_j(u)$ is marked in $\fun{F}j$}} \right).
			\]
	\end{enumerate}
\end{prop}
\begin{proof}
	Straightforward.
\end{proof}

\noindent
In this and other claims of local finite presentability, we use the criterion of \cite[Corollary 1.52]{adamek1994locally} that a category is locally finitely presentable if and only if it is a category of models of a finite limit sketch.

\begin{prop} \label{cor:local_presentability_of_mdcpx}
	The category $\mdCpx$ is locally finitely presentable.
\end{prop}
\begin{proof}
	A finite limit sketch for $\mdCpx$ is given by (the opposite of) its full subcategory on objects $\flatm{U}$ and $\mrk{U}$ ranging over atoms $U$, together with the set of squares
\begin{equation} \label{eq:marking_marking_pushout}
\begin{tikzcd}
	{\flatm{U}} & {\mrk{U}} \\
	{\mrk{U}} & {\mrk{U}}
	\arrow["{\m_U}", from=1-1, to=1-2]
	\arrow["{\m_U}", from=1-1, to=2-1]
	\arrow[equals, from=1-2, to=2-2]
	\arrow[equals, from=2-1, to=2-2]
\end{tikzcd}\end{equation}
	with $U$ ranging over atoms, which ensures that being marked is a property and not a structure on a cell (that is, if two marked cells have the same underlying cell, they are equal).
\end{proof}

\begin{comm}
	Since forcing the square (\ref{eq:marking_marking_pushout}) to be a pushout is equivalent to forcing the representable morphism $\m_U$ to be an epimorphism, $\mdCpx$ is in fact a quasitopos of separated presheaves; compare \cite[Corollary 2.28]{chanavat2024model}.
\end{comm}

\noindent
Given directed complexes $X$ and $Y$ and sets $A \subseteq \cell X$ and $B \subseteq \cell Y$, we let $A \gray B \subseteq \cell (X \gray Y)$ denote the set of cells of the form $u \gray v$ for $u \in A$ and $v \in B$.

\begin{dfn}[Gray product of marked directed complexes]
	Let $(X, A)$ and $(Y, B)$ be marked directed complexes.
	The \emph{Gray product of $(X, A)$ and $(Y, B)$} is the marked directed complex
	\[
		(X, A) \gray (Y, B) \eqdef (X \gray Y, (\cell X \gray B) \cup (A \gray \cell B)).
	\]
\end{dfn}

\begin{prop} \label{prop:gray_product_of_marked_directed_complexes}
	Let $f\colon (X, A) \to (Y, B)$ and $g\colon (X', A') \to (Y', B')$ be morphisms of marked directed complexes.
	Then the Gray product $f \gray g$ of their underlying morphisms of directed complexes determines a morphism
	\[
		f \gray g\colon (X, A) \gray (X', A') \to (Y, B) \gray (Y', B').
	\]
	This determines a biclosed monoidal structure $(\mdCpx, \gray, \flatm{1})$.
\end{prop}
\begin{proof}
	The fact that the Gray product determines a monoidal structure, as well as the fact that the embedding is strong monoidal, are both straightforward by direct calculation.
	Moreover, using the fact that Gray product of directed complexes is biclosed, hence preserves colimits in each variable, together with the characterisation of colimits in $\mdCpx$, it can be checked explicitly that the Gray product of marked directed complexes preserves colimits in each variable.
	Since $\mdCpx$ is locally presentable, this suffices to conclude that the monoidal structure is biclosed.
\end{proof}

\begin{rmk} \label{rmk:Um_strict_monoidal}
	By construction, $\Um\colon (\mdCpx, \otimes, \flatm{1}) \to (\dCpx, \otimes, 1)$ is strict monoidal.
\end{rmk}

\begin{comm}
	Since, in this article, we will only need to use the join of unmarked directed complexes, we omit any discussion marked versions of the join, or any other extensions, although it is certainly possible to do so.
	For the same reason, we do not discuss the \emph{pseudo} versions of either construction, which can be defined along the same lines as in \cite{loubaton2024inductive}.
\end{comm}

\subsection{Weak model structures on marked directed complexes} \label{sec:weakonmarked}

\noindent
Our goal in this section will be to apply Proposition \ref{prop:weak_cisinski_olschok} in order to put weak model structures on $\mdCpx$.
We start by letting
\[
	\fun{I} \eqdef \marr \gray -\colon \mdCpx \to \mdCpx,
\]
and, for each $\a \in \set{-, +}$ and each marked directed complex $(X, A)$,
\[
	\iota^\a_{(X, A)} \eqdef (0^\a \gray \idd{(X, A)})\lambda_{(X, A)}\colon (X, A) \iso \flatm{1} \gray (X, A) \to \marr \gray (X, A),
\]
where $\lambda\colon \Idd{} \iso \flatm{1} \gray \Idd{}$ is the left unitor for the Gray product.

\begin{lem} \label{lem:gray_functorial_cylinder}
	The pair $(\fun{I}, (\iota^-, \iota^+))$ is a functorial cylinder on $\mdCpx$.
\end{lem}
\begin{proof}
	The functor $\fun{I}$ is left adjoint by Proposition \ref{prop:gray_product_of_marked_directed_complexes}, and $\iota^-, \iota^+$ are composites of natural transformations, hence natural.
\end{proof}

\begin{lem} \label{lem:pushout_product_marking_with_cylinder_boundary}
	Let $s\colon (X, A) \to (Y, B)$ be a marking of marked directed complexes.
	Then $(\iota^-, \iota^+) \ppnat s$ is an isomorphism.
\end{lem}
\begin{proof}
	Same as \cite[Lemma 4.5]{chanavat2024model}.
\end{proof}

\noindent
Next, as a set of generating cofibrations, we let
\[
	I \eqdef \flatm{(\Ibd)} \cup \Imrk,
\]
the set of boundary inclusions and top-markings of atoms.

\begin{lem} \label{lem:I_cofibrations_are_monomorphisms}
	The class $\llp(\rlp(I))$ is the class of all monomorphisms in $\mdCpx$.
	Consequently,
	\begin{enumerate}
		\item every marked directed complex is $I$-cofibrant,
		\item the class $\ICof{I}$ is the class of all monomorphisms in $\mdCpx$.
	\end{enumerate}
\end{lem}
\begin{proof}
	Same as \cite[Lemma 3.20]{chanavat2024model}.
\end{proof}

\begin{lem} \label{lem:pushout_product_cofibration_cylinder_boundary}
	Let $i \in I$.
	Then $(\iota^-, \iota^+) \ppnat i$ is in $\llp(\rlp(I))$.
\end{lem}
\begin{proof}
	By Lemma \ref{lem:I_cofibrations_are_monomorphisms}, it suffices to show that the pushout-product is a monomorphism.
	If $i = \flatm{\bdmap}_U\colon \flatm{\bd{}{}U} \incl \flatm{U}$, then $(\iota^-, \iota^+) \ppnat i$ can be computed in $\frdCpx_{\L}$ as the mediating morphism in the commutative diagram
\[\begin{tikzcd}
	{\bd{}{}\arr \gray \bd{}{}U} & {\arr \gray \bd{}{}U} \\
	{\bd{}{}\arr \gray U} & {\bd{}{}\arr \gray U \cup \arr \gray \bd{}{}U} & {\arr \gray U}
	\arrow[hook, from=1-1, to=1-2]
	\arrow[hook, from=1-1, to=2-1]
	\arrow[hook, from=1-2, to=2-2]
	\arrow[curve={height=-12pt}, from=1-2, to=2-3]
	\arrow[hook, from=2-1, to=2-2]
	\arrow[curve={height=30pt}, from=2-1, to=2-3]
	\arrow["\lrcorner"{anchor=center, pos=0.125, rotate=180}, draw=none, from=2-2, to=1-1]
	\arrow[hook, dashed, from=2-2, to=2-3]
\end{tikzcd}
\]
	then lifted to $\mdCpx$, exhibiting it as the inclusion of the boundary of $\marr \gray \flatm{U}$.
	If $i = \m_U\colon \flatm{U} \to \mrk{U}$, then $i$ is a marking, and we conclude by Lemma \ref{lem:pushout_product_marking_with_cylinder_boundary}. 
\end{proof}

\noindent
All our weak model structures on $\mdCpx$ will share the functorial cylinder $(\fun{I}, (\iota^-, \iota^+))$ and the set $I$ of generating cofibrations; they will differ in their sets of generating anodyne extensions.
We closely follow \cite[Section 3.2]{chanavat2024model}.

\begin{dfn}[Marked round diagram]
	Let $u\colon U \to X$ be a round diagram in a marked directed complex $(X, A)$.
	We say that $u$ is \emph{marked} if $u\colon \flatm{U} \to (X, A)$ factors through $\m_U\colon \flatm{U} \to \mrk{U}$.
\end{dfn}

\begin{rmk}
	Equivalently, let $n \eqdef \dim{U}$; then $u$ is marked if and only if $\restr{u}{\clset{x}}\colon \clset{x} \to X$ is a marked cell for all $x \in \gr{n}{U}$.
	Notice that this is compatible with the notion of marked cell in the case that $U$ is an atom.
\end{rmk}

\noindent 
If $h$ is a round diagram of type $u \rdto v$, we may write $h\colon u \mrdto v$ to denote that $h$ is a marked round diagram, and if $h$ is a marked cell, then we write $h\colon u \mcelto v$.

\begin{dfn}[Marked horn]
	Let $(U, B)$ be a marked atom with greatest element $\top \in B$, let $\a \in \set{-, +}$, let $x \in \faces{}{\a}U$, and let $(\order{i}{L}, \order{i}{R})_{i=1}^k$ be sequences of submolecules of $\bd{}{\a}U$ such that
	\begin{enumerate}
		\item $\dim \order{i}{L}, \dim \order{i}{R} \leq i$ for each $i \in \set{1, \ldots, k}$,
		\item $\bd{}{\a}U = \order{k}{L} \cp{k-1} (\ldots \cp{1} (\order{1}{L} \cp{0} \clset{x} \cp{0} \order{1}{R}) \cp{1} \ldots) \cp{k-1} \order{k}{R}$,
		\item $\gr{i}{\order{i}{L}} \cup \gr{i}{\order{i}{R}} \subseteq B$ for each $i \in \set{1, \ldots, k}$, and
		\item $x \in B$ if and only if $\faces{}{-\a}U \subseteq B$.
	\end{enumerate}
	Then, letting $\Lambda_U^x \eqdef U \setminus \set{\top, x}$, we say that the embedding
	\[
		\lambda_U^x\colon (\Lambda_U^x, B \cap \Lambda_U^x) \incl (U, B)
	\]
	is a \emph{marked horn of $(U, B)$}.
	We let $\Jhorn$ denote the set of all marked horns of marked atoms.
\end{dfn}

\begin{comm}
	As discussed in \cite[Comment 3.13]{chanavat2024model}, and with the terminology of \cite[Section 2.1]{chanavat2024equivalences}, given a marked directed complex $(X, A)$ and a marked horn $\lambda_U^x$ of $(U, B)$, each morphism $e\colon (\Lambda_U^x, B \cap \Lambda_U^x) \to (X, B)$ classifies an \emph{equation} of the form
	\[
		\fun{E}x \qeq b
	\]
	in the indeterminate $x$, where $\fun{E}$ is an $A$\nbd context.
	Specifically, supposing $x \in \faces{}{-}U$, $\restr{e}{\bd{}{-}U \setminus \set{x}}$ determines the round $A$-context $\fun{E}$, while $b$ is the round diagram $\restr{e}{\bd{}{+}U}$.
	An extension of $e$ along $\lambda_U^x$, also known as a \emph{filler} for the horn, is then a \emph{lax solution} to this equation, that is, a cell $h\colon \fun{E}a \celto b$, where furthermore, 	\begin{enumerate}
		\item $h$ is required to be marked,
		\item if $\faces{}{+}U \subseteq B$, so $b$ is a marked diagram, then $a$ is also required to be marked.
	\end{enumerate}
	Dually, if $x \in \faces{}{+}U$, an extension is a \emph{colax solution} $h\colon b \celto \fun{E}a$ with the dual marking conditions.
	We will regularly identify equations and their (lax, colax) solutions with horns and their fillers.
\end{comm}

\begin{lem} \label{lem:pushout_product_marked_horn_cylinder_boundary}
	Let $j \in \Jhorn$.
	Then $(\iota^-, \iota^+) \ppnat j \in \Jhorn$.
\end{lem}
\begin{proof}
	Same as \cite[Lemma 4.6]{chanavat2024model}.
\end{proof}

\begin{lem} \label{lem:pushout_product_boundary_inclusions}
	Let $i \in I$ and $\a \in \set{-, +}$.
	Then $\iota^\a \ppnat i \in \llp(\rlp(\Jhorn))$.
\end{lem}
\begin{proof}
	Suppose $i = \flatm{\bdmap}_U\colon \flatm{\bd{}{}U} \incl \flatm{U}$.
	Then \cite[Lemma 4.7]{chanavat2024model} shows that $\iota^\a \ppnat i$, as well as $\iota^\a \ppnat \m_U i$, are in $\Jhorn \subseteq \llp(\rlp(\Jhorn))$.
	As in the proof of \cite[Lemma 4.8]{chanavat2024model}, then, we use the fact that $\iota^\a \ppnat \m_U$ is a marking to deduce that any morphism with the right lifting property against $\iota^\a \ppnat \m_U i$ also has the right lifting property against $\iota^\a \ppnat \m_U$, and we conclude.
\end{proof}

\noindent
Because directed complexes do not have an algebraic notion of units, we cannot reproduce either the ``walking equivalence'' or the ``walking pair of invertors'' from \cite[Section 3.2]{chanavat2024model}.
Instead, we emulate the generating anodyne extensions for ``saturation'' as in \cite[Definition 3.4]{loubaton2024inductive}.

\begin{dfn}[Saturation]
	Let $n > 0$, and let $U$, $V$, $W$ be $n$\nbd dimensional atoms such that $U \cp{} V \cp{} W$ is defined.
	Then, let $UV \eqdef \mrg{U \cp{} V}$, $VW \eqdef \mrg{V \cp{} W}$, $R \eqdef UV \celto (U \cp{} V)$, and $L \eqdef (V \cp{} W) \celto VW$.
	Let $u, v, w, uv, vw, r, \ell$ denote the greatest elements of $U, V, W, UV, VW, R, L$, respectively.
	Then, letting $\Sigma_{U,V,W} \eqdef (R \cp{} W) \cp{} (V \cp{} L)$, we say that the marking
	\[
		\sigma_{U,V,W}\colon (\Sigma_{U, V, W}, \set{uv, vw, r, \ell}) \to (\Sigma_{U,V,W}, \set{u, v, w, uv, vw, r, \ell})
	\]
	is a \emph{saturation} of shape $(U, V, W)$.
	We let $\Jsat$ denote the set of all saturations.
\end{dfn}

\begin{comm}
	Given a marked directed complex $(X, A)$, a morphism from $\Sigma_{U, V, W}$, where $U$, $V$, $W$ are $n$\nbd dimensional, classifies a pair of $(n+1)$\nbd dimensional cells of the form $z_L\colon a \cp{} a_L \celto e$ and $z_R\colon h \celto a_R \cp{} a$, where $e, h, a, a_L, a_R$ are all $n$\nbd dimensional cells.
	If $(X, A)$ has the right lifting property against $\Jsat$, whenever $z_L$, $z_R$, $e$, and $h$ are all marked, then $a, a_L, a_R$ are also marked.
\end{comm}

\noindent
Finally, for each $n \in \Ninfty$, we let
\[
	\Jn \eqdef \Jhorn \cup \Jsat \cup \set{\m_U \in I_\m \mid \dim U > n},
\]
and $\Mwn \eqdef (\mdCpx, \ICof{I}, \JFib{\Jn})$; notice that $\Jn = \Jhorn \cup \Jsat$ when $n = \infty$.

\begin{thm} \label{thm:weak_model_structures}
	For each $n \in \Ninfty$, $\Mwn$ is a weak model structure.
\end{thm}
\begin{proof}
	We use Proposition \ref{prop:weak_cisinski_olschok}.
	All morphisms in $\Jn$ are monomorphisms, so the first point follows from Lemma \ref{lem:I_cofibrations_are_monomorphisms}.
	The second point is Lemma \ref{lem:pushout_product_cofibration_cylinder_boundary}.
	For the third point, for each $i \in I$ and $\a \in \set{-, +}$, by 
	Lemma \ref{lem:pushout_product_boundary_inclusions} we have $\iota^\a \ppnat i \in \llp(\rlp(\Jhorn))$; but $\Jhorn \subseteq \Jn$ implies $\llp(\rlp(\Jhorn)) \subseteq \llp(\rlp(\Jn))$, and $\JFib{J} \subseteq \rlp(\Jn)$ implies $\llp(\rlp(\Jn)) \subseteq \llp(\JFib{\Jn})$.
	Finally, for the fourth point, we use Lemma \ref{lem:pushout_product_marked_horn_cylinder_boundary} when $j$ is a marked horn, and Lemma \ref{lem:pushout_product_marking_with_cylinder_boundary} when $j$ is either a saturation or $\m_U$ for some atom $U$ of dimension $> n$, and we conclude.
\end{proof}

\section{A weak model} \label{part:weak}

\subsection{Inflate-complexes and \inftyn-categories} \label{sec:inflate}

\noindent
Our next goal is to characterise the fibrant objects in $\Mwn$, but compared to the analogous task for diagrammatic sets, we are inconvenienced by the lack of algebraic units and unitors.
This would not be the case had we worked with presheaves on $\atom_{\C\E}$, with collapses producing at least some of the most useful degenerate cells considered in \cite[Section 1.2]{chanavat2024equivalences}.
On the other hand, the reason why we are working in a non-unital setting is that, since collapses are not closed under Gray products, there is no natural way to define a functorial cylinder on such a presheaf category, hence no natural way to define a (weak) model structure \`a la Cisinski--Olschok.

We find the following solution, balancing between the two opposing pulls: we will avoid putting a (weak) model structure on categories of objects with units and unitors, but we will show that every fibrant object in $\Mwn$ can be endowed with algebraic units and unitors; this can be compared with the classical result that semisimplicial sets satisfying the Kan condition can be given the structure of a simplicial set \cite{rourke1971delta, mcclure2013semisimplicial}, but is actually simpler, due to the freeness of cylindrical collapses as given by Proposition \ref{prop:freeness_of_collapses} (simplicial co-degeneracies, in comparison, have many non-trivial relations).

\begin{dfn}[Inflate-complex]
	An \emph{inflate-complex} is a $\Gamma$\nbd continuous presheaf on the category $\frdCpx_{\C\L}$ of finite regular directed complexes and local collapses.
\end{dfn}

\noindent
We let $\ICpx$ denote the category $\PSh_\Gamma(\frdCpx_{\C\L})$ of inflate-complexes.
The proof of Proposition \ref{prop:guises_of_directed_complexes} goes through essentially unmodified to show the following.

\begin{prop} \label{prop:guises_of_inflate_complexes}
	The following categories are equivalent:
	\begin{enumerate}[label=(\alph*)]
		\item the category of inflate-complexes;
		\item the category of presheaves on $\atom_{\C\E}$.
	\end{enumerate}
\end{prop}

\noindent
As directed complexes are ``diagrammatic sets without degeneracies'', inflate-complexes are ``diagrammatic sets with limited degeneracies'', namely, those induced by cylindrical collapses.

\begin{dfn}[Degenerate and non-degenerate cells]
	Let $u\colon U \to X$ be a cell in an inflate-complex.
	We say that $u$ is \emph{non-degenerate} if, for all collapses of atoms $p\colon U \surj V$ and cells $v\colon V \to X$, the equation $u = vp$ implies that $u = v$ and $p = \idd{U}$.
	We say that $u$ is \emph{degenerate} otherwise.
\end{dfn}

\noindent
We let $\dgn X$ denote the set of degenerate cells in an inflate-complex $X$.
Like diagrammatic sets, inflate-complexes enjoy the ``Eilenberg--Zilber property'' that every degenerate cell is degenerate in a unique way.

\begin{lem} \label{lem:eilenberg_zilber_collapses}
	Let $u\colon U \to X$ be a cell in an inflate-complex.
	Then there exists a unique pair of a collapse of atoms $p\colon U \surj V$ and a non-degenerate cell $v\colon V \to X$ such that $u = vp$.
\end{lem}
\begin{proof}
	Same as \cite[Lemma 2.3]{chanavat2024diagrammatic}, using Lemma \ref{lem:sections_of_collapses}.
\end{proof}

\noindent
Local collapses suffice to generate the following families of degenerate cells.

\begin{dfn}[Unit]
	Let $u\colon U \to X$ be a pasting diagram in an inflate-complex.
	The \emph{unit on $u$} is the pasting diagram $\un u\colon u \rdto u$ defined by $u\tau_{\bd{}{}U}\colon \arr \pcyl{\bd{}{}U} U \to X$.
\end{dfn}

\begin{dfn}[Left unitor]
	Let $u\colon U \to X$ be a pasting diagram in an inflate-complex and let $\iota\colon v \submol \bd{}{-}u$ be a rewritable subdiagram of shape $V$ in its input boundary.
	Let $K \eqdef \bd{}{}U \setminus \inter \iota(V)$.
	The \emph{left unitor of $u$ at $\iota$} is the pasting diagram $\lun{\iota}u\colon u \rdto \un v \cpsub{\iota} u$ defined by $u\tau_{K}\colon \arr \pcyl{K} U \to X$.
\end{dfn}

\begin{dfn}[Right unitor]
	Let $u\colon U \to X$ be a pasting diagram in a diagrammatic set and let $\iota\colon v \submol \bd{}{+}u$ be a rewritable subdiagram of shape $V$ in its output boundary.
	Let $K \eqdef \bd{}{}U \setminus \inter \iota(V)$.
	The \emph{right unitor of $u$ at $\iota$} is the pasting diagram $\run{\iota}u\colon u \subcp{\iota} \un v \rdto u$ defined by $u\tau_{K}\colon \arr \pcyl{K} U \to X$.
\end{dfn}

\begin{rmk}
	If $u$ is round or a cell, then so are $\un u$, $\lun{\iota}u$, and $\run{\iota}u$.
\end{rmk}

\begin{exm}
	If $u$ is a 0\nbd dimensional cell, there is only one degenerate 1\nbd dimensional cell over $u$, that is, the unit $\un u$.
	If $u\colon v \celto w$ is a 1\nbd dimensional cell, the degenerate 2\nbd dimensional cells over $u$ are precisely
	
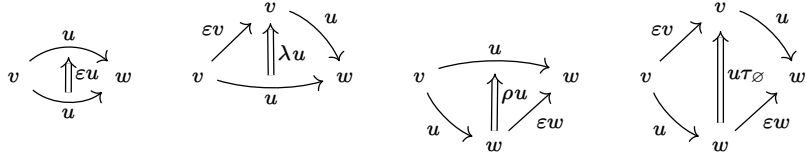
\[\begin{tikzcd}[column sep=small, row sep=scriptsize]
	&&&& {\scriptstyle v} &&&&&& {\scriptstyle v} \\
	{\scriptstyle v} && {\scriptstyle w} & {\scriptstyle v} && {\scriptstyle w} & {\scriptstyle v} && {\scriptstyle w} & {\scriptstyle v} && {\scriptstyle w} \\
	&&&&&&& {\scriptstyle w} &&& {\scriptstyle w}
	\arrow["u", curve={height=-6pt}, from=1-5, to=2-6]
	\arrow["u", curve={height=-6pt}, from=1-11, to=2-12]
	\arrow[""{name=0, anchor=center, inner sep=0}, "u"', curve={height=12pt}, from=2-1, to=2-3]
	\arrow[""{name=1, anchor=center, inner sep=0}, "u", curve={height=-12pt}, from=2-1, to=2-3]
	\arrow["{\un v}", from=2-4, to=1-5]
	\arrow[""{name=2, anchor=center, inner sep=0}, "u"', curve={height=6pt}, from=2-4, to=2-6]
	\arrow[""{name=3, anchor=center, inner sep=0}, "u", curve={height=-6pt}, from=2-7, to=2-9]
	\arrow["u"', curve={height=6pt}, from=2-7, to=3-8]
	\arrow["{\un v}", from=2-10, to=1-11]
	\arrow["u"', curve={height=6pt}, from=2-10, to=3-11]
	\arrow["{\un w}"', from=3-8, to=2-9]
	\arrow["{u\tau_\varnothing}"', shorten <=3pt, shorten >=3pt, Rightarrow, from=3-11, to=1-11]
	\arrow["{\un w}"', from=3-11, to=2-12]
	\arrow["{\un u}"', shorten <=3pt, shorten >=3pt, Rightarrow, from=0, to=1]
	\arrow["{\lun{}u}"', shorten <=4pt, Rightarrow, from=2, to=1-5]
	\arrow["{\run{} u}"', shorten >=4pt, Rightarrow, from=3-8, to=3]
\end{tikzcd}\]
	corresponding to the subsets $\set{0^-, 0^+}$, $\set{0^+}$, $\set{0^-}$, and $\varnothing$ of $\bd{}{}\arr$.
\end{exm}

\noindent
We have an evident presheaf extension-restriction adjunction
\[\begin{tikzcd}
	\dCpx && \ICpx.
	\arrow[""{name=0, anchor=center, inner sep=0}, "{\FInfl}", curve={height=-12pt}, from=1-1, to=1-3]
	\arrow[""{name=1, anchor=center, inner sep=0}, "{\UInfl}", curve={height=-12pt}, from=1-3, to=1-1]
	\arrow["\dashv"{anchor=center, rotate=-90}, draw=none, from=0, to=1]
\end{tikzcd}\]
We will use the following result to describe it explicitly.

\begin{lem} \label{lem:presheaf_restriction_to_right_class}
	Let $\cat{C}$ be a category with an orthogonal factorisation system $(\L, \R)$ and let $\cat{C}_\L$ and $\cat{C}_\R$ be its wide subcategories on $\L$ and $\R$-morphisms, respectively.
	Then
	\begin{enumerate}
		\item the presheaf restriction functor $\PSh(\cat{C}) \to \PSh(\cat{C}_\R)$ is monadic,
		\item the induced monad $\fun{M}$ admits the following explicit description:
		\begin{itemize}
			\item the underlying functor sends a presheaf $X$ on $\cat{C}_\R$ to the presheaf $\fun{M}X$ whose elements $C \to \fun{M}X$ are equivalence classes of pairs
			\[
				[x\colon C' \to X, \ell\colon C \to C'], \quad \ell \in \L,
			\]
			under the equivalence relation $[x, \ell] = [x\invrs{\varphi}, \varphi\ell]$ for all isomorphisms $\varphi\colon C' \iso C''$, with the action of morphisms in $\cat{C}$ given by
			\[
				[x, \ell]f \eqdef [xr', \ell'],
			\]
			where $\ell f = r'\ell'$ is an $(\L, \R)$ factorisation;
			\item the unit $\eta^{\fun{M}}_X\colon X \to \fun{M} X$ sends $x\colon C \to X$ to $[x, \idd{C}]$;
			\item the multiplication $\mu^\fun{M}_X\colon \fun{M}\fun{M} X \to X$ sends $[[x, \ell], \ell']$ to $[x, \ell\ell']$.
		\end{itemize}
	\end{enumerate}
\end{lem}
\begin{proof}
	Since the presheaf restriction is along a bijective-on-objects functor, by \cite[Example A4.2.7(b)]{johnstone2002elephant} it reflects isomorphisms, and since it is both a left and a right adjoint, the conditions of the crude monadicity theorem apply, so the adjunction is monadic.
	Given a presheaf $X$ on $\cat{C}_\R$, the presheaf extension $\fun{F}X$ of $X$ is computed by the coend
	\[
		\int^{C \in \Ob(\cat{C}_\R)} \Hom_{\cat{C}}(-, C) \times X(C),
	\]
	but since $(\L, \R)$ is an orthogonal factorisation system on $\cat{C}$, we have a natural isomorphism
	\[
		\Hom_{\cat{C}}(-, C) \simeq \int^{C' \in \Ob(\core(\cat{C}))} \Hom_{\cat{C}_\L}(-, C') \times \Hom_{\cat{C}_\R}(C', C)
	\]
	where $\core(\cat{C})$ is the core groupoid of $\cat{C}$.
	By coend calculus, $\fun{F}X$ is naturally isomorphic to
	\begin{align*}
		\int^{C \in \Ob(\cat{C}_\R), C' \in \Ob(\core(\cat{C}))} & \Hom_{\cat{C}_\L}(-, C') \times \Hom_{\cat{C}_\R}(C', C) \times X(C) \\
					& \simeq \int^{C' \in \Ob(\core(\cat{C}))} \Hom_{\cat{C}_\L}(-, C') \times X(C')
	\end{align*}
	whose restriction to $\PSh(\cat{C}_\R)$ is precisely $\fun{M}X$.
	Computing the unit and multiplication is a straightforward exercise.
\end{proof}

\noindent 
In the case that $\cat{C}$ has no non-trivial isomorphisms, the equivalence relation becomes trivial, and the elements of $\fun{M}X$ are simply pairs $(x, \ell)$.
This is the case for the category $\atom_{\C\E}$.

\begin{dfn}[Inflate monad]
	The \emph{inflate monad} is the monad $(\Infl, \mu^\Infl, \eta^\Infl)$ on $\dCpx$ induced by monadic adjunction $\FInfl \dashv \UInfl$ according to Lemma \ref{lem:presheaf_restriction_to_right_class}.
\end{dfn}

\noindent
Thus, given a directed complex $X$, cells of shape $U$ in $\Infl X$ are pairs $(v, p)$ of a collapse of atoms $p\colon U \surj V$ and a cell $v\colon V \to X$.
Next, we lift inflate-complexes to the marked world.

\begin{dfn}[Marked inflate-complex]
	A \emph{marked inflate-complex} is a pair $(X, A)$ of an inflate-complex $X$ and a set $A \subseteq \gr{>0}{\cell X}$ such that $\dgn X \subseteq A$.
	Given marked inflate-complexes $(X, A)$, $(Y, B)$, a \emph{morphism} $f\colon (X, A) \to (Y, B)$ is a morphism $f\colon X \to Y$ in $\ICpx$ such that $A \subseteq \invrs{f}B$.
\end{dfn}

\noindent
We let $\mICpx$ denote the category of marked inflate-complexes.

\begin{dfn}[Marked inflate monad]
	Let $(X, A)$ be a marked directed complex.
	We let $\mInfl(X, A)$ be the marked inflate-complex
	\[
		\left(\Infl X, \set{(u, \idd{}) \mid u \in A} \cup \set{(v, p) \mid v \in \cell X, p \neq \idd{}}\right).
	\]
	With the structure determined by the underlying morphisms of the inflate monad on $\dCpx$, this determines a monad $(\mInfl, \eta^\Infl, \mu^\Infl)$ on $\mdCpx$.
\end{dfn}

\noindent
The following is straightforward, observing that the condition for an $\Infl$\nbd algebra to lift to an $\mInfl$\nbd algebra is precisely that $\dgn X \subseteq A$.

\begin{prop} \label{prop:guises_of_marked_inflate_complexes}
	The following categories are equivalent:
	\begin{enumerate}[label=(\alph*)]
		\item the category of marked inflate-complexes;
		\item the category of $\mInfl$-algebras on $\mdCpx$.
	\end{enumerate}
\end{prop}

\noindent
We will not be particularly concerned with the category of (marked) inflate-complexes, but the description of its objects as $\Infl$-algebras is useful in the proof of the following result.

\begin{thm} \label{thm:fibrants_admit_inflate_algebra_structure}
	Let $(X, A)$ be a $\Jhorn$-fibrant marked directed complex.
	Then $(X, A)$ admits an $\mInfl$-algebra structure.
\end{thm}
\begin{proof}
	Since we assumed $\atom_\E$ skeletal, unique representatives for cylindrical collapses are fixed, and we will freely use parametrisations of atoms given by unique isomorphisms of atoms.
	Let $p\colon U \surj V$ be a non-trivial collapse of atoms.
	By Proposition \ref{prop:freeness_of_collapses}, $p$ admits a unique factorisation $p = q\tau_K$, where $U = \arr \pcyl{K} U'$ and $\tau_K\colon \arr \pcyl{K} U' \surj U'$ is a generating collapse.
	Then, we let $K(p) \eqdef \clset{(0^-, \top_{U'})} \subseteq \bd{}{}U$ and
	\[
		a(p)\eqdef p\tau_{K(p)}\colon \arr \pcyl{K(p)} U \surj V.
	\]
	This operation produces a non-trivial collapse, so it can be iterated: we let $a^0(p) \eqdef p$ and $a^n(p) \eqdef a(a^{n-1}(p))$ for $n > 0$.
	Moreover, it has the following property: let $x \in \faces{}{}U' \setminus K$, let $K_x \eqdef \clset{x} \cap K$, and let 
	\[
		p_x \eqdef \restr{p}{\clset{(1, x)}}\colon \arr \pcyl{K_x} \clset{x} \surj \clset{p(x)};
	\]
	then, since $K(p) \cap \clset{(1, x)} = \clset{(0^-, x)} = K(p_x)$, we have
	\begin{equation} \label{eq:restrictions_of_the_a(p)}
		\restr{a(p)}{\clset{(1, (1, x))}} = a(p_x)\colon \arr \pcyl{K(p_x)} (\arr \pcyl{K_x} \clset{x}) \surj \clset{p(x)}.
	\end{equation}	
	Now, we will construct an $\Infl$-algebra structure $\alpha\colon \Infl X \to X$ inductively, as follows.
	Unitality fixes the structure uniquely on pairs $(u, \idd{})$, so it suffices to define $\alpha$ on pairs $(v, p)$ such that $p\colon U \surj V$ is non-trivial; this means in particular that $\alpha$ is uniquely defined on 0\nbd cells of $\Infl X$.
	Moreover, for each $m \in \mathbb{N}$, suppose that $\alpha$ is well-defined on $k$\nbd cells of $\Infl X$ for all $k \leq m$, satisfying the restricted unitality and associativity equations.
	Then, since the proof of Lemma \ref{lem:eilenberg_zilber_collapses} for $m$\nbd cells does not use any higher-dimensional collapses or embeddings, we can assume that each $m$\nbd cell $v$ is equal to $\alpha(w, q)$ for a unique pair of a collapse $q$ and non-degenerate cell $w$ of dimension $\leq m$.

	Inductively on $m > 0$, for each pair $(v, p)$ of a non-trivial collapse $p\colon U \surj V$ with $\dim U = m$, we will construct marked cells $\alpha(v, a^n(p))$ for all $n \geq 0$; for the reason just mentioned, we may freely assume that $v$ is non-degenerate, for if $v = \alpha(v', p')$ with non-trivial $p'$, then we must have
	\[
		\alpha(v, a^n(p)) = \alpha(\alpha(v', p'), a^n(p)) = \alpha(v', a^n(p'p)).
	\]
	Now, let $p = q\tau_K$ be the unique factorisation given by Proposition \ref{prop:freeness_of_collapses}, and let $u' \eqdef \alpha(v, q)$.
	If $q$ is non-trivial and $p = a(q)$, then $\alpha(v, a^n(p))$ must be equal to $\alpha(v, a^{n+1}(q))$, which is already defined by the inductive hypothesis.
	In any other case, $\alpha(u', a^n(\tau_K))$ is yet undefined.
	Since by Proposition \ref{prop:freeness_of_collapses} there are no non-trivial equations between composites of cylindrical collapses, we can extend $\alpha$ freely on this family of cells, inductively on $n \geq 0$, as long as we can construct marked cells with the correct boundaries at each step.

	Let $U'$ be the shape of $u'$ and consider the sequence of marked atoms
	\[
		\order{0}{U} \eqdef \marr \pcyl{K} \flatm{(U')}, \quad \quad
		\order{n}{U} \eqdef \marr \pcyl{K(a^{n-1}(\tau_K))} \order{n-1}{U} 
		\quad \text{for $n > 0$},
	\]
	whose underlying atom is the domain of $a^n(\tau_K)$ for each $n \geq 0$, and the set of marked cells, for each $n \geq 0$, is $\set{x \in \order{n}{U} \mid \dim a^n(\tau_K)(x) < \dim x}$.
	Moreover, let $x_0 \eqdef (0^+, \top) \in \order{0}{U}$ and $x_n \eqdef (1, x_{n-1}) \in \order{n}{U}$ for each $n > 0$.
	We will construct, inductively on $n \geq 0$, morphisms 
	\[
		s_n\colon \order{n}{U} \to (X, A)
	\]
	classifying marked cells of $X$, and, for each $n > 0$, we will then let
	\[
		\alpha(u', a^{n-1}(\tau_K)) \eqdef \restr{s_n}{\clset{x_n}},
	\]
	which is marked because every cell in $\faces{}{}\order{n}{U}$ is.
	Let $n = 0$ and let $\top$ be the greatest element of $U'$.
	Then we have a marked horn $(\Lambda^{(0^-, \top)}, \order{0}{B}) \incl \order{0}{U}$, where $\order{0}{B}$ is the restriction of the marked structure on $\order{0}{U}$.
	For each $x \in \faces{}{}U' \setminus K$, let $K_x$ and $p_x$ be defined as before, and let $u'_x \eqdef \restr{u'}{\clset{x}}$.
	Then
	\[
		\restr{e_0}{\clset{(0^+, \top)}} \eqdef u', \quad \quad
		\restr{e_0}{\clset{(1, x)}} \eqdef \alpha(u'_x, p_x)
	\]
	is well-defined as a morphism $e_0\colon (\Lambda^{(0^-, \top)}, \order{0}{B}) \to (X, A)$, so because $(X, A)$ is $\Jhorn$-fibrant, it extends to a morphism $s_0\colon \order{0}{U} \to (X, A)$.
	This classifies a marked cell of $X$, and satisfies $\restr{s_0}{\clset{x_0}} = u'$.

	Now, let $n > 0$.
	Then we have a marked horn $(\Lambda^{x_n}, \order{n}{B}) \incl \order{n}{U}$, where $\order{n}{B}$ is the restriction of the marked structure on $\order{n}{U}$.
	We define a morphism $e_n\colon (\Lambda^{x_n}, \order{n}{B}) \to (X, A)$ as follows: for each top-dimensional $z \in \Lambda^{x_n}$, we let
	\begin{itemize}
		\item $\restr{e_n}{\clset{z}} \eqdef s_{n-1}$ if $z \in \invrs{a^n(\tau_K)}\top$,
		\item $\restr{e_n}{\clset{z}} \eqdef \alpha(u'_x, a^n(p_x))$ if $z \in \invrs{a^n(\tau_K)}x$ for some $x \in \faces{}{}U' \setminus K$.
	\end{itemize}
	By $\Jhorn$-fibrancy, this extends to a morphism $s_n\colon \order{n}{U} \to (X, A)$.
	Both the fact that $e_n$ is well-defined, and the fact that $\restr{s_n}{\clset{x_n}}$ has the correct boundaries for $\alpha(u', a^{n-1}(\tau_K))$, follow from the recursive use of equation (\ref{eq:restrictions_of_the_a(p)}).
	This concludes the inductive step.
	By construction, every degenerate cell is mapped to a marked cell, so the $\Infl$\nbd algebra lifts to an $\mInfl$\nbd algebra.
\end{proof}

\begin{exm}
	Since the construction in Theorem \ref{thm:fibrants_admit_inflate_algebra_structure} may not be immediately transparent, we give an illustration of what it looks like in the simplest possible case, where $v$ is a 0\nbd dimensional cell and $p\colon \arr \to 1$ is the unique possible collapse.
	In the first step, to construct $s_0$, we consider the marked horn $(\Lambda^{(0^-, \top)}, \varnothing)$ of $\order{0}{U} \eqdef \marr$, which is simply the inclusion $\flatm{\set{0^+}} \incl \marr$, and $e_0$ is simply the classifying morphism of $v$.
	Then a filler $s_0$ classifies a marked 1\nbd dimensional cell
	\[
		s_0\colon v' \mcelto v.
	\]
	Next, $\order{1}{U}$ is $\sharpm{(\arr \pcyl{\set{0^-}} \arr)}$, the horn $\Lambda^{x_1}$ looks like the inclusion
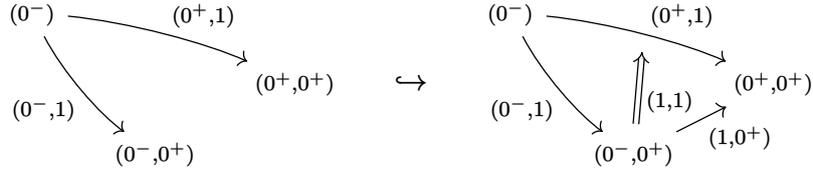
\[\begin{tikzcd}[sep=small]
	{{\scriptstyle(0^-)}} &&&& {{\scriptstyle(0^-)}} \\
	&& {{\scriptstyle(0^+, 0^+)}} & \incl &&& {{\scriptstyle(0^+, 0^+)}} \\
	& {{\scriptstyle(0^-, 0^+)}} &&&& {{\scriptstyle(0^-, 0^+)}}
	\arrow["{(0^+, 1)}", curve={height=-6pt}, from=1-1, to=2-3]
	\arrow["{(0^-, 1)}"', curve={height=6pt}, from=1-1, to=3-2]
	\arrow[""{name=0, anchor=center, inner sep=0}, "{(0^+, 1)}", curve={height=-6pt}, from=1-5, to=2-7]
	\arrow["{(0^-, 1)}"', curve={height=6pt}, from=1-5, to=3-6]
	\arrow["{(1, 0^+)}"', from=3-6, to=2-7]
	\arrow["{(1, 1)}"', shorten <=3pt, shorten >=6pt, Rightarrow, from=3-6, to=0]
\end{tikzcd}\]
	and $e_1$ classifies the diagram
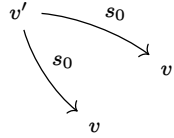
\[\begin{tikzcd}[sep=small]
	{{\scriptstyle v'}} \\
	&& {{\scriptstyle v}} \\
	& {{\scriptstyle v}}
	\arrow["{s_0}", curve={height=-6pt}, from=1-1, to=2-3]
	\arrow["{s_0}", curve={height=6pt}, from=1-1, to=3-2]
\end{tikzcd}\]
	whose filler exhibits a marked 2\nbd cell $s_1\colon s_0 \cp{} e \mcelto s_0$, where $e\colon v \mcelto v$.
	We will let $\alpha(v, p) \equiv \un v \eqdef e$.
	Now, $K(p) \subseteq \bd{}{}\order{1}{U}$ is $\clset{(0^-, 1)}$, so $\order{2}{U}$ is the 3\nbd dimensional atom whose input and output boundaries are, respectively,
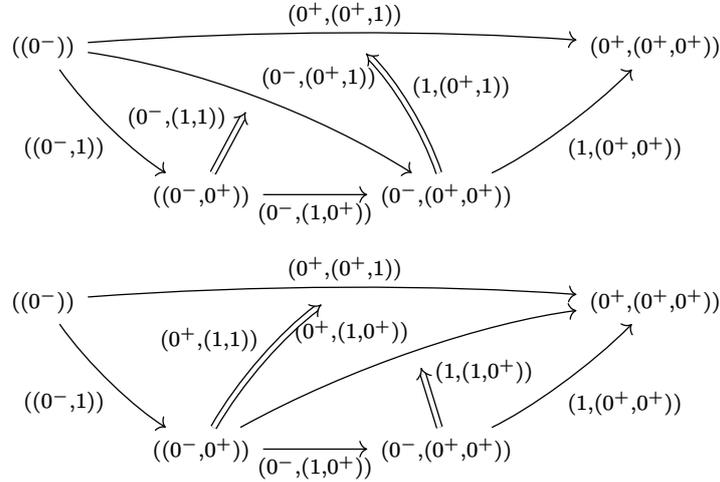
\[\begin{tikzcd}[column sep=scriptsize]
	{{\scriptstyle((0^-))}} &&&& {{\scriptstyle(0^+, (0^+, 0^+))}} \\
	\\
	& {{\scriptstyle((0^-, 0^+))}} && {{\scriptstyle(0^-, (0^+, 0^+))}}
	\arrow[""{name=0, anchor=center, inner sep=0}, "{(0^+, (0^+, 1))}", curve={height=-6pt}, from=1-1, to=1-5]
	\arrow["{((0^-, 1))}"', curve={height=6pt}, from=1-1, to=3-2]
	\arrow[""{name=1, anchor=center, inner sep=0}, "{(0^-, (0^+, 1))}", curve={height=-12pt}, from=1-1, to=3-4]
	\arrow["{(0^-, (1, 0^+))}"', from=3-2, to=3-4]
	\arrow["{(1, (0^+, 0^+))}"', curve={height=6pt}, from=3-4, to=1-5]
	\arrow["{(0^-, (1, 1))}", shorten >=6pt, Rightarrow, from=3-2, to=1]
	\arrow["{(1, (0^+, 1))}"', curve={height=6pt}, shorten >=10pt, Rightarrow, from=3-4, to=0]
\end{tikzcd}\]
\[\begin{tikzcd}[column sep=scriptsize]
	{{\scriptstyle((0^-))}} &&&& {{\scriptstyle(0^+, (0^+, 0^+))}} \\
	\\
	& {{\scriptstyle((0^-, 0^+))}} && {{\scriptstyle(0^-, (0^+, 0^+))}}
	\arrow[""{name=0, anchor=center, inner sep=0}, "{(0^+, (0^+, 1))}", curve={height=-6pt}, from=1-1, to=1-5]
	\arrow["{((0^-, 1))}"', curve={height=6pt}, from=1-1, to=3-2]
	\arrow[""{name=1, anchor=center, inner sep=0}, "{(0^+, (1, 0^+))}", curve={height=-12pt}, from=3-2, to=1-5]
	\arrow["{(0^-, (1, 0^+))}"', from=3-2, to=3-4]
	\arrow["{(1, (0^+, 0^+))}"', curve={height=6pt}, from=3-4, to=1-5]
	\arrow["{(0^+, (1, 1))}", curve={height=-6pt}, shorten >=10pt, Rightarrow, from=3-2, to=0]
	\arrow["{(1, (1, 0^+))}"', shorten >=6pt, Rightarrow, from=3-4, to=1]
\end{tikzcd}\]
	with all cells marked in dimension $> 0$, and $e_2$ classifies the diagram
	
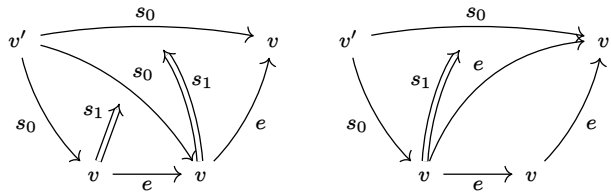
\[\begin{tikzcd}[column sep=small]
	{{\scriptstyle v'}} &&&& {{\scriptstyle v}} & {{\scriptstyle v'}} &&&& {{\scriptstyle v}} \\
	\\
	& {{\scriptstyle v}} && {{\scriptstyle v}} &&& {{\scriptstyle v}} && {{\scriptstyle v}}
	\arrow[""{name=0, anchor=center, inner sep=0}, "{s_0}", curve={height=-6pt}, from=1-1, to=1-5]
	\arrow["{s_0}"', curve={height=6pt}, from=1-1, to=3-2]
	\arrow[""{name=1, anchor=center, inner sep=0}, "{s_0}", curve={height=-12pt}, from=1-1, to=3-4]
	\arrow[""{name=2, anchor=center, inner sep=0}, "{s_0}", curve={height=-6pt}, from=1-6, to=1-10]
	\arrow["{s_0}"', curve={height=6pt}, from=1-6, to=3-7]
	\arrow["e"', from=3-2, to=3-4]
	\arrow["e"', curve={height=6pt}, from=3-4, to=1-5]
	\arrow["e", curve={height=-18pt}, from=3-7, to=1-10]
	\arrow["e"', from=3-7, to=3-9]
	\arrow["e"', curve={height=6pt}, from=3-9, to=1-10]
	\arrow["{s_1}", shorten >=6pt, Rightarrow, from=3-2, to=1]
	\arrow["{s_1}"', curve={height=6pt}, shorten >=10pt, Rightarrow, from=3-4, to=0]
	\arrow["{s_1}", curve={height=-6pt}, shorten >=10pt, Rightarrow, from=3-7, to=2]
\end{tikzcd}\]
	whose filler exhibits a cell
	\[
		s_2\colon (s_1 \cp{} e) \cp{} s_1 \mcelto (s_0 \cp{} h) \cp{} s_1
	\]
	where $h\colon e \cp{} e \mcelto e$.
	We will let $\alpha(v, a(p)) \equiv \run{}(\un v) \eqdef h$.
	Proceeding like this, at each stage we will have a horn whose top-dimensional cells can be mapped to the filler produced at the previous stage, and the result of filling the ``missing'' cell will be the desired structural degeneracy.
	The only difference when $v$ is not 0\nbd dimensional is that some of the top-dimensional cells will need to be mapped to previously-constructed degenerate cells over lower-dimensional faces of $v$.
\end{exm}

\noindent
We are now almost ready to define our weak model of \inftyn-categories.

\begin{dfn}[Marked-reducibility and marked-equivalence]
	Let $u$, $v$ be a pair of parallel round diagrams in a marked directed complex.
	We say that $u$ is \emph{marked-reducible to $v$}, and write $u \mrdto v$, if there exists a marked round diagram $h\colon u \mrdto v$.
	We say that $u$ is \emph{marked-equivalent to $v$}, and write $u \meqv v$, if $u \mrdto v$ and $v \mrdto u$.
\end{dfn}

\begin{lem} \label{lem:marked-equivalence_in_generic_inflate_complexes}
	Let $(X, A)$ be a marked inflate-complex, $n \in \mathbb{N}$, let $u, v, v', w$ in $\gr{n}{\Rd X}$, and let $\iota\colon v \submol u$ be a rewritable subdiagram.
	Then
	\begin{enumerate}
		\item $u \meqv u$, that is, $\mrdto$ and $\meqv$ are reflexive,
		\item if $u \mrdto v$ and $v \mrdto w$, then $u \mrdto w$, that is, $\mrdto$ is transitive,
		\item if $v' \mrdto v$, then $\subs{u}{v'}{\iota(v)} \mrdto u$, and if $v \mrdto v'$, then $u \mrdto \subs{u}{v'}{\iota(v)}$.
	\end{enumerate}
\end{lem}
\begin{proof}
	Reflexivity is exhibited by the unit $\un u\colon u \mrdto u$.
	If $h\colon u \mrdto v$ and $k\colon v \mrdto w$, then $h \cp{} k\colon u \mrdto w$.
	Finally, if $h\colon v' \mrdto v$, then we have a marked round diagram $h \cpsub{} \un u\colon \subs{u}{v'}{\iota(v)} \mrdto u$, and dually when $v \mrdto v'$.
\end{proof}

\begin{comm}
	It follows from Lemma \ref{lem:marked-equivalence_in_generic_inflate_complexes} that $\mrdto$ is always a preorder in a marked inflate-complex, so $\meqv$ is always an equivalence relation.
\end{comm}

\begin{rmk}
	As a particular case of the last point of Lemma \ref{lem:marked-equivalence_in_generic_inflate_complexes}, if $v, w$ are round diagrams of the same dimension such that $v \cp{} w$ is defined, then $v \mrdto v'$ implies $v \cp{} w \mrdto v' \cp{} w$, and similarly $w \mrdto w'$ implies $v \cp{} w \mrdto v \cp{} w'$, that is, $\mrdto$ is always compatible with pasting in codimension 1.
\end{rmk}

\begin{lem} \label{lem:unitors_as_marked-reductions}
	Let $(X, A)$ be a marked inflate-complex, $n > 0$, let $u \in \gr{n}{\Rd X}$, and for each $\a \in \set{-, +}$, let $j^\a\colon u^\a \submol \bd{}{\a}u$ be a rewritable subdiagram.
	Then $u \subcp{j^+} \un u^+ \mrdto u \mrdto \un u^- \cpsub{j^-} u$
\end{lem}
\begin{proof}
	These are exhibited by the right unitor $\run{j^+}u$ and the left unitor $\lun{j^-}u$, respectively.
\end{proof}

\begin{dfn}[Marked-invertible round diagram]
	Let $e\colon u \rdto v$ be a round diagram in a marked inflate-complex.
	We say that $e$ is \emph{marked-invertible} if there exists a round diagram $e^*\colon v \rdto u$ such that $e \cp{} e^* \meqv \un u$ and $e^* \cp{} e \meqv \un v$.
	In this case, we call $e^*$ a \emph{weak inverse} of $e$.
\end{dfn}

\noindent 
The following is immediate from symmetry in the definition.

\begin{lem} \label{lem:marked_invertibility_is_symmetric}
	Let $e$ be a marked-invertible round diagram in a marked inflate-complex, and let $e^*$ be a weak inverse of $e$.
	Then $e^*$ is marked-invertible.
\end{lem}

\begin{lem} \label{lem:marked_equivalence_preserves_marked_invertibility}
	Let $(X, A)$ be a marked inflate-complex, let $e$, $h$ be parallel round diagrams, and suppose $e \meqv h$.
	Then $e$ is marked-invertible if and only if $h$ is marked-invertible.
\end{lem}
\begin{proof}
	Suppose $e\colon u \rdto v$ is marked-invertible and let $e^*$ be a weak inverse.
	Then, by Lemma \ref{lem:marked-equivalence_in_generic_inflate_complexes}, we have
	\[
		h \cp{} e^* \meqv e \cp{} e^* \meqv \un u, \quad \quad e^* \cp{} h \meqv e^* \cp{} e \meqv \un v
	\]
	so $h$ is marked-invertible and $e^*$ is also a weak inverse of $h$.
	The converse is symmetric.
\end{proof}

\begin{dfn}[Weak composite]
	Let $u$ be a round diagram of shape $U$ in a marked directed complex.
	A \emph{weak composite of $u$} is a cell $\mrg{u}$ of shape $\mrg{U}$ such that $u \meqv \mrg{u}$.
\end{dfn}

\begin{dfn}[\inftyn-category]
	A marked inflate-complex is an \emph{\inftyinf\nbd category} if it satisfies the following axioms.
	\begin{enumerate}
		\item \emph{(Weak composites)}.
			Every round diagram has a weak composite.
		\item \emph{(Completeness)}.
			Marked cells coincide with marked-invertible cells.
	\end{enumerate}
	For each $n \in \mathbb{N}$, an \inftyinf\nbd category is an \emph{\inftyn\nbd category} if, furthermore, every cell of dimension $> n$ is marked.
\end{dfn}

\begin{comm}
	Our use of ``\inftyn\nbd category''---what is meant to be a model-independent notion---for this specific model should be read informally, as a shorthand for ``\emph{complete marked inflate-complex with weak composites}''.
\end{comm}

\begin{comm}
	As discussed in \cite{barwick2020unicity, loubaton2024inductive}, while for finite $n$ the notion of \inftyn\nbd category is presumed unique, there are two or three possible inequivalent notions of \inftyinf\nbd category; our definition is presumed to model the \emph{inductive} notion.
	We believe that it should be possible to produce a model of walking coherent coinductive equivalences \cite{hadzihasanovic2025model} in marked directed complexes so that an appropriate localisation gives a coinductive model of \inftyinf\nbd categories, but this is not as straightforward as in the diagrammatic model \cite{chanavat2024model}, and we will not attempt to do so here.
\end{comm}

\noindent
In an \inftyn\nbd category, we may unambiguously speak of an \emph{equivalence} to refer either to a marked cell, or to a marked-invertible cell, since the two notions coincide.
We will also be interested in the following, slightly weaker notion.

\begin{dfn}[Essential \inftyn-category]
	A marked inflate-complex is an \emph{essential \inftyinf\nbd category} if it satisfies the following axioms.
	\begin{enumerate}
		\item \emph{(Weak composites)}.
			Every round diagram has a weak composite.
		\item \emph{(Essential completeness)}.
			Every marked cell is marked-invertible, and every marked-invertible cell is marked-equivalent to a marked cell.
	\end{enumerate}
	For each $n \in \mathbb{N}$, an essential \inftyinf\nbd category is an \emph{essential \inftyn\nbd category} if, furthermore, every cell of dimension $> n$ is marked.
\end{dfn}

\noindent
We will now show that, for each $n \in \Ninfty$, a $\Jn$-fibrant marked directed complex can be given a structure of \inftyn\nbd category.

\begin{lem} \label{lem:marked_equivalence_is_symmetric_in_Jhorn_fibrant}
	Let $(X, A)$ be a $\Jhorn$\nbd fibrant marked directed complex, let $u$, $v$ be parallel round diagrams in $X$, and suppose $u \mrdto v$.
	Then $u \meqv v$.
\end{lem}
\begin{proof}
	By Theorem \ref{thm:fibrants_admit_inflate_algebra_structure}, we can assume that $(X, A)$ is a marked inflate-complex.
	Let $e\colon u \mrdto v$ exhibit $u \mrdto v$.
	Then there is a marked horn classifying the equation $e \cp{} x \qeq \un u$.
	Because $(X, A)$ is $\Jhorn$\nbd fibrant, this admits a lax solution $h\colon e \cp{} e^* \celto \un u$, where, since $e$ and $\un u$ are both marked round diagrams, the cell $e^*$ can be taken to be marked.
	Then $e^*\colon v \mcelto u$ exhibits $v \mrdto u$.
\end{proof}

\begin{prop} \label{prop:fibrants_are_inftyn_categories}
	Let $n \in \Ninfty$ and let $(X, A)$ be fibrant in $\Mwn$.
	Then $(X, A)$ is the underlying marked directed complex of an \inftyn\nbd category.
\end{prop}
\begin{proof}
	By Theorem \ref{thm:fibrants_admit_inflate_algebra_structure}, $(X, A)$ admits a structure of marked inflate-complex; fix such a structure.
	Let $u$ be a round diagram in $X$.
	Then there is a marked horn classifying the equation $u \qeq x$ in the indeterminate $x$, and since $X$ is $\Jhorn$\nbd fibrant, this admits a lax solution $h\colon u \mrdto \mrg{u}$.
	By Lemma \ref{lem:marked_equivalence_is_symmetric_in_Jhorn_fibrant}, also $u \meqv \mrg{u}$.
	This proves that $(X, A)$ has weak composites.
	Moreover, if $u$ is a marked round diagram, then $\mrg{u}$ can be taken to be a marked cell.

	Next, let $e\colon u \celto v$ be a marked cell in $X$.
	Then there is a marked horn classifying the equation $e \cp{} x \qeq \un u$.
	By $\Jhorn$\nbd fibrancy, this has a lax solution $z\colon e \cp{} e^* \mrdto \un u$.
	This exhibits $e \cp{} e^* \mrdto \un u$, and by Lemma \ref{lem:marked_equivalence_is_symmetric_in_Jhorn_fibrant} $e \cp{} e^* \meqv \un u$.
	Dually, a lax solution to $x \cp{} e \qeq \un v$ produces a marked cell $h\colon e' \cp{} e \mrdto \un v$.
	Now, by Lemma \ref{lem:marked-equivalence_in_generic_inflate_complexes} and Lemma \ref{lem:marked_equivalence_is_symmetric_in_Jhorn_fibrant}, we have
	\[	
		e^* \meqv \un v \cp{} e^* \meqv e' \cp{} e \cp{} e^* \meqv e' \cp{} \un u \meqv e'
	\]
	hence $e^* \cp{} e \meqv e' \cp{} e \meqv \un v$.
	This proves that $e$ is marked-invertible.
	Conversely, suppose that $e$ is a marked-invertible cell and let $e^*$ be a weak inverse.	
	By the first part, we can find weak composites $\mrg{e^*} \meqv e^*$, $\mrg{\un{u}} \meqv \un{u}$, and $\mrg{\un{v}} \meqv \un{v}$, where the last two can be taken to be marked, and we have marked-equivalences
	\[
		e \cp{} \mrg{e^*} \meqv \mrg{\un u}, \quad\quad \mrg{\un v} \meqv \mrg{e^*} \cp{} e.
	\]
	Let $z$ and $h$ be the marked round diagrams exhibiting these; then $z$ and $h$ admit weak composites $\mrg{z}$ and $\mrg{h}$ which can be taken to be marked.
	Since $(X, A)$ is $\Jsat$\nbd fibrant, and $\mrg{z}, \mrg{h}, \mrg{\un{u}}$, and $\mrg{\un{v}}$ are all marked, it follows that $e$ and $\mrg{e^*}$ are also marked.
	We conclude that $(X, A)$ is also complete.

	Finally, if $n < \infty$, then $\Jn$\nbd fibrancy implies directly that every cell of dimension $> n$ is marked, and we conclude.
\end{proof}

\noindent
In the next section, we will prove the converse to this result.

\subsection{Equivalences in \inftyn-categories} \label{sec:equivalences}

\noindent
Fix $n \in \Ninfty$.
Our first aim in this section is to reprove enough of the results from \cite{chanavat2024equivalences}, which are relative to coinductively weakly invertible round diagrams in a diagrammatic set, in the case of marked-invertible round diagrams in an essential \inftyn\nbd category.
The following two results correspond to \cite[Corollary 2.14]{chanavat2024equivalences}.

\begin{lem} \label{lem:marked_equivalence_symmetric_in_ess_complete}
	Let $(X, A)$ be an essential \inftyn\nbd category, let $u$, $v$ be parallel round diagrams in $X$, and suppose $u \mrdto v$.
	Then $u \meqv v$.
\end{lem}
\begin{proof}
	Let $e\colon u \mrdto v$ be a witness of marked-reducibility.
	Then $e$ is marked, so by essential completeness, $e$ is marked-invertible; let $e^*\colon v \rdto u$ be a weak inverse of $e$.
	By Lemma \ref{lem:marked_invertibility_is_symmetric}, $e^*$ is marked-invertible, and by Lemma \ref{lem:marked_equivalence_preserves_marked_invertibility}, so is any weak composite $\mrg{e^*}$ of $e^*$.
	By essential completeness, $\mrg{e^*} \meqv e'$ for some marked cell $e'\colon v \mrdto u$, and we conclude that $u \meqv v$.
\end{proof}

\begin{lem} \label{lem:pastings_of_marked_invertible}
	Let $(X, A)$ be an essential \inftyn\nbd category, let $e$, $h$ be round diagrams in $X$ with $\dim e = \dim h$, and suppose $e$ is marked-invertible.
	Then
	\begin{enumerate}
		\item if a pasting $e \cpsub{} h$ is defined, then $e \cpsub{} h$ is marked-invertible if and only if $h$ is marked-invertible,
		\item if a pasting $h \subcp{} e$ is defined, then $h \subcp{} e$ is marked-invertible if and only if $h$ is marked-invertible.
	\end{enumerate}
\end{lem}
\begin{proof}
	Let $e\colon u \rdto v$, $h\colon v' \rdto w$ be the types of $e$ and $h$, let $v \submol v'$ be the rewritable subdiagram determining the pasting $e \cpsub{} h$, and fix a weak inverse $e^*\colon v \rdto u$ of $e$.
	Suppose $h$ is marked-invertible and let $h^*\colon w \rdto v'$ be a weak inverse.
	Then, letting $u' \eqdef \subs{v'}{u}{v}$, we have
	\begin{align*}
		(e \cpsub{} h) \cp{} (h^* \subcp{} e^*) & = e \cpsub{} (h \cp{} h^*) \subcp{} e^* 
		\meqv e \cpsub{} \un v' \subcp{} e^* \\
							& \meqv \un u' \cp{} (e \cpsub{} \un v' \subcp{} e^*) = ((\un u' \subcp{} e) \cp{} \un v') \subcp{} e^* \\
							& \meqv (\un u' \subcp{} e) \subcp{} e^* = \un u' \subcp{} (e \cp{} e^*) \\
							& \meqv \un u' \subcp{} \un u \meqv \un u',
	\end{align*}
	using various instances of Lemma \ref{lem:marked-equivalence_in_generic_inflate_complexes} and Lemma \ref{lem:unitors_as_marked-reductions} in conjunction with Lemma \ref{lem:marked_equivalence_symmetric_in_ess_complete}, as well as $\omega$\nbd categorical equations of pasting.
	More easily,
	\begin{align*}
		(h^* \subcp{} e^*) \cp{} (e \cpsub{} h) & = h^* \cp{} ((e^* \cp{} e) \cpsub{} h) \meqv h^* \cp{} (\un v \cpsub{} h) \\
							& \meqv h^* \cp{} h \meqv \un w.
	\end{align*}
	This proves that $e \cpsub{} h$ is marked-invertible with weak inverse $h^* \subcp{} e^*$.
	Conversely, suppose that $e \cpsub{} h$ is marked-invertible, and let $k\colon w \rdto u'$ be a weak inverse.
	Then
	\begin{align*}
		h \cp{} (k \subcp{} e) & \meqv \un v \cpsub{} (h \cp{} (k \subcp{} e)) \meqv (e^* \cp{} e) \cpsub{} (h \cp{} (k \subcp{} e)) \\
					 & = e^* \cpsub{} ((e \cpsub{} h) \cp{} k) \subcp{} e
					 \meqv e^* \cpsub{} \un u' \subcp{} e \\
					 & \meqv ((\un v' \subcp{} e^*) \cp{} \un u') \subcp{} e 
					 \meqv \un v' \subcp{} (e^* \cp{} e) \meqv \un v' \subcp{} \un v \meqv \un v',
	\end{align*}
	while $(k \subcp{} e) \cp{} h = k \cp{} (e \cpsub{} h) \meqv \un w$, which proves that $h$ is invertible with weak inverse $k \subcp{} e$.
	The case of $h \subcp{} e$ is dual.
\end{proof}

\begin{prop} \label{prop:round_diagrams_of_marked_inv_cells_are_marked-invertible}
	Let $(X, A)$ be an essential \inftyn\nbd category, $m > 0$, let $u \in \gr{m}{\Rd X}$, and suppose every $m$\nbd dimensional cell $v \submol u$ is marked-invertible.
	Then $u$ is marked-invertible.
\end{prop}
\begin{proof}
	Follows from Lemma \ref{lem:pastings_of_marked_invertible} by the proof of \cite[Proposition 2.15]{chanavat2024equivalences}.
\end{proof}

\begin{cor} \label{cor:marked_round_diagrams_are_marked_invertible}
	Let $(X, A)$ be an essential \inftyn\nbd category and let $u$ be a marked round diagram in $X$.
	Then $u$ is marked-invertible.
\end{cor}
\begin{proof}
	Follows from Proposition \ref{prop:round_diagrams_of_marked_inv_cells_are_marked-invertible} and essential completeness. 
\end{proof}

\begin{prop} \label{prop:marked_round_diagrams_have_marked_weak_composites}
	Let $(X, A)$ be an essential \inftyn\nbd category and let $u$ be a marked round diagram.
	Then $u$ admits a marked weak composite.
\end{prop}
\begin{proof}
	By Corollary \ref{cor:marked_round_diagrams_are_marked_invertible}, $u$ is marked-invertible, and by Lemma \ref{lem:marked_equivalence_preserves_marked_invertibility}, any weak composite $\mrg{u}$ of $u$ is also marked-invertible.
	By essential completeness, $\mrg{u}$ is marked-equivalent to a marked cell, which is also a weak composite.
\end{proof}

\begin{rmk}
	As a consequence of Proposition \ref{prop:marked_round_diagrams_have_marked_weak_composites}, if we have $u \meqv v$ in an essential \inftyn\nbd category, we may always assume that the marked-equivalence is witnessed by a marked \emph{cell} $h\colon u \mcelto v$.
\end{rmk}

\begin{prop} \label{prop:marked_invertible_have_marked_weak_inverses}
	Let $(X, A)$ be an essential \inftyn\nbd category and let $u$ be a marked round diagram.
	Then $u$ admits a marked weak inverse.
\end{prop}
\begin{proof}
	By Corollary \ref{cor:marked_round_diagrams_are_marked_invertible}, $u$ is marked-invertible.
	By Lemma \ref{lem:marked_invertibility_is_symmetric}, any weak inverse $u^*$ of $u$ is marked-invertible.
	Then any weak composite $\mrg{u^*}$ is also a weak inverse of $u$, and by essential completeness it can always be replaced with a marked cell.
\end{proof}

\noindent
To make the next results more readable, we introduce the following terminology.
Given an inflate-complex $X$ and a set $A \subseteq \Rd X$, we will say
\begin{itemize}
	\item a round diagram $u$ in $X$ is an \emph{$A$-round diagram} if, letting $m \eqdef \dim u$, every $m$\nbd dimensional cell $v \submol u$ is in $A$,
	\item a round diagram $e\colon u \rdto v$ in $X$ is \emph{$A$-invertible} if there exist $e^*\colon v \rdto u$ and round diagrams $z\colon e \cp{} e^* \rdto \un u$, $h\colon e^* \cp{} e \rdto \un v$, $z'\colon \un u \rdto e \cp{} e^*$, and $h'\colon \un v \rdto e^* \cp{} e$ such that $z, h, z', h' \in A$.
\end{itemize}
We let $A^+$ and $\I(A)$ denote the sets of $A$\nbd round diagrams and of $A$\nbd invertible diagrams, respectively.
Note that, for a marked inflate-complex $(X, A)$,
\begin{itemize}
	\item marked round diagrams are exactly $A$\nbd round diagrams,
	\item marked-invertible round diagrams are $A^+$\nbd invertible round diagrams.
\end{itemize}
Finally, let $\satur{A} \eqdef \I(A^+) \cap \cell X$.
Then completeness of $(X, A)$ is precisely the statement that $A = \satur{A}$, while essential completeness is the statement that $A \subseteq \satur{A}$ and for all $v \in \satur{A}$ there exists $u \in A$ such that $u \meqv v$.

\begin{lem} \label{lem:marked-invertible_iff_invertible_up_to_marked-invertible}
	Let $(X, A)$ be an essential \inftyn\nbd category and let $e$ be a round diagram in $X$.
	The following are equivalent:
	\begin{enumerate}[label=(\alph*)]
		\item $e$ is marked-invertible;
		\item $e$ is $\I(A^+)$-invertible;
		\item $e$ is $(\satur{A})^+$-invertible.
	\end{enumerate}
\end{lem}
\begin{proof}
	Let $e\colon u \rdto v$ be the type of $e$.
	Suppose $e$ is marked-invertible, let $e^*$ be a weak inverse, and let $z$, $h$, $z'$, $h'$ be witnesses of $e \cp{} e^* \meqv \un u$ and $e^* \cp{} e \meqv \un v$.
	By Corollary \ref{cor:marked_round_diagrams_are_marked_invertible}, $z$, $h$, $z'$, $h'$ are marked-invertible, so $e$ is $\I(A^+)$\nbd invertible.

	Next, suppose that $e$ is $\I(A^+)$-invertible, so there exist marked-invertible round diagrams $z$, $h$, $z'$, $h'$ as by the definition.
	Then by Lemma \ref{lem:marked_equivalence_preserves_marked_invertibility}, any weak composites $\mrg{z}$, $\mrg{h}$, $\mrg{z'}$, $\mrg{h'}$ are also marked-invertible, so $e$ is $\satur{A}$\nbd invertible, and \emph{a fortiori} $(\satur{A})^+$\nbd invertible.

	Finally, suppose that $e$ is $(\satur{A})^+$-invertible, so there exist round diagrams of types $z$, $h$, $z'$, $h'$ as by the definition, whose top-dimensional cells are marked-invertible.
	By Proposition \ref{prop:round_diagrams_of_marked_inv_cells_are_marked-invertible}, $z$, $h$, $z'$, $h'$ are marked-invertible, and so are any weak composites $\mrg{z}$, $\mrg{h}$, $\mrg{z'}$, $\mrg{h'}$.
	By essential completeness, they can be replaced with marked cells exhibiting $e \cp{} e^* \meqv \un u$ and $e^* \cp{} e \meqv \un v$.
\end{proof}

\begin{thm} \label{thm:naive_completion_is_completion}
	Let $(X, A)$ be an essential \inftyn-category.
	Then $(X, \satur{A})$ is an \inftyn-category.
\end{thm}
\begin{proof}
	By essential completeness, $A \subseteq \satur{A}$, so $A^+ \subseteq (\satur{A})^+$.
	\emph{A fortiori}, then, $(X, \satur{A})$ has weak composites, and every cell in $\satur{A}$, being $A^+$\nbd invertible, is also $(\satur{A})^+$\nbd invertible.
	The converse, that an $(\satur{A})^+$\nbd invertible cell is $A^+$\nbd invertible, is the content of Lemma \ref{lem:marked-invertible_iff_invertible_up_to_marked-invertible}.
\end{proof}

\noindent
We give another useful criterion for marked-invertibility of a round diagram in an essential \inftyn\nbd category.

\begin{lem} \label{lem:marked_invertibility_from_non-inverses}
	Let $(X, A)$ be an essential \inftyn\nbd category and $a\colon u \rdto v$ a round diagram.
	The following are equivalent:
	\begin{enumerate}[label=(\alph*)]
		\item there exist round diagrams $a_L\colon v \rdto u'$, $a_R\colon v' \rdto u$, $e\colon u \mrdto u'$, $h\colon v' \mrdto v$, as well as $z_L\colon a\cp{}a_L \mrdto e$, $z_R\colon h \mrdto a_R \cp{} a$;
		\item $a$ is marked-invertible.
	\end{enumerate}
	Moreover, under either of the equivalent conditions,
	\begin{enumerate}
		\item $a_L$ and $a_R$ are also marked-invertible,
		\item $a_L$, $a_R$, $e$, $h$, $z_L$, $z_R$ can be taken to be marked cells,
		\item $u'$ and $v'$ can be taken to be equal to $u$ and $v$, respectively.
	\end{enumerate}
\end{lem}
\begin{proof}
	One direction follows from the definition, taking $a_L, a_R \eqdef a^*$, $e \eqdef \un u$, $h \eqdef \un v$, and $z_L$ and $z_R$ to be witnesses of $a \cp{} a^* \mrdto \un u$ and $\un v \mrdto a^* \cp{} a$.
	This also implies that we can take $u' = u$ and $v' = v$.
	Conversely, let $e^*$ and $h^*$ be weak inverses of $e$ and $h$, respectively.
	Then $(e^* \cp{} a_R) \cp{} a \meqv \un u$ and $a \cp{} (a_L \cp{} h^*) \meqv \un v$, and by the standard argument already used in the proof of Proposition \ref{prop:fibrants_are_inftyn_categories}, $e^* \cp{} a_R \meqv a_L \cp{} h^*$, so $a$ is marked-invertible.
	Since, by Lemma \ref{lem:marked_equivalence_preserves_marked_invertibility}, $a \cp{} a_L$ and $a_R \cp{} a$ are also marked-invertible, by Lemma \ref{lem:pastings_of_marked_invertible} $a_L$ and $a_R$ are marked-invertible.
	Finally, by Proposition \ref{prop:marked_round_diagrams_have_marked_weak_composites}, we can always pass to weak composites while preserving the property of being marked.
\end{proof}

\noindent 
The following result, implied by the previous Lemma, will come useful later.

\begin{prop} \label{prop:morphisms_preserve_marked_invertibility}
	Let $(X, A)$ and $(Y, B)$ be essential \inftyn\nbd categories, let $f\colon (X, A) \to (Y, B)$ be a morphism of their underlying marked directed complexes, and let $e$ be a marked-invertible round diagram in $(X, A)$.
	Then $f(e)$ is marked-invertible in $(Y, B)$.
\end{prop}
\begin{proof}
	Immediate from the characterisation in Lemma \ref{lem:marked_invertibility_from_non-inverses} and the fact that morphisms send marked cells to marked cells.
\end{proof}

\noindent
Now, we can easily import the results of \cite[Section 3.2, Section 5]{chanavat2024equivalences}.

\begin{dfn}[Marked-equivalence of round contexts]
	Let $X$ be an essential \inftyn\nbd category and let $\fun{F}, \fun{G}\colon \Pd X(v, w) \to \Pd X(v', w')$ be round contexts.
	A family of marked round diagrams $\th a\colon \fun{F}a \mrdto \fun{G}a$ indexed by round diagrams $a\colon v \rdto w$ is a \emph{marked-equivalence from $\fun{F}$ to $\fun{G}$} if, for all round diagrams $a, b\colon v \rdto w$, there exists a marked-equivalence from ${\fun{F}_{a,b}-} \cp{} \th b$ to $\th a \cp{} {\fun{G}_{a,b}-}$ as round contexts $\Pd X(a, b) \to \Pd X(\fun{F}a, \fun{G}b)$.
\end{dfn}

\noindent 
We write $\th\colon \fun{F} \mrdto \fun{G}$ to indicate that $\th$ is a marked-equivalence from $\fun{F}$ to $\fun{G}$.

\begin{prop} \label{prop:properties_of_marked_equivalences_of_contexts}
	Let $(X, A)$ be an essential \inftyn\nbd category.
	Then:
	\begin{enumerate}
		\item if $\th\colon \fun{F} \mrdto \fun{G}$ and $\psi\colon \fun{G} \mrdto \fun{H}$ are marked-equivalences of round contexts, then the family
			\[
				(\th \cp{} \psi)a \eqdef \th a \cp {} \psi a\colon \fun{F} a \mrdto \fun{H} a
			\]
			determines a marked-equivalence $\th \cp{} \psi\colon \fun{F} \mrdto \fun{H}$;
		\item if $\fun{F}, \fun{G}\colon \Pd(v, w) \to \Pd(v', w')$ are parallel round contexts, $\th\colon \fun{F} \mrdto \fun{G}$ is a marked-equivalence, and $\fun{H}$ is a round context on $\Pd(v', w')$, then the family
			\[
				\fun{H} \th a \cp{} \un{(\fun{HG}a)}\colon \fun{HF}a \mrdto \fun{HG}a
			\]
			determines a marked-equivalence $\fun{H}\th \cp{} \un(\fun{HG})\colon \fun{HF} \mrdto \fun{HG}$;
		\item if $\fun{F}, \fun{G}$ are parallel round contexts on $\Pd X(v', w')$, $\th\colon \fun{F} \mrdto \fun{G}$ is a marked-equivalence, and $\fun{H}\colon \Pd X(v, w) \to \Pd X(v', w')$ is a round context, then the family		
		\[
			\th \fun{H}a\colon \fun{FH}a \mrdto \fun{GH}a
		\]
			determines a marked-equivalence $\th{\fun{H}}\colon \fun{FH} \mrdto \fun{FG}$;
		\item if $\th\colon \fun{F} \mrdto \fun{G}$ is a marked-equivalence of round contexts, then any choice of componentwise marked weak inverses
			\[
				(\th a)^*\colon \fun{G} a \mrdto \fun{F} a
			\]
			determines a marked-equivalence $\th^*\colon \fun{G} \mrdto \fun{F}$;
		\item for all parallel round diagrams $v, w$, letting $-$ be the identity context on $\Pd(v, w)$, the family of units
			\[
				\un a \colon a \mrdto a
			\]
			determines a marked-equivalence $\un\colon - \mrdto -$;
		\item for all parallel round diagrams $v, w$ and rewritable $\iota\colon u \submol v$, the family of left unitors
			\[
				\lun{\iota}a\colon a \mrdto \un u \cpsub{\iota} a
			\]
			determines a marked-equivalence $\lun{\iota}\colon - \mrdto \un u \cpsub{\iota} -$;
		\item for all parallel round diagrams $v, w$ and rewritable $j\colon u \submol w$, the family of right unitors
			\[
				\run{j}a\colon a \subcp{j} \un u \mrdto a
			\]
			determines a marked-equivalence $\run{j}\colon - \subcp{j} \un u \mrdto -$;
		\item for all round contexts $\fun{F}$, rewritable context subdiagrams $\iota\colon u \submol \fun{F}$, and marked round diagrams $h\colon u \mrdto v$, the family
			\[
				\un{(\fun{F}a)} \subcp{\iota_a} h\colon \fun{F}a \rdto \subs{\fun{F}a}{v}{\iota_a(u)}
			\]
			determines a marked-equivalence $\un{(\fun{F})} \subcp{\iota} h\colon \fun{F} \mrdto \subs{\fun{F}}{v}{\iota_a(u)}$. 
	\end{enumerate}
\end{prop}
\begin{proof}
	The proof of \cite[Theorem 3.22]{chanavat2024equivalences} goes through unmodified after replacing weakly invertible with marked round diagrams, and natural equivalences with marked-equivalences of round contexts.
\end{proof}

\noindent
Given parallel round contexts $\fun{F}$, $\fun{G}$ in an essential \inftyn\nbd category, we write $\fun{F} \meqv \fun{G}$ if there exists a marked-equivalence $\th\colon \fun{F} \mrdto \fun{G}$.

\begin{cor} \label{cor:equivalence_relation_on_marked_contexts}
	Let $(X, A)$ be an essential \inftyn\nbd category.
	Then the relation $\meqv$ on round contexts in $X$ is
	\begin{enumerate}
		\item an equivalence relation,
		\item a congruence with respect to composition of round contexts,
		\item compatible with the relation $\meqv$ on round diagrams of the same dimension, that is, if $v \meqv w$ and $\fun{F} \meqv \fun{G}$ with $\dim{v} = \dim{\fun{F}}$, then
			\begin{itemize}
				\item if $v \cp{} {\fun{F}-}$ is defined, then $v \cp{} {\fun{F}-} \meqv w \cp{} {\fun{G}-}$,
				\item if ${\fun{F}-} \cp{} v$ is defined, then ${\fun{F}-} \cp{} v \meqv {\fun{G}-} \cp{} w$.
			\end{itemize}
	\end{enumerate}
\end{cor}
\begin{proof}
	Same as \cite[Proposition 3.24]{chanavat2024equivalences}.
\end{proof}

\begin{dfn}[Marked-invertible round context]
	Let $(X, A)$ be an essential \inftyn\nbd category and let $\fun{E}\colon \Pd(v, w) \to \Pd(v', w')$ be a round context in $X$.
	We say that $\fun{E}$ is \emph{marked-invertible} if there exists a round context $\fun{E}^*\colon \Pd(v', w') \to \Pd(v, w)$ such that $\fun{E}^*\fun{E} \meqv -$ and $\fun{E}\fun{E}^* \meqv -$.
	In this case, we call $\fun{E}^*$ a \emph{weak inverse} of the round context $\fun{E}$.
\end{dfn}

\begin{lem} \label{lem:contexts_of_marked-invertible_are_marked-invertible}
	Let $(X, A)$ be an essential \inftyn\nbd category and let $\fun{E}$ be an $\I(A^+)$\nbd context.
	Then $\fun{E}$ is marked-invertible and admits an $\I(A^+)$\nbd context $\fun{E}^*$ as its weak inverse.
\end{lem}
\begin{proof}
	The proof of \cite[Theorem 5.22]{chanavat2024equivalences} goes through unmodified after replacing weakly invertible with marked-invertible round diagrams, and natural equivalences with marked-equivalences of round contexts.
\end{proof}

\begin{prop} \label{prop:solutions_of_equations}
	Let $(X, A)$ be an essential \inftyn\nbd category, let $\fun{E}$ be an $\I(A^+)$\nbd context $\Pd(v, w) \to \Pd(v', w')$, and $b\colon v' \rdto w'$ a round diagram.
	Then
	\begin{enumerate}
		\item there exists $a\colon v \rdto w$ such that $\fun{E}a \meqv b$,
		\item if $b$ is marked-invertible, then $a$ is marked-invertible,
		\item $a$ is weakly unique, in the sense that if $a'\colon v \rdto w$ is another round diagram such that $\fun{E}a' \meqv b$, then $a \meqv a'$.
	\end{enumerate}
\end{prop}
\begin{proof}
	The existence and weak uniqueness are proved using Lemma \ref{lem:contexts_of_marked-invertible_are_marked-invertible} as in \cite[Lemma 5.10]{chanavat2024equivalences}, letting $a \eqdef \fun{E}^*b$ for a weak inverse $\fun{E}^*$ of $\fun{E}$.
	If $b$ is marked-invertible, then one shows that $\fun{E}^*b$ is marked-invertible by induction on the construction of the $\I(A^+)$\nbd context $\fun{E}^*$, as in 
	Proposition \ref{prop:round_diagrams_of_marked_inv_cells_are_marked-invertible}.
\end{proof}

\begin{lem} \label{lem:essential_inftyn_cats_are_Jhorn_fibrant}
	Let $(X, A)$ be an essential \inftyn\nbd category.
	Then $(X, A)$ is $\Jhorn$\nbd fibrant.
\end{lem}
\begin{proof}
	Let $\lambda^x_U$ be a marked horn, and suppose without loss of generality that $x \in \faces{}{-}U$.
	A morphism from the domain of $\lambda^x_U$ to $(X, A)$ classifies an equation $\fun{E}x \qeq b$ in the indeterminate $x$, where $\fun{E}$ is an $A$\nbd context.
	By essential completeness, $\fun{E}$ is also an $\I(A^+)$\nbd context, so by Proposition \ref{prop:solutions_of_equations}, this has a solution $h\colon \fun{E}a \mrdto b$; moreover, if $b$ is a marked round diagram, then $a$ is marked-invertible.
	Let $\mrg{a}$ be a weak composite of $a$, witnessed by $k\colon \mrg{a} \mrdto a$; if $a$ is marked-invertible, then $\mrg{a}$ is marked-invertible, so by essential completeness, we may take $\mrg{a}$ to be marked.
	Then, letting $\iota\colon a \submol \fun{E}a$ be the evident subdiagram, $k \cpsub{\iota} h$ is a marked round diagram, which by Proposition \ref{prop:marked_round_diagrams_have_marked_weak_composites} admits a marked weak composite; this is a filler for the marked horn.
\end{proof}

\noindent
We can now characterise the fibrant objects in $\Mwn$: they are precisely the \inftyn-categories, up to a choice of $\Infl$\nbd algebra structure.

\begin{thm} \label{thm:characterisation_of_fibrants}
	Let $(X, A)$ be a marked directed complex.
	The following are equivalent:
	\begin{enumerate}[label=(\alph*)]
		\item $(X, A)$ is fibrant in $\Mwn$;
		\item $(X, A)$ is the underlying marked directed complex of an \inftyn\nbd category.
	\end{enumerate}
\end{thm}
\begin{proof}
	One direction is Proposition \ref{prop:fibrants_are_inftyn_categories}.
	For the converse, suppose $(X, A)$ is an \inftyn\nbd category. 
	By Lemma \ref{lem:essential_inftyn_cats_are_Jhorn_fibrant}, $(X, A)$ is $\Jhorn$\nbd fibrant.
	Moreover, by definition all cells of dimension $> n$ in $X$ are marked, so it suffices to show that $(X, A)$ is $\Jsat$\nbd fibrant.
	This is equivalent to the following property: if $a$, $a_L$, $a_R$, $e$, $h$ are cells of the same dimension such that $a_R \cp{} a \cp{} a_L$ is defined, $a \cp{} a_L \meqv e$, $a_R \cp{} a \meqv h$, and both $e$ and $h$ are marked, then $a$, $a_L$, $a_R$ are all marked.
	From Lemma \ref{lem:marked_invertibility_from_non-inverses}, we know that, under these conditions, $a$, $a_L$, and $a_R$ are all marked-invertible.
	By completeness of $(X, A)$, $a$, $a_L$, and $a_R$ are all marked.
\end{proof}

\begin{cor} \label{cor:naive_saturation_fibrant_replacement}
	Let $(X, A)$ be an essential \inftyn\nbd category.
	Then the marking $(X, A) \acof (X, \satur{A})$ is a fibrant replacement in $\Mwn$.
\end{cor}
\begin{proof}
	The marking is an acyclic cofibration because it can be constructed as a transfinite composition of pushouts along saturations.
	We conclude by Theorem \ref{thm:naive_completion_is_completion} combined with Theorem 
	\ref{thm:characterisation_of_fibrants}.
\end{proof}

\noindent
The characterisation of Theorem \ref{thm:characterisation_of_fibrants}, coupled with Lemma 
\ref{lem:marked_invertibility_from_non-inverses}, allows us to give a notion of marked-invertible round diagram in a fibrant marked directed complex, without reference to an $\mInfl$\nbd algebra structure.

\begin{dfn}[Marked-invertible round diagram, without units]
	Let $a\colon u \rdto v$ be a round diagram in an $\Mwn$\nbd fibrant marked directed complex.
	We say that $a$ is \emph{marked-invertible} if it satisfies the first condition of Lemma 
	\ref{lem:marked_invertibility_from_non-inverses}.
\end{dfn}

\noindent
This shows that \emph{in fibrants} the notion of marked-invertibility is independent of any choice of an $\mInfl$\nbd algebra structure.

\subsection{Equivalences of \inftyn-categories} \label{sec:equivalencesof}

\noindent
Throughout this section, we fix $n \in \Ninfty$.

\begin{dfn}[Functor of \inftyn-categories]
	Let $(X, A)$, $(Y, B)$ be \inftyn\nbd categories.
	A \emph{functor} $f\colon (X, A) \to (Y, B)$ is a morphism of their underlying marked directed complexes.
\end{dfn}

\begin{comm}
	By Theorem \ref{thm:characterisation_of_fibrants} and the fact that all objects are cofibrant in $\Mwn$, the category of \inftyn\nbd categories and functors is equivalent to both $\fibs{\Mwn}$ and $\bifs{\Mwn}$.
\end{comm}

\noindent
A functor is \emph{not} required to respect the $\Infl$-algebra structure.
However, it still preserves units up to marked-equivalence.

\begin{lem} \label{lem:functors_preserve_marked_equivalence}
	Let $f\colon (X, A) \to (Y, B)$ be a functor of \inftyn\nbd categories, let $u, v \in \Rd X$ be parallel, and suppose $u \meqv v$.
	Then $f(u) \meqv f(v)$.
\end{lem}
\begin{proof}
	Let $h\colon u \mrdto v$ be a witness.
	Since $f$ sends marked cells to marked cells, $f(h)\colon f(u) \rdto f(v)$ is a marked round diagram, hence $f(u) \meqv f(v)$.
\end{proof}

\begin{lem} \label{lem:units_are_marked_idempotents}
	Let $(X, A)$ be an \inftyn\nbd category and let $e\colon u \mrdto u$ be a marked round diagram.
	The following are equivalent:
	\begin{enumerate}[label=(\alph*)]
		\item $e \meqv \un u$;
		\item $e$ is a weak idempotent, that is, $e \meqv e \cp{} e$.
	\end{enumerate}
\end{lem}
\begin{proof}
	We have $\un u \meqv \un u \cp{} \un u$, so if $e \meqv \un u$, by compatibility of marked-equivalence with pasting in codimension 1, $e \meqv e \cp{} e$.
	Conversely, by Corollary \ref{cor:marked_round_diagrams_are_marked_invertible}, $e$ is marked-invertible, so let $e^*$ be a weak inverse.
	Then, if $e$ is a weak idempotent, $\un u \meqv e^* \cp{} e \meqv e^* \cp{} e \cp{} e \meqv \un u \cp{} e \meqv e$.
\end{proof}

\begin{prop} \label{prop:functors_weakly_preserve_units}
	Let $f\colon (X, A) \to (Y, B)$ be a functor of \inftyn\nbd categories and let $u \in \Rd X$.
	Then $f(\un u) \meqv \un f(u)$.
\end{prop}
\begin{proof}
	By Lemma \ref{lem:functors_preserve_marked_equivalence}, we have $f(\un u) \meqv f(\un u \cp{} \un u) = f(\un u) \cp{} f(\un u)$, and we conclude by Lemma \ref{lem:units_are_marked_idempotents}.
\end{proof}

\noindent
The aim for the rest of this section is to characterise the equivalences between $\infty$\nbd categories as being exactly those that are ``acyclic fibrations up to $\meqv$''.
Most of the development is formally analogous to \cite[Section 4.3]{chanavat2024model} which, in turn, closely followed \cite[Chapter 20]{ara2025polygraphs}.

\begin{dfn}[Essential acyclic fibration]
	Let $f\colon (X, A) \to (Y, B)$ be a functor of \inftyn\nbd categories.
	We say that $f$ is an \emph{essential acyclic fibration} if
	\begin{enumerate}
		\item for all $v \in \gr{0}{\cell Y}$, there exists $u \in \gr{0}{\cell X}$ such that $v \meqv f(u)$,
		\item for all $n > 0$, parallel pairs $u^-, u^+ \in \gr{n-1}{\Rd X}$, and cells $v\colon f(u^-) \celto f(u^+)$ in $Y$, there exists a cell $u\colon u^- \celto u^+$ such that $v \meqv f(u)$, and if $v$ is marked, then $u$ is marked.
	\end{enumerate}
\end{dfn}

\begin{lem} \label{lem:ess_surjective_reflects_meqv}
	Let $f\colon (X, A) \to (Y, B)$ be an essential acyclic fibration of \inftyn\nbd categories, let $u, v \in \Rd X$ be parallel, and suppose $f(u) \meqv f(v)$.
	Then $u \meqv v$.
\end{lem}
\begin{proof}
	Let $h\colon f(u) \mrdto f(v)$ be a marked cell witnessing the equivalence.
	Then, by definition, there exists $e\colon u \mrdto v$ such that $h \meqv f(e)$.
\end{proof}

\begin{lem} \label{lem:acyclic_fibration_is_ess_surjective}
	Let $f\colon (X, A) \to (Y, B)$ be an acyclic fibration of \inftyn\nbd categories.
	Then $f$ is an essential acyclic fibration.
\end{lem}
\begin{proof}
	The right lifting property against the boundary inclusions $\flatm{\bdmap_U}$ and the markings $\m_U$ implies that the conditions of an essential acyclic fibration hold up to equality rather than marked-equivalence, and we conclude by reflexivity of marked-equivalence.
\end{proof}

\begin{lem} \label{lem:ess_surjective_2_out_of_3_simple_cases}
	Let $f\colon (X, A) \to (Y, B)$ and $g\colon (Y, B) \to (Z, C)$ be functors of \inftyn\nbd categories, and suppose $g$ is an essential acyclic fibration.
	Then $gf$ is an essential acyclic fibration if and only if $f$ is an essential acyclic fibration.
\end{lem}
\begin{proof}
	Essentially the same as \cite[Proposition 20.1.15 and 20.1.17]{ara2025polygraphs}.
\end{proof}

\noindent
Next, let $(\fun{P}, (\pi^-, \pi^+))$ be the functorial path object dual to the functorial cylinder $(\fun{I}, (\iota^-, \iota^+))$.
Given a marked directed complex $(X, A)$ and a cell $\gamma$ of shape $U$ in $\fun{P}(X, A)$, we write $\gamma\colon u^- \cyl u^+$, where for each $\a \in \set{-, +}$, $u^\a$ is the cell in $X$ classified by $\pi^\a_{(X, A)}\gamma$.

\begin{comm}
	A cell $\gamma\colon u^- \cyl u^+$ is a morphism $\marr \gray \flatm{U} \to (X, A)$, that is, a cell of shape $I \gray U$ in $X$ whose restrictions to $\clset{(1, x)}$ are marked for all $x \in U$.
\end{comm}

\begin{lem} \label{lem:projections_of_path_object_are_acycofs}
	Let $(X, A)$ be an \inftyn\nbd category and let $\a \in \set{+, -}$.
	Then $\pi^\a_{(X, A)}\colon \fun{P}(X, A) \to (X, A)$ is an acyclic fibration.
\end{lem}
\begin{proof}
	By Theorem \ref{thm:characterisation_of_fibrants}, $(X, A)$ is fibrant, so the statement immediately follows from Proposition \ref{prop:weak_cisinski_olschok}.
\end{proof}

\begin{comm}
	It follows that, when $(X, A)$ is an \inftyn\nbd category, then $\fun{P}(X, A)$ is fibrant and can be endowed with an $\mInfl$\nbd algebra structure making it an \inftyn\nbd category.
	Then, by Lemma \ref{lem:acyclic_fibration_is_ess_surjective}, $\pi^\a_{(X, A)}$ becomes an essential acyclic fibration of \inftyn\nbd categories.
\end{comm}

\noindent
The following, combined with Lemma \ref{lem:projections_of_path_object_are_acycofs}, is an analogue of the ``transport lemma'' \cite[Lemma 4.29]{chanavat2024model}, \cite[20.3.5]{ara2025polygraphs}.

\begin{lem} \label{lem:transport_lemma}
	Let $(X, A)$ be an \inftyn\nbd category and let $\gamma\colon u \cyl v$ and $\delta\colon u' \cyl v'$ be parallel cells in $\fun{P}(X, A)$.
	Then 
	\begin{enumerate}
		\item $u \meqv u'$ if and only if $v \meqv v'$,
		\item $u$ is marked if and only if $v$ is marked.
	\end{enumerate}
\end{lem}
\begin{proof}
	Let $\gamma, \delta$ also denote the cells classified by the transpose morphisms $\marr \gray \flatm{U} \to (X, A)$, and suppose $u \meqv u'$, witnessed by $e\colon u \mrdto u'$.
	Then $\iota \eqdef (0^-, -)\colon U \submol \bd{}{-}(\arr \gray U)$ determines a rewritable subdiagram $\iota\colon u' \submol \bd{}{-}\delta$, so we have a marked round diagram $e \cpsub{\iota} \delta$ in $X$, such that any marked weak composite determines a cell $\gamma'\colon u \cyl v'$.
	Now, removing $(1, \top)$ and $(0^+, \top)$ from $\marr \gray \flatm{U}$ determines a marked horn on which the restrictions of $\gamma$ and $\gamma'$ are equal, defining an equation $\fun{E}x \qeq b$ where $\fun{E}$ is an $A$\nbd context.
	Then $v$ and $v'$ are both solutions to this equation, witnessed by $\gamma$ and $\gamma'$, and we conclude by Proposition \ref{prop:solutions_of_equations}.	
	Next, suppose that $u$ is marked.
	Then, by the same argument that $\gamma$ is a filler for a marked horn, combined with Proposition \ref{prop:solutions_of_equations}, $v$ is marked-invertible, so by completeness $v$ is marked.
	The converse is dual in both cases.
\end{proof}

\noindent
Now, suppose $(X, A)$ is a marked inflate-complex.
Then there is a morphism $\tau_{(X, A)}\colon \fun{I}(X, A) \to (X, A)$ of marked directed complexes (\emph{not} a morphism of marked inflate-complexes) defined by
\[
	0^\a \gray u \mapsto u, \quad \quad 1 \gray u \mapsto u\tau_\varnothing
\]
for each $\a \in \set{-, +}$ and each cell $u\colon U \to X$, where $\tau_\varnothing\colon \arr \pcyl{\varnothing} U \equiv \arr \gray U \surj U$ is a cylindrical collapse. 
Moreover, by construction, $\tau_{(X, A)}$ is a retraction of $\iota^\a_{(X, A)}$ for each $\a \in \set{-, +}$.
We let
\[
	\nu_{(X, A)}\colon (X, A) \to \fun{P}(X, A)
\]
denote the transpose of $\tau_{(X, A)}$ through the adjunction $\fun{I} \dashv \fun{P}$.
By duality, $\nu_{(X, A)}$ is a section of $\pi^\a_{(X, A)}$ for each $\a \in \set{-, +}$.

\begin{lem} \label{lem:identity_path_is_ess_surjective}
	Let $(X, A)$ be an \inftyn\nbd category and endow $\fun{P}(X, A)$ with an $\mInfl$\nbd algebra structure.
	Then $\nu_{(X, A)}\colon (X, A) \to \fun{P}(X, A)$ is an essential acyclic fibration of \inftyn\nbd categories.
\end{lem}
\begin{proof}
	For each $\a \in \set{-, +}$, we have $\pi^\a_{(X, A)}\nu_{(X, A)} = \idd{(X, A)}$, and since both $\pi^\a_{(X, A)}$ and $\idd{(X, A)}$ are essential acyclic fibrations, by Lemma 
	\ref{lem:ess_surjective_2_out_of_3_simple_cases} we conclude that $\nu_{(X, A)}$ is an essential acyclic fibration.
\end{proof}

\begin{dfn}[Mapping path space]
	Let $(X, A)$, $(Y, B)$ be \inftyn\nbd categories and let $f\colon (X, A) \to (Y, B)$ be a functor.
	The \emph{mapping path space} of $f$ is the pullback
	\[\begin{tikzcd}
	{\fun{P}_f} & {\fun{P}(Y, B)} \\
	{(X, A)} & {(Y, B)}
	\arrow["{p_2}", from=1-1, to=1-2]
	\arrow["{p_1}", two heads, from=1-1, to=2-1]
	\arrow["\lrcorner"{anchor=center, pos=0.125}, draw=none, from=1-1, to=2-2]
	\arrow["{\pi^-_{(Y, B)}}", two heads, from=1-2, to=2-2]
	\arrow["f", from=2-1, to=2-2]
\end{tikzcd}\]
	in $\mdCpx$.
	Explicitly, cells in $\fun{P}_f$ are pairs $(u, \gamma\colon f(u) \cyl v)$ of a cell $u$ in $(X, A)$ and a cell $\gamma\colon f(u) \cyl v$ in $\fun{P}(Y, B)$.
\end{dfn}

\noindent 
Since $\pi^-_{(Y, B)}$ is an acyclic fibration and $(X, A)$ is fibrant, $p_1\colon \fun{P}_f \to (X, A)$ is an acyclic fibration, and $\fun{P}_f$ can be endowed with an $\mInfl$\nbd algebra structure, making it an \inftyn\nbd category.
With such a structure,
\[
	p_f \eqdef \pi^+_{(Y, B)}p_2\colon \fun{P}_f \to (Y, B)
\]
becomes a functor of \inftyn\nbd categories.
Moreover, letting $\iota_f$ be the functor universally determined in the commutative diagram
\[\begin{tikzcd}
	& {(Y, B)} \\
	{(X, A)} & {\fun{P}_f} & {\fun{P}(Y, B)} \\
	& {(X, A)} & {(Y, B)},
	\arrow["{\nu_{(Y, B)}}", curve={height=-6pt}, from=1-2, to=2-3]
	\arrow["f", curve={height=-6pt}, from=2-1, to=1-2]
	\arrow["{\iota_f}", dashed, from=2-1, to=2-2]
	\arrow[curve={height=12pt}, equals, from=2-1, to=3-2]
	\arrow["{p_2}", from=2-2, to=2-3]
	\arrow["{p_1}", two heads, from=2-2, to=3-2]
	\arrow["\lrcorner"{anchor=center, pos=0.125}, draw=none, from=2-2, to=3-3]
	\arrow["{\pi^-_{(Y, B)}}", two heads, from=2-3, to=3-3]
	\arrow["f", from=3-2, to=3-3]
\end{tikzcd}\]
we have $p_f\iota_f = \pi^+_{(Y, B)}p_2\iota_f = \pi^+_{(Y, B)}\nu_{(Y, B)}f = f$.
At this point, we can proceed in complete analogy with \cite[Section 4.3]{chanavat2024model} and \cite[Section 20.3]{ara2025polygraphs}.

\begin{lem} \label{lem:ess_surjective_iff_pf_acyclic_fibration}
	Let $f\colon (X, A) \to (Y, B)$ be a functor of \inftyn\nbd categories.
	The following are equivalent:
	\begin{enumerate}[label=(\alph*)]
		\item $f$ is an essential acyclic fibration;
		\item $p_f\colon \fun{P}_f \to (Y, B)$ is an acyclic fibration.
	\end{enumerate}
\end{lem}
\begin{proof}
	Same as \cite[Proposition 20.3.10]{ara2025polygraphs}, additionally using the second point of Lemma \ref{lem:transport_lemma} for the lifting property against markings.
\end{proof}

\begin{lem} \label{lem:ess_surjective_2_out_of_3}
	Let $f\colon (X, A) \to (Y, B)$ and $g\colon (Y, B) \to (Z, C)$ be functors of \inftyn\nbd categories, and suppose $f$ and $gf$ are essential acyclic fibrations.
	Then $g$ is an essential acyclic fibration.
\end{lem}
\begin{proof}
	Same as \cite[Theorem 20.3.11]{ara2025polygraphs}.
\end{proof}

\begin{thm} \label{thm:characterisation_of_equivalences}
	Let $f\colon (X, A) \to (Y, B)$ be a functor of \inftyn\nbd categories.
	The following are equivalent:
	\begin{enumerate}[label=(\alph*)]
		\item $f$ is an essential acyclic fibration;
		\item $f$ is an equivalence in $\Mwn$.
	\end{enumerate}
\end{thm}
\begin{proof}
	Since $(X, A)$ and $(Y, B)$ are bifibrant objects in $\Mwn$, by 
	Proposition \ref{prop:characterisation_of_homotopy_category} it suffices to prove that $f$ is an essential acyclic fibration if and only if it is a homotopy equivalence with respect to the cylinders functorially determined by $(\fun{I}, (\iota^-, \iota^+))$, which follows from the same proof as \cite[Theorem 4.37]{chanavat2024model}.
\end{proof}

\noindent
From now on, we can unambiguously call an essential acyclic fibration an \emph{equivalence of \inftyn\nbd categories}.

\subsection{Homotopy hypothesis} \label{sec:homohyp}

\noindent
The aim of this section is to prove the homotopy hypothesis for our model of weak $(\infty, 0)$\nbd categories.
Since our directed complexes are closely related to Henry's regular polygraphs, we will be able to do so by closely following the proofs of \cite[Section 6.3]{henry2018regular}, once we formalise the intuitive fact that, in the case $n = 0$, we can do away with the marked structure.

The endofunctor \( \arr \gray - \colon \dCpx \to \dCpx \), together with the natural transformations \( \iota^\a = (0^\a \gray \idd{}) \lambda \) for each \( \a \in \set{-, +} \), equip \( \dCpx \) with a functorial cylinder.
We denote again by $\Jhorn$ the set of ``unmarked'' horns of atoms, that is, inclusions
\begin{equation*}
	\lambda^x_U \colon \Lambda^x_U \incl U,
\end{equation*}
where \( U \) is an atom with greatest element \( \top \), \( x \in \faces{}{\a} U \) for some \( \a \in \set{-, +} \), and \( \Lambda^x_U = U \setminus \set{\top, x} \).

\begin{lem} \label{lem:0_model_structure_directed_cplx}
	The triple \( \Nw_0 \eqdef (\dCpx, \ICof{\Ibd}, \JFib{\Jhorn}) \) is a weak model structure on \( \dCpx \) such that the adjoint functors \( \Um \dashv \sharpm{(-)} \) restrict to functors 
	\begin{equation*}
		\Um \colon \cofs{\Mw_0} \to \cofs{\Nw_0}, \quad\quad
		\sharpm{(-)} \colon \fibs{\Nw_0} \to \fibs{\Mw_0},
	\end{equation*}
	exhibiting a Quillen adjunction between $\Mw_0$ and $\Nw_0$.
\end{lem}
\begin{proof}
	First of all, $\dCpx$ is locally presentable.
	Since the functor \( \Um \) is strict monoidal and left adjoint, it commutes with the functorial cylinders on \( \mdCpx \) and \( \dCpx \), so it preserves pushout-products with their components.
	Since markings have underlying isomorphisms, we see that \( \ICof{\Ibd} = \ICof{\Um(I)} \) and, similarly, \( \JFib{\Um(\Jhorn)} = \JFib{\Jhorn} \), hence the verification of the conditions of Proposition \ref{prop:weak_cisinski_olschok} for \( \Mw_0 \) implies the verification of those same conditions for \( \Nw_0 \).
	For the same reason, it is immediate that \( \Um \) and \( \sharpm{(-)} \) form a Quillen adjunction between \( \Mw_0 \) and \( \Nw_0 \).
\end{proof}

\begin{prop} \label{prop:0_model_structure_quillen_equivalence}
	The Quillen adjunction \( (\Um, \sharpm{(-)}) \) is a Quillen equivalence between \( \Mw_0 \) and \( \Nw_0 \).
\end{prop}
\begin{proof}
	Let \( (X, A) \) be a marked directed complex and let \( i \colon (X, A) \acof (Y, B) \) be a fibrant replacement in \( \Mw_0 \).
	By Theorem \ref{thm:characterisation_of_fibrants}, \( (Y, B) \) is the underlying marked directed complex of an \( (\infty, 0) \)\nbd category, hence \( (Y, B) = \sharpm{Y} \).
	Since \( \Um \) is left Quillen, \( \Um(i) \colon X \acof Y \) is an acyclic cofibration, and since \( \sharpm{Y} \) is fibrant, it has the right lifting property against \( \Jhorn \), so by adjunction, \( Y \) has the right lifting property against \( \Um(\Jhorn) \), that is, \( Y \) is fibrant in \( \Nw_0 \).
	This shows that \( \Um(i) \) is a fibrant replacement of \( \Um(X, A) \) in \( \Nw_0 \).
	By construction, its transpose is \( i \), which is an equivalence.
	This proves one of the conditions of Proposition \ref{prop:quillen_equivalence_criterion}.
	Since there is no need for cofibrant replacement in \( \Mw_0 \) and the counit of the adjunction \( \Um \dashv \sharpm{(-)} \) is the identity, the other condition immediately holds.
	This concludes the proof. 
\end{proof}

\begin{lem} \label{lem:submol_inclusions_are_acyclic}
	Let \( \iota \colon V \submol U \) be a submolecule inclusion.
	Then, in \( \Nw_0 \),
	\begin{enumerate}
		\item \( \iota \) is an acyclic cofibration,
		\item if \( U \) is round and \( \iota \) is rewritable, then \( \bd{}{} U \incl U \setminus \inter{V} \) is an acyclic cofibration. 
	\end{enumerate}
\end{lem}
\begin{proof}
	Same as \cite[Lemma 5.3]{chanavat2024model}, using Lemma \ref{lem:cube_lemma}.
\end{proof}

\noindent
By \cite[Proposition 9.2.14]{hadzihasanovic2024combinatorics}, the full subcategory of \( \atom_\E \) spanned by the oriented simplices
\begin{equation*}
	\simplex{k} \eqdef \underbrace{1 \join \ldots \join 1}_{(k + 1) \text{ times}},\quad k \in \mathbb{N},
\end{equation*}
can be identified with the semi-simplex category \( \ssimcat \), that is, the wide subcategory of the simplex category on the injective maps only.
We denote by \( \ssSet \) the category of presheaves on \( \ssimcat \).
Then, we have a presheaf extension-restriction adjunction
\[\begin{tikzcd}
	\ssSet && \dCpx
	\arrow[""{name=0, anchor=center, inner sep=0}, "{\Fdelta}", curve={height=-12pt}, from=1-1, to=1-3]
	\arrow[""{name=1, anchor=center, inner sep=0}, "{-_\Delta}", curve={height=-12pt}, from=1-3, to=1-1]
	\arrow["\dashv"{anchor=center, rotate=-90}, draw=none, from=0, to=1]
\end{tikzcd}\]
such that the left adjoint is full and faithful, and identifies \( \ssSet \) with the full subcategory of \( \dCpx \) on presheaves \( X \) on \( \atom_\E \) with the property that if \( X(U) \) is not empty, then \( U = \simplex{k} \) for some \( k \geq 0 \).

Next, by \cite[Lemma 10.1.15]{hadzihasanovic2024combinatorics}, there is a functor 
\begin{equation*}
	\ordcpx{(-)} \colon \atom_\E \to \ssSet
\end{equation*}
such that, given an atom \( U \), the \( k \)\nbd simplices of \( \ordcpx{U} \) correspond to the injective chains of length \( k \) in the underlying poset of \( U \).
By left Kan extension along the Yoneda embedding, this produces an adjunction
\[\begin{tikzcd}
	\dCpx && \ssSet.
	\arrow[""{name=0, anchor=center, inner sep=0}, "{\ordcpx{-}}", curve={height=-12pt}, from=1-1, to=1-3]
	\arrow[""{name=1, anchor=center, inner sep=0}, "{\fun{N}}", curve={height=-12pt}, from=1-3, to=1-1]
	\arrow["\dashv"{anchor=center, rotate=-90}, draw=none, from=0, to=1]
\end{tikzcd}\]
We endow \( \dCpx \) with the model structure \( \Nw_0 \), and \( \ssSet \) with the Kan--Quillen weak model structure \( \MwKan 0 \) constructed in \cite[Theorem 5.2.1]{henry2020weak}, whose generating cofibrations and anodyne extensions can be identified with boundary inclusions and horns of oriented simplices, respectively.

\begin{lem} \label{lem:Fdelta_left_Quillen}
	The pair \( (\Fdelta, -_\Delta) \) determines a Quillen adjunction between $\MwKan 0$ and $\Nw_0$.
\end{lem}
\begin{proof}
	Same as \cite[Proposition 6.3.5]{henry2018regular}.
\end{proof}

\begin{lem} \label{lem:ordcpx_left_Quillen}
	The pair $(\ordcpx{-}, \fun{N})$ determines a Quillen adjunction between $\Nw_0$ and $\MwKan 0$.
\end{lem}
\begin{proof}
	Same as \cite[Proposition 6.3.3]{henry2018regular}, using \cite[Proposition 10.3.2]{hadzihasanovic2024combinatorics} to show that \( \ordcpx{U} \) is contractible for all atoms \( U \).
\end{proof}

\begin{lem} \label{lem:point_in_cone_is_equivalence}
	Let \( X \) be a directed complex.
	Then the canonical inclusion \( 1 \incl 1 \join X \) is an acyclic cofibration in $\Nw_0$.
\end{lem}
\begin{proof}
	Every directed complex is obtained from $\varnothing$ as a transfinite composition of pushouts along boundary inclusions of atoms, indexed by its cells in non-decreasing dimension.
	Because $1 \join -$ preserves connected colimits, and the canonical inclusion $1 \incl 1 \join X$ is precisely $1 \join -$ applied to $\varnothing \incl X$, by stability of acyclic cofibrations under pushouts and transfinite composition, it suffices to show that $1 \join \bd{}{}U \incl 1 \join U$ is an acyclic cofibration for all atoms $U$. 
	But, letting $0$ denote the only cell of $1$, this inclusion can be identified with the horn $\lambda^{\inl 0}_{1 \join U}$ of the atom $1 \join U$, which is an acyclic cofibration by definition.
\end{proof}

\noindent We denote by \( \term \) the terminal object of \( \dCpx \).

\begin{lem} \label{lem:morphims_atom_to_terminal_is_equivalence}
	Let \( U \) be an atom.
	The unique morphism \( U \to \term \) is an equivalence in \( \Nw_0 \).
\end{lem}
\begin{proof}
	By Lemma \ref{lem:submol_inclusions_are_acyclic} applied to the inclusion \( \bd{0}{-} U \incl U \) and 2-out-of-3 for equivalences, it is enough to show that the unique morphism \( 1 \to \term \) is an equivalence.
	Let \( i \colon 1 \incl 1 \join \term \) and \( j \colon \term \incl 1 \join \term \) be the canonical inclusions. 
	By Lemma \ref{lem:point_in_cone_is_equivalence}, \( i \) is an acyclic cofibration.  
	Since \( \term \) is terminal, \( j \) has a right inverse given by the unique morphism \( j' \colon 1 \join \term \to \term \).
	Up to a fibrant replacement of \( 1 \join \term \), we construct a homotopy between \( j j' i \) and \( i \), which is enough to conclude that \( j' \) is also a homotopy left inverse of \( j \), thus that \( j \) and \( j' \) are equivalences. 
	Letting \( u \colon 1 \to 1 \) be the cell representing the identity on \( 1 \), and \( v \colon 1 \to \term \) be the unique cell of shape \( 1 \) in \( \term \), the cell \( u \join v \colon 1 \join 1 \to 1 \join \term \) has shape \( \arr \) in \( 1 \join \term \).
	Since \( \arr \gray - \) is a functorial cylinder for \( \Nw_0 \), one checks that, up to a fibrant replacement, \( u \join v \) is a homotopy from \( i \) to \( j j'i \).  
	Finally, the unique morphism \( 1 \to \term \) is the composite of \( i \) and \( j' \).
	Since both have been shown to be equivalences, we conclude.
\end{proof}

\begin{lem} \label{lem:homotopy_hypothesis_dCpx}
	The functors \( \Fdelta(\ordcpx{-}) \) and \( \ordcpx{(\Fdelta-)} \) determine self-Quillen equivalences on $\Nw_0$ and on $\MwKan 0$, respectively.
\end{lem}
\begin{proof}
	By Lemma \ref{lem:ordcpx_left_Quillen} and Lemma \ref{lem:Fdelta_left_Quillen}, both functors are left Quillen.
	Thus, we may apply \cite[Theorem 6.3.7]{henry2018regular} first to \( \Fdelta(\ordcpx{-}) \), using \( 1 \) as distinguished object, and where the only non-trivial hypothesis is satisfied by Lemma \ref{lem:morphims_atom_to_terminal_is_equivalence}; and to \( \ordcpx{(\Fdelta-)} \), using \( \simplex{0} \) as distinguished object, where all the hypothesis are satisfied by standard arguments.
\end{proof}

\begin{prop} \label{prop:quillen_equivalences_kan_unmarked}
	The pairs $(\Fdelta, -_\Delta)$ and $(\ordcpx{-}, \fun{N})$ determine Quillen equivalences between $\Nw_0$ and $\MwKan 0$.
\end{prop}
\begin{proof}
	By Lemma \ref{lem:homotopy_hypothesis_dCpx}, $\Ho(\Fdelta)$ and $\Ho(\ordcpx{-})$ are a pair of equivalences between $\Ho(\MwKan 0)$ and $\Ho(\Nw_0)$.
	Since an adjunction in which the left adjoint is an equivalence is an adjoint equivalence, we conclude.
\end{proof}

\begin{thm} [Homotopy hypothesis] \label{thm:homotopy_hypothesis}
	The pair of functors
	\[
		\ordcpx{(\Um-)}\colon \cofs{\Mw_0} \to \cofs{\MwKan 0}, \quad \quad
		\sharpm{(\fun{N}-)}\colon \fibs{\MwKan 0} \to \fibs{\Mw_0}
	\]
	exhibits a Quillen equivalence between $\Mw_0$ and $\MwKan 0$.
\end{thm}
\begin{proof}
	Follows from Proposition \ref{prop:0_model_structure_quillen_equivalence} and Proposition \ref{prop:quillen_equivalences_kan_unmarked}.
\end{proof}

\begin{comm}
	By the results of \cite[Section 5.5]{henry2020weak}, we conclude that $\Mw_0$ presents the homotopy theory of the classical homotopy types, hence our model of $(\infty, 0)$\nbd categories satisfies the homotopy hypothesis.
\end{comm}

\noindent
We briefly discuss the possible extension of Proposition \ref{prop:quillen_equivalences_kan_unmarked} from $(\infty, 0)$\nbd categories to \inftyn\nbd categories.
Recall that a \emph{marked}, or \emph{stratified semi-simplicial set} is a semi-simplicial set \( K \) together with a subset \( A \) of marked simplices of dimension \( > 0 \). 
Marked semi-simplicial sets and morphisms that respect the marked structure form a category \( \mssSet \), with an evident forgetful functor \( \Um \colon \mssSet \to \ssSet \).
Moreover, we can lift the extension-restriction adjunction between $\ssSet$ and $\dCpx$ to an adjunction
\[\begin{tikzcd}
	\mssSet && \mdCpx,
	\arrow[""{name=0, anchor=center, inner sep=0}, "{\Fdelta}", curve={height=-12pt}, from=1-1, to=1-3]
	\arrow[""{name=1, anchor=center, inner sep=0}, "{-_\Delta}", curve={height=-12pt}, from=1-3, to=1-1]
	\arrow["\dashv"{anchor=center, rotate=-90}, draw=none, from=0, to=1]
\end{tikzcd}\]
letting $\Fdelta(K, A) \eqdef (\Fdelta K, \Fdelta A)$, in such a way that the square
\[
	\begin{tikzcd}
		\mssSet & \mdCpx \\
		\ssSet & {\dCpx.}
		\arrow["\Fdelta", hook, from=1-1, to=1-2]
		\arrow["\Um", from=1-1, to=2-1]
		\arrow["\Um", from=1-2, to=2-2]
		\arrow["\Fdelta", hook, from=2-1, to=2-2]
	\end{tikzcd}
\]
commutes.
By results of \cite[Section 5.5]{henry2020weak}, for each \( n \geq 0 \), there is a weak model structure \( \MwKan{n} \eqdef (\mssSet, \ICof{I^\Delta}, \JFib{\Jn^\Delta}) \) on semi-simplicial sets which is Quillen equivalent to the \( n \)\nbd complicial model structure on simplicial sets \cite{verity2008weak, ozornova2020model}.
The set \( I^\Delta \) can be identified, via \( \Fdelta \), with the subset of generating cofibrations of $\Mwn$ whose codomain is a marked oriented simplex; while the set \( \Jn^\Delta \) can be identified, via \( \Fdelta \), with the union of
\begin{itemize}
	\item a subset \( \Jhorn^\Delta \) of \( \Jhorn \), comprising some marked horns of oriented simplices, called the \emph{oriented complicial horn extension};
	\item the set of markings \( \set{\m_{\simplex{k}} \colon \flatm{(\simplex{k})} \to \mrk{(\simplex{k})} \mid k > n} \), called the \emph{oriented \( n \)\nbd triviality extension};
	\item markings of the form \( (\simplex{k}, A_{k, \ell}) \incl (\simplex{k}, A'_{k, \ell}) \) for all \( k \geq 2 \) and \( 0 \le \ell \le k \) and appropriate values of \( A_{k, \ell} \) and \( A'_{k, \ell} \), called the \emph{oriented complicial thinness extension};
	\item markings of the form \( (\simplex{k}, B_k) \incl (\simplex{k}, B'_k) \) for all \( k \geq 2 \) and appropriate values of \( B_k \) and \( B'_k \), called the \emph{oriented complicial saturation extensions}.
\end{itemize}  

\noindent 
The oriented complicial horn extensions and the oriented \( n \)\nbd triviality extensions are, by definitions, acyclic cofibrations in \( \Mwn \), and we are confident that the same is true of oriented complicial thinness extensions and saturation extensions, the first from the right lifting property of fibrations against marked horns where one boundary is a marked round diagram, the second from the right lifting property against saturations.
We then make the following conjecture, which, in the light of \cite{loubaton2023theory}, would imply that our model is equivalent to the main cluster of geometric models of \inftyn\nbd categories.

\begin{conj}
	The pair $(\Fdelta, -_\Delta)$ is a Quillen equivalence between \( \MwKan{n} \) and \( \Mwn \). 
\end{conj}

\section{A semi-strict model} \label{part:semistrict}

\subsection{Merge-complexes} \label{sec:merge}

\noindent
As Ronnie Brown used to say, a notion of composition is an ``algebraic inverse to subdivision'' \cite{brown2003intuitions}, and indeed, our notion of composition on directed complexes will arise by duality from the action of subdivisions.

\begin{dfn}[Merge-complex]
	A \emph{merge-complex} is a $\Gamma_\S$\nbd continuous presheaf on the category $\frdCpx_{\S\L}$ of finite regular directed complexes and local subdivisions.
\end{dfn}

\noindent
We let $\MCpx$ denote the category $\PSh_{\Gamma_\S}(\frdCpx_{\S\L})$ of merge-complexes.
We have the usual embedding
\[
	\frdCpx_{\S\L} \incl \MCpx
\]
as well as an adjunction
\[\begin{tikzcd}
	\dCpx && \MCpx
	\arrow[""{name=0, anchor=center, inner sep=0}, "{\FMerg}", curve={height=-12pt}, from=1-1, to=1-3]
	\arrow[""{name=1, anchor=center, inner sep=0}, "{\UMerg}", curve={height=-12pt}, from=1-3, to=1-1]
	\arrow["\dashv"{anchor=center, rotate=-90}, draw=none, from=0, to=1]
\end{tikzcd}\]
where $\UMerg$ is presheaf restriction, while $\FMerg$ is the composite of the restricted presheaf extension $\dCpx \to \PSh(\frdCpx_{\S\L})$ with the reflector onto $\MCpx$.
In fact, the following result shows that the reflector acts trivially on the image of $\dCpx$.
Let $i$ denote the inclusion $\frdCpx_\L \incl \frdCpx_{\S\L}$, so that $\Lan i$ is the induced presheaf extension functor.

\begin{lem} \label{lem:lan_takes_gamma_continuous_to_gammas_continuous}
	Let $X$ be a directed complex.
	Then the extension $\Lan i X$ is a merge-complex.
\end{lem}
\begin{proof}
	We need to show that $\Lan i X$ is $\Gamma_\S$\nbd continuous, under the assumption that $X$ is $\Gamma$\nbd continuous.
	Since $(\S, \L)$ is an orthogonal factorisation system, we are in the conditions of Lemma \ref{lem:presheaf_restriction_to_right_class}, so we can explicitly describe elements $P \to \Lan i X$ as equivalence classes of pairs $[u, s]$ of a subdivision $s\colon P \sd P'$ and a diagram $u\colon P' \to X$, modulo the action of isomorphisms in the middle.
	We will portray such a pair as a ``formal diagram''
\[\begin{tikzcd}
	P & {P'} & X.
	\arrow[arloop->, "s", from=1-1, to=1-2]
	\arrow["u", "\shortmid"{marking}, from=1-2, to=1-3]
\end{tikzcd}\]
	First of all, since the only subdivision with domain $\varnothing$ is the identity, and $X$ is $\Gamma$\nbd continuous so there is a unique diagram $\varnothing \to X$, there is a unique formal diagram
	\[\begin{tikzcd}
	\varnothing & \varnothing & X
	\arrow[arloop->, "{\idd{}}", from=1-1, to=1-2]
	\arrow["{!}", "\shortmid"{marking}, from=1-2, to=1-3]
\end{tikzcd}\]
	which proves that $\Lan i X$ is continuous with respect to the initial object.
	Next, let $(\iota, \iota')$ be a span of embeddings in $\frdCpx_{\S\L}$.
	A cone under this span with tip $\Lan i X$ is represented by a formal commutative diagram
	\[\begin{tikzcd}
	& P & {s(P)} \\
	U & {t(U)} && X \\
	& {P'} & {s'(P')}
	\arrow[arloop->, "s", from=1-2, to=1-3]
	\arrow["u", "\shortmid"{marking}, from=1-3, to=2-4]
	\arrow["\iota", hook, from=2-1, to=1-2]
	\arrow[arloop->, "t", from=2-1, to=2-2]
	\arrow["{\iota'}", hook, from=2-1, to=3-2]
	\arrow["j", hook, from=2-2, to=1-3]
	\arrow["{j'}", hook, from=2-2, to=3-3]
	\arrow[arloop->, "{s'}", from=3-2, to=3-3]
	\arrow["{v}","\shortmid"{marking}, from=3-3, to=2-4]
\end{tikzcd}\]
	using the $(\S, \E)$ factorisation.
	Since $X$ is $\Gamma$\nbd continuous, the cone given by $(u, v)$ under the span $(j, j')$ of embeddings induces universally a diagram $u \cup v\colon s(P) \pout{t(U)} s'(P') \to X$.
	Precomposing this with the subdivision 
	\[
		s \cup s'\colon P \pout{U} P' \sd s(P) \pout{t(U)} s'(P')
	\]
	obtained by universality of the pushouts in $\frdCpx_{\S\L}$, we obtain a formal commutative diagram
	\[\begin{tikzcd}
	& P & {s(P)} \\
	U && {P \pout{U} P'} & {s(P) \pout{t(U)} s'(P')} & X \\
	& {P'} & {s'(P')}
	\arrow[arloop->, "s", from=1-2, to=1-3]
	\arrow[hook, from=1-2, to=2-3]
	\arrow[hook, from=1-3, to=2-4]
	\arrow["u", "\shortmid"{marking}, curve={height=-18pt}, from=1-3, to=2-5]
	\arrow["\iota", hook, from=2-1, to=1-2]
	\arrow["{\iota'}", hook, from=2-1, to=3-2]
	\arrow["\lrcorner"{anchor=center, pos=0.125, rotate=-135}, draw=none, from=2-3, to=2-1]
	\arrow[arloop->, dashed, "{s \cup s'}", from=2-3, to=2-4]
	\arrow[dashed, "{u \cup v}", "\shortmid"{marking}, from=2-4, to=2-5]
	\arrow[hook, from=3-2, to=2-3]
	\arrow[arloop->, "{s'}", from=3-2, to=3-3]
	\arrow[hook, from=3-3, to=2-4]
	\arrow["v", "\shortmid"{marking}, curve={height=18pt}, from=3-3, to=2-5]
\end{tikzcd}\]
	which proves that $\Lan i X$ is continuous with respect to the pushout of $(\iota, \iota')$.
	Finally, let $(\iota, s)$ be span of an embedding and a subdivision.
	Then, using only $(\S, \E)$ factorisations, a cone under this span with tip $\Lan i X$ is represented by a formal diagram
	\[\begin{tikzcd}[row sep=scriptsize]
	& V & {t(U)} \\
	U &&&& X. \\
	& P && {s'(P)}
	\arrow[arloop->, "{t'}", from=1-2, to=1-3]
	\arrow["u", "\shortmid"{marking}, curve={height=-12pt}, from=1-3, to=2-5]
	\arrow["{\iota'}", hook, from=1-3, to=3-4]
	\arrow[arloop->, "s", from=2-1, to=1-2]
	\arrow[arloop->, "t", curve={height=12pt}, from=2-1, to=1-3]
	\arrow["\iota", hook, from=2-1, to=3-2]
	\arrow[arloop->, "{s'}", from=3-2, to=3-4]
	\arrow["v", "\shortmid"{marking}, from=3-4, to=2-5]
\end{tikzcd}\]
	By universality of the pushout $\subs{P}{V}{\iota(U)}_s$ in $\frdCpx_{\S\L}$, there is a unique subdivision making the formal diagram
\[\begin{tikzcd}
	& V & {t(U)} \\
	U && {\subs{P}{V}{\iota(U)}_s} && X \\
	& P && {s'(P)}
	\arrow[arloop->, "{t'}", from=1-2, to=1-3]
	\arrow[hook, from=1-2, to=2-3]
	\arrow["u", "\shortmid"{marking}, curve={height=-12pt}, from=1-3, to=2-5]
	\arrow["{\iota'}", curve={height=-6pt}, hook, from=1-3, to=3-4]
	\arrow[arloop->, "s", from=2-1, to=1-2]
	\arrow["\iota", hook, from=2-1, to=3-2]
	\arrow["\lrcorner"{anchor=center, pos=0.125, rotate=-135}, draw=none, from=2-3, to=2-1]
	\arrow[arloop->, dashed, from=2-3, to=3-4]
	\arrow[arloop->, from=3-2, to=2-3]
	\arrow[arloop->, "{s'}", from=3-2, to=3-4]
	\arrow["v", "\shortmid"{marking}, from=3-4, to=2-5]
\end{tikzcd}\]
	commute.
	This proves that $\Lan i X$ is continuous with respect to the pushout of $(\iota, s)$, and we conclude that $\Lan i X$ is $\Gamma_\S$\nbd continuous, that is, it is a merge-complex.
\end{proof}

\begin{comm}
	It follows from Lemma \ref{lem:lan_takes_gamma_continuous_to_gammas_continuous} that the left adjoint $\FMerg$ is simply a restriction and corestriction of $\Lan i$, and we can explicitly describe diagrams $P \to \FMerg X$ as pairs $[u\colon P' \to X, s\colon P \sd P']$.
\end{comm}

\begin{rmk}
	Unlike the adjunction $\fun{F}_\Infl \dashv \fun{U}_\Infl$, the adjunction $\FMerg \dashv \UMerg$ is \emph{not} monadic.
	Indeed, one can observe that every $\Gamma$\nbd continuous presheaf on $\frdCpx_{\S\L}$ is also an algebra for the monad induced by this adjunction, and it is easy to come up with examples of $\Gamma$\nbd continuous presheaves which are not $\Gamma_\S$\nbd continuous.
	In fact, the category of algebras for $\UMerg\FMerg$ is equivalent to the category of $\Gamma$\nbd continuous, not necessarily $\Gamma_\S$\nbd continuous presheaves.
\end{rmk}

\begin{rmk}
	Note that the analogue of Lemma \ref{lem:lan_takes_gamma_continuous_to_gammas_continuous} does not hold for the inclusion $\frdCpx_\L \incl \frdCpx_{\C\L}$; the point of failure is the absence of a $(\C, \E)$ factorisation system.
\end{rmk}

\noindent
With regards to \emph{diagrams, pasting diagrams}, and all related notions, since we have the sequence of inclusions $\frdCpx_{\L} \incl \frdCpx_{\S\L} \incl \MCpx$, we can import all the terminology from directed complexes to merge-complexes.
The ability to act on diagrams with subdivisions, however, introduces new possibilities.

\begin{dfn}[Local decomposition of a diagram]
	Let $u\colon U \to X$ and $w\colon W \to X$ be diagrams in a merge-complex, let $\iota\colon V \incl U$ be an embedding, $s\colon V \sd W$ be a subdivision, $v \eqdef u\iota$, and suppose that $v = ws$.
	We let $\subs{u}{w}{\iota(v)}_s$ denote the diagram universally obtained by $\Gamma_\S$\nbd continuity of $X$ in
	\[\begin{tikzcd}
	V & W \\
	U & {\subs{U}{W}{\iota(V)}_s} && X
	\arrow[arloop->, "s", from=1-1, to=1-2]
	\arrow["\iota", hook, from=1-1, to=2-1]
	\arrow[hook, from=1-2, to=2-2]
	\arrow["w", curve={height=-12pt}, from=1-2, to=2-4]
	\arrow[arloop->, from=2-1, to=2-2]
	\arrow["u", curve={height=30pt}, from=2-1, to=2-4]
	\arrow["\lrcorner"{anchor=center, pos=0.125, rotate=180}, draw=none, from=2-2, to=1-1]
	\arrow["{\subs{u}{w}{\iota(v)}_s}", dashed, from=2-2, to=2-4]
\end{tikzcd}\]
	and call it a \emph{local decomposition of $u$}.
\end{dfn}

\begin{dfn}[Merger of a round diagram]
	Let $u\colon U \to X$ be a round diagram in a merge-complex.
	The \emph{merger of $u$} is the cell $\mrgcom{u} \eqdef u\mrg{}_U\colon \mrg{U} \to X$.
\end{dfn}

\noindent
Let $U$ be a round molecule, $n \eqdef \dim U$.
By \cite[Proposition 9.1.11]{hadzihasanovic2024combinatorics}, there is a unique subdivision $o\colon \globe{n} \sd U$.

\begin{dfn}[Globular composite of a round diagram]
	Let $u\colon U \to X$ be a round diagram in a merge-complex, $n \eqdef \dim U$.
	The \emph{globular composite of $u$} is the cell $\glcom{u} \eqdef uo\colon \globe{n} \to X$.
\end{dfn}

\begin{prop} \label{prop:merge_complexes_are_lfp}
	The category $\MCpx$ is locally finitely presentable.
\end{prop}
\begin{proof}
	It is a category of continuous presheaves with respect to a set of finite colimit cones, which is directly a category of models of a finite limit sketch.
\end{proof}

\begin{thm} \label{thm:gray_product_of_merge-complexes}
	There is an essentially unique biclosed monoidal structure $(\MCpx, \gray, 1)$ such that the embedding of $(\frdCpx_{\S\L}, \gray, 1)$ is strong monoidal.
	Moreover, the functor $\FMerg\colon (\dCpx, \gray, 1) \to (\MCpx, \gray, 1)$ is strong monoidal.
\end{thm}
\begin{proof}
	By Proposition \ref{prop:merge_complexes_are_lfp} and Proposition \ref{prop:gray_preservation_of_colimits}, $(\frdCpx_{\S\L}, \gray, 1)$ with the class of colimit cones $\Gamma_\S$ meets the conditions of Lemma \ref{lem:day_for_gamma_continuous_presheaves}.
	Moreover, by Lemma \ref{lem:lan_takes_gamma_continuous_to_gammas_continuous}, we have
	\[
		\fun{r}_{\Gamma_\S} (\Lan i) \fun{r}_\Gamma \simeq (\Lan i) \fun{r}_\Gamma,
	\]
	and because $\Gamma_\S$ is a superset of $i(\Gamma)$,
	\[
		\fun{r}_{\Gamma_\S} (\Lan i) \fun{r}_\Gamma \simeq \fun{r}_{\Gamma_\S} (\Lan i).
	\]
	Since $i\colon (\frdCpx_\L, \gray, 1) \incl (\frdCpx_{\S\L}, \gray, 1)$ is strong monoidal, we conclude again using Lemma \ref{lem:day_for_gamma_continuous_presheaves}.
\end{proof}

\begin{comm}
	We can give an alternative description of the Gray product of two merge-complexes as follows.
	Let $\Merg \eqdef \UMerg\FMerg$ be the monad determined by the adjunction $\FMerg \dashv \UMerg$ and let $X$, $Y$ be directed complexes.
	Using our explicit description of $\FMerg$, we see that $\Merg$ is commutative, with the left strength
	\begin{align*}
		\tau_{X, Y}\colon X \gray \Merg Y & \to \Merg(X \gray Y), \\
			u \gray [v, s] & \mapsto [u \gray v, \idd{} \gray s]
	\end{align*}
	and the dual right strength
	\begin{align*}
		\sigma_{X, Y}\colon \Merg X \gray Y & \to \Merg(X \gray Y), \\
			[u, s] \gray v & \mapsto [u \gray v, s \gray \idd{}]
	\end{align*}
	giving rise to the natural transformation
	\begin{align*}
		\alpha_{X, Y}\colon \Merg X \gray \Merg Y & \to \Merg(X \gray Y), \\
			[u, s] \gray [v, t] & \mapsto [u \gray v, s \gray t].
	\end{align*}
	For each merge-complex $X$, let $\varepsilon_X\colon \FMerg \UMerg X \to X$ be the counit of the adjunction.
	Then, given merge-complexes $X$, $Y$, the Gray product $X \gray Y$ can be presented as the coequaliser of the diagram
\[\begin{tikzcd}[column sep=tiny]
	& {\FMerg\Merg(\UMerg X \gray \UMerg Y)} \\
	{\FMerg(\Merg\UMerg X \gray \Merg\UMerg Y)} && {\FMerg(\UMerg X \gray \UMerg Y)}
	\arrow["{\varepsilon_{\FMerg(\UMerg X \gray \UMerg Y)}}", from=1-2, to=2-3]
	\arrow["{\FMerg(\alpha_{\UMerg X, \UMerg Y})}", from=2-1, to=1-2]
	\arrow["{\FMerg(\UMerg\varepsilon_X \gray \UMerg\varepsilon_Y)}"', from=2-1, to=2-3]
\end{tikzcd}\]
	in $\MCpx$.
	Less formally and more concretely, $X \gray Y$ is the quotient of $\FMerg(\UMerg X \gray \UMerg Y)$, whose cells of shape $U$ are of the form $[v \gray w, s]$ for a pair of diagrams $u\colon V \to X$ and $w\colon W \to Y$ and a subdivision $s\colon U \sd V \gray W$, by the equivalence relation generated by
	\[
		[v \gray w, s \gray t] = [vs \gray wt, \idd{}]
	\]
	whenever the equation is well-formed.
\end{comm}

\noindent
Next, we define marked versions of merge-complexes.
Let $X$ be a merge-complex and $A \subseteq \gr{>0}{\cell X}$.
Then, we let $\clcom{A}$ be the closure of $A$ under the following conditions: 
\begin{enumerate}
	\item if $U$ is an atom, $s\colon U \sd V$ a subdivision, and $v\colon V \to X$ an $\clcom{A}$\nbd round diagram, then $vs\colon U \to X$ is in $\clcom{A}$,
	\item if $U$, $V$ are atoms, $s\colon U \sd V$ a subdivision, $v\colon V \to X$ a cell, and $vs$ is in $\clcom{A}$, then $v$ is in $\clcom{A}$.
\end{enumerate}

\begin{rmk}
	The second condition is equivalent to the following: given a cell $u\colon U \to X$ in $\clcom{A}$, a subdivision $s\colon \bd{}{}U \sd \bd{}{}V$ of the boundary of $U$, and a diagram $\bd{}{}v\colon \bd{}{}V \to X$ such that $\bd{}{}u = (\bd{}{}v)s$, the decomposition $\subs{u}{\bd{}{}v}{\bd{}{}u}_s$ is in $\clcom{A}$.
	Note that the converse implication is a special case of the first condition.
\end{rmk}

\begin{dfn}[Marked merge-complex]
	A \emph{marked merge-complex} is a pair $(X, A)$ of a merge-complex $X$ and a set $A \subseteq \gr{>0}{\cell X}$ such that $A = \clcom{A}$.
	Given marked merge-complexes $(X, A)$, $(Y, B)$, a \emph{morphism} $f\colon (X, A) \to (Y, B)$ is a morphism $f\colon X \to Y$ in $\MCpx$ such that $A \subseteq \invrs{f}B$.
\end{dfn}

\begin{comm}
	The intuition between the two closure conditions on $A$ is the following.
	The first condition asks that composites of marked round diagrams be marked; this is an expected property of the composition of equivalences in an \inftyn\nbd category.
	The second condition asks that the property of being marked remain stable under decompositions in the boundary of a cell.
	The idea is that different boundary decompositions of a cell $u$ represent the ``same'' cell---which, ultimately, has the globular composite $\glcom{u}$ as representative---so the marking should be a property of the entire equivalence class.
\end{comm}

\noindent
We let $\mMCpx$ denote the category of marked merge-complexes.
As in the case of marked directed complexes, we have a forgetful functor
\[
	\Um\colon \mMCpx \to \MCpx
\]
as well as the ``same'' left and right adjoints $\flatm{-}$ and $\sharpm{-}$ sending $X$ to $(X, \varnothing)$ and to $(X, \gr{>0}{\cell X})$, respectively.
Moreover, the adjunction $\FMerg \dashv \UMerg$ lifts to an adjunction
\[\begin{tikzcd}
	\mdCpx && \mMCpx,
	\arrow[""{name=0, anchor=center, inner sep=0}, "{\FmMerg}", curve={height=-12pt}, from=1-1, to=1-3]
	\arrow[""{name=1, anchor=center, inner sep=0}, "{\UmMerg}", curve={height=-12pt}, from=1-3, to=1-1]
	\arrow["\dashv"{anchor=center, rotate=-90}, draw=none, from=0, to=1]
\end{tikzcd}\]
such that $\FmMerg(X, A) \eqdef (\FMerg X, \clcom{A})$, where the set $A \subseteq \gr{>0}{\cell X}$ is identified with the set $\set{[u, \idd{}] \mid u \in A} \subseteq \gr{>0}{\cell (\FMerg X)}$, and $\UmMerg(X, A) = (\UMerg X, A)$.
The following is proved by a routine argument.

\begin{prop} \label{prop:colimits_of_marked_merge_complexes}
	Let $\fun{F}\colon \cat{J} \to \mMCpx$ be a small diagram.
	Then
	\begin{enumerate}
		\item every colimit cone $\gamma$ under $\Um\fun{F}$ in $\MCpx$ with tip $X$ lifts to a colimit cone under $\fun{F}$ in $\mMCpx$ with tip
			\[
				\left( X, \clcom{\bigcup_{j \in \Ob \cat{J}} \set{\gamma_j(u) \mid \text{$u$ is marked in $\fun{F}j$}}} \right),
			\]
		\item every limit cone $\delta$ over $\Um\fun{F}$ in $\MCpx$ with tip $X$ lifts to a limit cone over $\fun{F}$ in $\mMCpx$ with tip
			\[
				\left( X, \bigcap_{j \in \Ob \cat{J}} \set{u \in \gr{>0}{\cell X} \mid \text{$\delta_j(u)$ is marked in $\fun{F}j$}} \right).
			\]
	\end{enumerate}
\end{prop}

\begin{prop} \label{prop:local_presentability_of_marked_mcpx}
	The category $\mMCpx$ is locally finitely presentable.
\end{prop}
\begin{proof}
	Since subdivisions are determined by their restrictions on atoms, whose images are round molecules, and all round molecules can be constructed as pastings of molecules, we can construct a finite limit sketch for $\mMCpx$ as follows.
	We take (the opposite of) the full subcategory on objects $\flatm{U}$ and $\mrk{U}$ with $U$ ranging over molecules.
	As set of cones, we take 
	\begin{enumerate}
		\item pastings of unmarked molecules
\[\begin{tikzcd}
	{\flatm{\bd{k}{+}U} = \flatm{\bd{k}{-}V}} & {\flatm{V}} \\
	{\flatm{U}} & {\flatm{(U \cp{k} V)}}
	\arrow[hook, from=1-1, to=1-2]
	\arrow[hook, from=1-1, to=2-1]
	\arrow[hook, from=1-2, to=2-2]
	\arrow[hook, from=2-1, to=2-2]
\end{tikzcd}\]
		as well as pastings of marked molecules of the same dimension
\[\begin{tikzcd}
	{\flatm{\bd{k}{+}U} = \flatm{\bd{k}{-}V}} & {\mrk{V}} \\
	{\mrk{U}} & {\mrk{(U \cp{k} V)},}
	\arrow[hook, from=1-1, to=1-2]
	\arrow[hook, from=1-1, to=2-1]
	\arrow[hook, from=1-2, to=2-2]
	\arrow[hook, from=2-1, to=2-2]
\end{tikzcd} \quad \quad \dim U = \dim V \]
		and ``mixed'' pastings when one molecules has strictly lower dimension
\[\begin{tikzcd}
	{\flatm{\bd{k}{+}U} = \flatm{\bd{k}{-}V}} & {\mrk{V}} \\
	{\flatm{U}} & {\mrk{(U \cp{k} V)},}
	\arrow[hook, from=1-1, to=1-2]
	\arrow[hook, from=1-1, to=2-1]
	\arrow[hook, from=1-2, to=2-2]
	\arrow[hook, from=2-1, to=2-2]
\end{tikzcd}	\quad \dim U < \dim V, \]
\[\begin{tikzcd}
	{\flatm{\bd{k}{+}U} = \flatm{\bd{k}{-}V}} & {\flatm{V}} \\
	{\mrk{U}} & {\mrk{(U \cp{k} V)},}
	\arrow[hook, from=1-1, to=1-2]
	\arrow[hook, from=1-1, to=2-1]
	\arrow[hook, from=1-2, to=2-2]
	\arrow[hook, from=2-1, to=2-2]
\end{tikzcd}	\quad \dim U > \dim V; \]
		these are all the pasting squares whose morphisms are all conservative, and result in a molecule of the form $\flatm{U}$ and $\mrk{U}$;
		\item substitutions along co-mergers of atoms embedded in the boundary of other atoms, either marked or unmarked:
		\[\begin{tikzcd}
	{\flatm{\mrg{V}}} & {\flatm{V}} \\
	{\flatm{U}} & {\flatm{\subs{U}{V}{\iota\mrg{V}}_s}},
	\arrow[arloop->, "s", from=1-1, to=1-2]
	\arrow["\iota", hook, from=1-1, to=2-1]
	\arrow[hook, from=1-2, to=2-2]
	\arrow[arloop->, from=2-1, to=2-2]
\end{tikzcd}	\quad \quad \dim V < \dim U,
\]
\[\begin{tikzcd}
	{\flatm{\mrg{V}}} & {\flatm{V}} \\
	{\mrk{U}} & {\mrk{\subs{U}{V}{\iota\mrg{V}}_s}},
	\arrow[arloop->, "s", from=1-1, to=1-2]
	\arrow["\iota", hook, from=1-1, to=2-1]
	\arrow[hook, from=1-2, to=2-2]
	\arrow[arloop->, from=2-1, to=2-2]
\end{tikzcd}	\quad \quad \dim V < \dim U;
\]
		\item the squares (\ref{eq:marking_marking_pushout}) of top-markings of atoms.
	\end{enumerate}
	The first set encodes $\Gamma$\nbd continuity, the second set encodes $\Gamma_\S$\nbd continuity as well as stability of marking under boundary decompositions, and the third set encodes the fact that marking is a property and not structure on a cell.
\end{proof}

\begin{dfn}[Gray product of marked merge-complexes]
	Let $(X, A)$ and $(Y, B)$ be marked merge-complexes.
	The \emph{Gray product of $(X, A)$ and $(Y, B)$} is the marked merge-complex
	\[
		(X, A) \gray (Y, B) \eqdef (X \gray Y, \clcom{(\cell X \gray B) \cup (A \gray \cell Y)}).
	\]
\end{dfn}

\noindent
To prove that the Gray product defines a monoidal structure on $\mMCpx$, we can follow closely the proof of \cite[Proposition 2.29]{loubaton2024inductive}, which proves the analogous fact for the Gray product of marked strict $\omega$\nbd categories.

\begin{lem} \label{lem:closures_of_marked_cells_gray_product}
	Let $X$, $Y$ be merge-complexes and let $A \subseteq \cell X$, $B \subseteq \cell Y$.
	Then $\clcom{A \gray B} = \clcom{\clcom{A} \gray B} = \clcom{A \gray \clcom{B}} = \clcom{\clcom{A} \gray \clcom{B}}$.
\end{lem}
\begin{proof}
	It suffices to prove that $\clcom{A \gray B} = \clcom{\clcom{A} \gray B}$; the second identity is dual and the third follows from the others.
	From $A \subseteq \clcom{A}$ we have immediately $\clcom{A \gray B} \subseteq \clcom{\clcom{A} \gray B}$.
	Conversely, it suffices to show $\clcom{A} \gray B \subseteq \clcom{A \gray B}$.
	Fix $v\colon V \to Y$ in $B$; we show inductively that, for all $u\colon U \to X$ in $\clcom{A}$, $u \gray v \in \clcom{A \gray B}$.
	The case $u \in A$ is immediate.
	Suppose that $u = u's$ for some $\clcom{A}$-round diagram $u'\colon U' \to X$ and subdivision $s\colon U \sd U'$ such that, for all top-dimensional $w \submol u'$, we have $w \gray v \in \clcom{A \gray B}$.
	Then top-dimensional cells in $u' \gray v$ are precisely those of the form $w \gray v$ for $w$ top-dimensional in $u'$, so $u' \gray v$ is an $\clcom{A \gray B}$\nbd round diagram.
	We conclude that $u \gray v = (u' \gray v)(s \gray \idd{V})$ is in $\clcom{A \gray B}$.
	Next, suppose that there exist a cell $u'\colon U' \to X$ in $\clcom{A}$ as well as a subdivision $t\colon U' \sd U$ such that $u' = ut$.
	Then, by the inductive hypothesis, $u' \gray v$ is in $\clcom{A \gray B}$, and $u' \gray v = (u \gray v)(t \gray \idd{V})$, which implies that $u \gray v$ is in $\clcom{A \gray B}$.
	This completes the induction and the proof.
\end{proof}

\begin{prop} \label{prop:gray_product_of_marked_merge-complexes}
	Let $f\colon (X, A) \to (Y, B)$ and $g\colon (X', A') \to (Y', B')$ be morphisms of marked merge-complexes.
	Then the Gray product $f \gray g$ of their underlying morphisms of merge-complexes determines a morphism
	\[
		f \gray g\colon (X, A) \gray (X', A') \to (Y, B) \gray (Y', B').
	\]
	This determines a biclosed monoidal structure $(\mMCpx, \gray, \flatm{1})$, such that the functor $\FmMerg\colon (\mdCpx, \gray, \flatm{1}) \to (\mMCpx, \gray, \flatm{1})$ is strong monoidal.
\end{prop}
\begin{proof}
	The fact that $f \gray g$ is compatible with the marked structure on the Gray product is straightforward, so it suffices to show that associators and unitors are compatible.
	First of all, we observe that $\clcom{\cell X \gray \cell Y} = \cell{(X \gray Y)}$ follows from the description of $X \gray Y$ as a quotient of $\FMerg(\UMerg X \gray \UMerg Y)$, and that for all pairs $C, D \subseteq \cell X$ we have $\clcom{C \cup D} = \clcom{\clcom{C} \cup D} = \clcom{C \cup \clcom{D}} = \clcom{\clcom{C} \cup \clcom{D}}$.
	Compatibility with associators then follows from Lemma \ref{lem:closures_of_marked_cells_gray_product} by the same proof as \cite[Lemma 2.27]{loubaton2024inductive}, while compatibility with unitors is straightforward from the relation $(\cell X \gray \varnothing) \cup (A \gray \cell 1) = A \gray \cell 1 \simeq A$.
	Finally, the fact that the monoidal structure is biclosed follows from the fact that the underlying monoidal structure on $\MCpx$ is biclosed, coupled with direct inspection of colimit preservation using Proposition \ref{prop:colimits_of_marked_merge_complexes}, and local presentability.
	Strong monoidality of $\FmMerg$ is straightforward after Theorem 
	\ref{thm:gray_product_of_merge-complexes}.
\end{proof}

\begin{rmk}
	Note that the right adjoints $\UMerg$ and $\UmMerg$ are only lax monoidal: given merge-complexes $X$, $Y$, the Gray product $\UMerg X \gray \UMerg Y$ usually has strictly fewer cells than $\UMerg (X \gray Y)$.
\end{rmk}

\subsection{Weak model structures on marked merge-complexes} \label{sec:weakonmerge}

\noindent
With the results of Sections \ref{sec:weakonmarked} and \ref{sec:merge}, it is straightforward to construct weak model structures on $\mMCpx$ mirroring those on $\mdCpx$: we have a commutative diagram of strong monoidal functors
\begin{equation} \label{eq:inclusions_directed_to_merge}
\begin{tikzcd}
	{(\frdCpx_\L, \gray, 1)} & {(\dCpx, \gray, 1)} & {(\mdCpx, \gray, \flatm{1})} \\
	{(\frdCpx_{\S\L}, \gray, 1)} & {(\MCpx, \gray, 1)} & {(\mMCpx, \gray, \flatm{1})}
	\arrow[hook, from=1-1, to=1-2]
	\arrow[hook, from=1-1, to=2-1]
	\arrow["{\flatm{-}}", from=1-2, to=1-3]
	\arrow["\FMerg", from=1-2, to=2-2]
	\arrow["\FmMerg", from=1-3, to=2-3]
	\arrow[hook, from=2-1, to=2-2]
	\arrow["{\flatm{-}}", from=2-2, to=2-3]
\end{tikzcd}
\end{equation}
which allows us to transfer both the functorial cylinder and the generating sets of cofibrations and anodyne extensions of $\Mwn$ from $\mdCpx$ to $\mMCpx$. 
In fact, with our notational conventions, all of these objects admit literally the same definition: for each $n \in \Ninfty$, we take
\begin{enumerate}
	\item the functorial cylinder $\fun{I} \eqdef \marr \gray -\colon \mMCpx \to \mMCpx$ with the natural transformations $\iota^\a \eqdef (0^\a \gray \idd{})\lambda$ for each $\a \in \set{-, +}$, where now everything is relative to the monoidal structure $(\mMCpx, \gray, \flatm{1})$,
	\item the set of generating cofibrations $I$ and the set of generating anodyne extensions $\Jn$, now interpreted in $\mMCpx$ through any of the admissible paths in the diagram (\ref{eq:inclusions_directed_to_merge}).
\end{enumerate}

\begin{thm} \label{thm:model_structures_on_marked_merge}
	Let $n \in \Ninfty$.
	Then $\MMwn \eqdef (\mMCpx, \ICof{I}, \JFib{\Jn})$ is a weak model structure on $\mMCpx$ such that
	\begin{enumerate}
		\item the cofibrations are the transfinite compositions of (coproducts of) pushouts along $\flatm{\bdmap_U}$ and $\m_U$, with $U$ ranging over atoms,
		\item the fibrations are the morphisms $p$ such that $\UmMerg p$ is a fibration in $\Mwn$.
	\end{enumerate}
	Moreover, the adjoint functors $\FmMerg \dashv \UmMerg$ restrict to functors
	\[
		\FmMerg\colon \cofs{\Mwn} \to \cofs{\MMwn}, \quad \quad
		\UmMerg\colon \fibs{\MMwn} \to \fibs{\Mwn}
	\]
	exhibiting a Quillen adjunction between $\Mwn$ and $\MMwn$.
\end{thm}
\begin{proof}
	First of all, $\mMCpx$ is locally presentable.
	By commutativity of (\ref{eq:inclusions_directed_to_merge}), the functor $\FmMerg$ is a left adjoint and commutes with the functorial cylinders on $\mdCpx$ and $\mMCpx$ up to natural isomorphism, so it preserves pushout-products with their components.
	Because the generating sets of cofibrations and of anodyne extensions for $\MMwn$ are the image through $\FmMerg$ of those of $\Mwn$, the verification of the conditions of Proposition \ref{prop:weak_cisinski_olschok} for $\Mwn$ implies the verification of those same conditions for $\MMwn$.
	For the same reason, it is immediate that $\FmMerg$ and $\UmMerg$ form a Quillen adjunction between $\Mwn$ and $\MMwn$.
	The characterisation of cofibrations is a standard consequence of the small object argument, while the characterisation of fibrations follows by duality from the adjunction $\FmMerg \dashv \UmMerg$ and the fact that $\Jn$ is in the image of $\FmMerg$.
\end{proof}

\begin{comm}
	Intuitively, cofibrant objects are the ones that are ``freely generated'', starting from $\varnothing$, by cellular extensions (pushouts along boundary inclusions of atoms) and by markings (pushouts along top-markings of atoms); their underlying merge-complexes are the ``polygraphs'' or ``computads'' with respect to the algebra of merge-complexes.
	Note that, because the squares
\[\begin{tikzcd}
	{\bd{}{}\globe{n}} & {\globe{n}} \\
	{\bd{}{}U} & U
	\arrow[hook, from=1-1, to=1-2]
	\arrow[arloop->, "{\bd{}{}o}", from=1-1, to=2-1]
	\arrow[arloop->, "o", from=1-2, to=2-2]
	\arrow[hook, from=2-1, to=2-2]
	\arrow["\lrcorner"{anchor=center, pos=0.125, rotate=180}, draw=none, from=2-2, to=1-1]
\end{tikzcd}\]
	are pushouts for each $n \in \mathbb{N}$ and each $n$\nbd dimensional atom $U$, preserved by the left adjoint $\flatm{-}$, by the pasting law for pushouts we can replace any cellular extension with a ``globular'' cellular extension; indeed, unlike in $\Mwn$, and like in strict $\omega$\nbd categories \cite{lafont2010folk}, we can replace the set $\Ibd$ with the set $\set{\bd{}{}\globe{n} \incl \globe{n} \mid n \in \mathbb{N}}$ in defining our generating set of cofibrations.
\end{comm}

\noindent 
By Theorem \ref{thm:characterisation_of_fibrants}, an $\MMwn$\nbd fibrant marked merge-complex is precisely one whose underlying marked directed complex is an \inftyn\nbd category up to a choice of $\mInfl$\nbd algebra structure.
In fact, having both the structure of a marked merge-complex \emph{and} of a marked inflate-complex on the underlying marked directed complex, even when they do not interact, automatically produces weak composites, allowing us to simplify the characterisation of fibrants.

\begin{prop} \label{prop:algebraic_weak_composites}
	Let $(X, A)$ be a marked merge-complex which also has a structure of marked inflate-complex, and let $u$ be a round diagram in $X$.
	Then its merger $\mrgcom{u}$ is a weak composite of $u$.
\end{prop}
\begin{proof}
	Let $U$ be the shape of $u$.
	There is a unique subdivision 
	\[
		c_U\colon (U \celto \mrg{U}) \sd (\arr \pcyl{\bd{}{}U} U)
	\]
	which restricts to $\idd{U}$ on the input boundary and to $\mrg{}_U$ on the output boundary.
	Then $(\un u)c_U$ is a cell of type $u \celto \mrgcom{u}$, which is furthermore marked, since all top-dimensional cells in $\un u$ are degenerate, hence they are marked.
	This exhibits $u \mcelto \mrgcom{u}$.
	Dually, there is a subdivision
	\[
		c'_U\colon (\mrg{U} \celto U) \sd (\arr \pcyl{\bd{}{}U} U)
	\]
	such that $(\un u)c'_U$ exhibits $\mrgcom{u} \mcelto u$.
\end{proof}

\begin{cor} \label{cor:characterisation_of_fibrants_merge}
	Let $(X, A)$ be a marked merge-complex.
	The following are equivalent:
	\begin{enumerate}[label=(\alph*)]
		\item $(X, A)$ is fibrant in $\MMwn$;
		\item the underlying marked directed complex of $(X, A)$ admits a structure of marked inflate-complex such that $A = \satur{A}$.
	\end{enumerate}
\end{cor}

\noindent
When the separate structure of marked inflate-complex makes $(X, A)$ an essential \inftyn\nbd category, we also have a compatibility between composition and marked-invertibility.

\begin{lem} \label{lem:composition_of_marked_invertible}
	Let $(X, A)$ be a marked merge-complex which also has a structure of marked inflate-complex making it an essential \inftyn\nbd category, let $a\colon V \to X$ be a round diagram, and let $s\colon U \sd V$ be a subdivision.
	The following are equivalent:
	\begin{enumerate}[label=(\alph*)]
		\item $a$ is marked-invertible;
		\item $as$ is marked-invertible.
	\end{enumerate}
\end{lem}
\begin{proof}
	Suppose that $a\colon u \rdto v$ is marked-invertible.
	By Lemma \ref{lem:marked_invertibility_from_non-inverses}, we can pick cells $a_L$, $a_R \colon v \celto u$, $e\colon u \mcelto u$, $h\colon v \mcelto v$, and $z_L\colon a \cp{} a_L \mcelto e$, $z_R\colon h \mcelto a_R \cp{} a$.
	Then $z_L$ has shape 
	\[
		Z_L \eqdef (V \cp{} (\bd{}{+}V \celto \bd{}{-}V)) \celto (\bd{}{-}V \celto \bd{}{-}V).
	\]
	Consider the atom
	\[
		Z'_L \eqdef (U \cp{} (\bd{}{+}U \celto \bd{}{-}U)) \celto (\bd{}{-}U \celto \bd{}{-}U);
	\]
	there is a unique subdivision $t_L\colon Z'_L \sd Z_L$ restricting to $s$ on $U \submol \bd{}{-}Z'_L$.
	Then the cell $z'_L \eqdef z_L t_L$ has type $as \cp{} a'_L \celto e'$ for some cells $a'_L$ and $e'$, and by stability of $A$ under composition, both $z'_L$ and $e'$ are marked.
	Dually, we can find a subdivision $t_R$ such that $z'_R \eqdef z_R t_R$ has type $a'_R \cp{} as \celto h'$, and both $z'_R$ and $h'$ are marked.
	By Lemma \ref{lem:marked_invertibility_from_non-inverses}, we conclude that $as$ is marked-invertible.
	
	Conversely, suppose that $as\colon u' \rdto v'$ is marked-invertible; by the same result, we find $a'_L$, $a'_R\colon v' \celto u'$, $e'\colon u' \mcelto u'$, $h'\colon v' \mcelto v'$, $z'_L\colon as \cp{} a'_L \mcelto e'$, $z'_R\colon h' \mcelto a'_R \cp{} as$.
	Now, $z'_L$ has shape $Z'_L$; let $K'_L$ be the closed subset of $Z'_L$ obtained by removing its top element, as well as the top elements of the atoms $\bd{}{+}U \celto \bd{}{-}U$ and $\bd{}{-}U \celto \bd{}{-}U$, and let $K_L = t_L(K'_L)$ be the corresponding closed subset of $Z_L$.
	These are respectively isomorphic to
	\[
		U \pout{\bd{}{}(\bd{}{-}U)} \bd{}{-}U \text{ and } V \pout{\bd{}{}(\bd{}{-}V)} \bd{}{-}V.
	\]
	Then, $Z_L$ is the substitution of $K_L$ for $K'_L$ in $Z'_L$ along the restriction of $t_L$.
	Letting $k\colon K_L \to X$ be equal to $a$ on $U$ and to $u$ on $\bd{}{-}U$, from $\Gamma_\S$\nbd continuity of $X$ we obtain a unique cell $z_L\colon Z_L \to X$ such that the diagram
	\[\begin{tikzcd}
	{K'_L} & {K_L} \\
	{Z'_L} & {Z_L} & X
	\arrow[arloop->, "{\restr{t_L}{K'_L}}", from=1-1, to=1-2]
	\arrow[hook, from=1-1, to=2-1]
	\arrow[hook, from=1-2, to=2-2]
	\arrow["k", curve={height=-12pt}, from=1-2, to=2-3]
	\arrow[arloop->, "{t_L}", from=2-1, to=2-2]
	\arrow["{z'_L}", curve={height=30pt}, from=2-1, to=2-3]
	\arrow["\lrcorner"{anchor=center, pos=0.125, rotate=180}, draw=none, from=2-2, to=1-1]
	\arrow["{z_L}", dashed, from=2-2, to=2-3]
\end{tikzcd}\]
	commutes; this cell has type $a \cp{} a_L \celto e$ for some $a_L\colon v \celto u$ and $e\colon u \celto u$.
	Moreover, by stability of $A$ under decompositions in boundaries, both $z_L$ and $e$ are marked.
	Dually, from $z'_R$ we construct $z_R\colon h \mcelto a_R \cp{} a$ with $h$ marked, which proves that $a$ is marked-invertible.
\end{proof}

\noindent
In the next section, we will show that the structures of merge-complex and inflate-complex can be made compatible; the resulting algebraic structure will be the foundation of our semi-strict model of \inftyn\nbd categories.

\subsection{Merge-inflate-complexes and merge-\ntext-categories} \label{sec:mergeinflate}

\noindent
The existence of the $(\S, \C, \L)$ factorisation system on $\rdCpx_{\S\C\L}$ allows us to put together the algebra of units obtained by duality from collapses, with the algebra of composition obtained by duality from subdivisions.

\begin{dfn}[Merge-inflate-complex]
	A \emph{merge-inflate-complex} is a $\Gamma_\S$\nbd continuous presheaf on the category $\frdCpx_{\S\C\L}$ of finite regular directed complexes and local subdivision-collapses.
\end{dfn}

\noindent
We let $\MICpx$ denote the category $\PSh_{\Gamma_\S}(\frdCpx_{\S\C\L})$.
We have the usual embedding
\[
	\frdCpx_{\S\C\L} \incl \MICpx
\]
as well as an adjunction
\[\begin{tikzcd}
	\ICpx && \MICpx
	\arrow[""{name=0, anchor=center, inner sep=0}, "{\FMerg}", curve={height=-12pt}, from=1-1, to=1-3]
	\arrow[""{name=1, anchor=center, inner sep=0}, "{\UMerg}", curve={height=-12pt}, from=1-3, to=1-1]
	\arrow["\dashv"{anchor=center, rotate=-90}, draw=none, from=0, to=1]
\end{tikzcd}\]
and, letting $i$ now denote the inclusion $\frdCpx_{\C\L} \incl \frdCpx_{\S\C\L}$, the exact same proof as for Lemma \ref{lem:lan_takes_gamma_continuous_to_gammas_continuous}, replacing the factorisation system $(\S, \L)$ with the factorisation system $(\S, \C\L)$, gives us the following.

\begin{lem} \label{lem:lan_takes_inflate_to_merge_inflate}
	Let $X$ be an inflate-complex.
	Then the extension $\Lan i X$ is a merge-inflate-complex.
\end{lem}

\noindent
It follows from the resulting explicit description of $\FMerg\colon \ICpx \to \MICpx$, to which Lemma \ref{lem:presheaf_restriction_to_right_class} applies equally as to $\FMerg\colon \dCpx \to \MCpx$, that the diagram
\[\begin{tikzcd}
	\ICpx & \MICpx \\
	\dCpx & \MCpx
	\arrow["\FMerg", from=1-1, to=1-2]
	\arrow["\UInfl", from=1-1, to=2-1]
	\arrow["\UInfl", from=1-2, to=2-2]
	\arrow["\FMerg", from=2-1, to=2-2]
\end{tikzcd}\]
of categories and functors commutes.
We will adopt all the terminology relative to diagrams and the action of collapses and subdivisions from both merge-complexes and inflate-complexes.

\begin{dfn}[Marked merge-inflate-complex]
	A \emph{marked merge-inflate-complex} is a pair $(X, A)$ of a merge-inflate-complex $X$ and a set $A \subseteq \gr{0}{\cell X}$ such that $A = \clcom{(A \cup \dgn X)}$.
	Given marked merge-inflate-complexes $(X, A)$, $(Y, B)$, a \emph{morphism} $f\colon (X, A) \to (Y, B)$ is a morphism $f\colon X \to Y$ in $\MICpx$ such that $A \subseteq \invrs{f}B$.
\end{dfn}

\noindent
We let $\mMICpx$ denote the category of marked merge-inflate-complexes and their morphisms.
As was the case for the category $\mICpx$, we will not be particularly concerned with morphisms that strictly preserve units---they will not play a role in semi-strictification; it will suffice for us to observe that the adjunction $\FMerg \dashv \UMerg$ between $\ICpx$ and $\MICpx$ lifts to an adjunction
\[\begin{tikzcd}
	\mICpx && \mMICpx,
	\arrow[""{name=0, anchor=center, inner sep=0}, "{\FmMerg}", curve={height=-12pt}, from=1-1, to=1-3]
	\arrow[""{name=1, anchor=center, inner sep=0}, "{\UmMerg}", curve={height=-12pt}, from=1-3, to=1-1]
	\arrow["\dashv"{anchor=center, rotate=-90}, draw=none, from=0, to=1]
\end{tikzcd}\]
by letting $\FmMerg(X, A) \eqdef (\FMerg X, \clcom{A})$ and $\UmMerg(X, A) = (\UMerg X, A)$.
Indeed, to see that $\FmMerg$ is well-defined, it suffices to see that $\dgn (\FmMerg X) \subseteq \clcom{A}$ under the assumption that $\dgn X \subseteq A$, which follows easily from $(\S, \C)$ factorisation.
Moreover, we have a commutative diagram
\[\begin{tikzcd}
	\mICpx & \mMICpx \\
	\mdCpx & \mMCpx
	\arrow["\FmMerg", from=1-1, to=1-2]
	\arrow["\UmInfl", from=1-1, to=2-1]
	\arrow["\UmInfl", from=1-2, to=2-2]
	\arrow["\FmMerg", from=2-1, to=2-2]
\end{tikzcd}\]
of categories and functors, where \( \UmInfl \) sends \( (X, A) \) to \( (\UInfl X, A) \).
We are now ready to define our semi-strict model.

\begin{dfn}[Merge-$\infty$\nbd category]
	A marked merge-inflate-complex is a \emph{merge-$\infty$\nbd category} if it satisfies the following axiom.
\begin{itemize}
	\item (\emph{Completeness}). 
		Marked cells coincide with marked-invertible cells.
\end{itemize}
	For each $n \in \mathbb{N}$, a merge $\infty$\nbd category is a \emph{merge-$n$\nbd category} if, furthermore, every cell of dimension $> n$ is marked.
\end{dfn}

\noindent 
Let $n \in \Ninfty$.
Symbolically, given a merge-inflate-complex $X$, we can concisely characterise sets $A$ which make $(X, A)$ into a merge-$n$\nbd category by the equations
\[
	A = \clcom{(A \cup \dgn X \cup \gr{>n}{\cell X})} = \satur{A}.
\]

\begin{dfn}[Semi-strict functor of merge-$n$-categories]
	Let $(X, A)$, $(Y, B)$ be merge-$n$\nbd categories.
	A \emph{semi-strict functor} $f\colon (X, A) \to (Y, B)$ is a morphism of their underlying marked merge-complexes.
\end{dfn}

\noindent
Thus, as anticipated, we only require semi-strict functors to strictly preserve composition, and not units.

\begin{dfn}[Essential merge-$\infty$\nbd category]
	A marked merge-inflate-complex is an \emph{essential merge-$\infty$\nbd category} if it satisfies the following axiom.
\begin{itemize}
	\item (\emph{Essential completeness}).
		Every marked cell is marked-invertible, and every marked-invertible cell is marked-equivalent to a marked cell.
\end{itemize}
	For each $n \in \mathbb{N}$, an essential merge $\infty$\nbd category is an \emph{essential merge-$n$\nbd category} if, furthermore, every cell of dimension $> n$ is marked.
\end{dfn}

\noindent
From here onward, we fix $n \in \Ninfty$.
From Corollary \ref{cor:characterisation_of_fibrants_merge}, we have immediately the following.

\begin{prop} \label{prop:merge_n_categories_are_inftyn_categories}
	Let $(X, A)$ be a marked merge-inflate-complex.
	The following are equivalent:
	\begin{enumerate}[label=(\alph*)]
		\item $(X, A)$ is a merge-$n$\nbd category;
		\item the underlying marked merge-complex of $(X, A)$ is fibrant in $\MMwn$;
		\item the underlying marked inflate-complex of $(X, A)$ is an \inftyn\nbd category.
	\end{enumerate}
\end{prop}

\noindent
The following is also immediate from Proposition \ref{prop:algebraic_weak_composites}.

\begin{prop} \label{prop:essential_merge_n_categories_are_essential_inftyn_categories}
	Let $(X, A)$ be a marked merge-inflate-complex.
	The following are equivalent:
	\begin{enumerate}[label=(\alph*)]
		\item $(X, A)$ is an essential merge-$n$\nbd category;
		\item the underlying marked inflate-complex of $(X, A)$ is an essential \inftyn\nbd category.
	\end{enumerate}
\end{prop}

\noindent
Next, we show that the fibrant replacement of Corollary \ref{cor:naive_saturation_fibrant_replacement} lifts from essential \inftyn\nbd categories to essential merge-$n$\nbd categories.

\begin{prop} \label{prop:naive_saturation_merge_n_category}
	Let $(X, A)$ be an essential merge-$n$\nbd category.
	Then
	\begin{enumerate}
		\item $(X, \satur{A})$ is a merge-$n$\nbd category,
		\item the marking $(X, A) \acof (X, \satur{A})$ is a fibrant replacement in $\MMwn$.
	\end{enumerate}
\end{prop}
\begin{proof}
	For the first point, by Proposition \ref{prop:essential_merge_n_categories_are_essential_inftyn_categories} combined with Theorem \ref{thm:naive_completion_is_completion} and Proposition \ref{prop:merge_n_categories_are_inftyn_categories}, it suffices to show that $\satur{A} = \clcom{(\satur{A})}$, implying that $(X, \satur{A})$ is well-defined as a marked merge-inflate-complex; but this follows directly from the definition and Lemma \ref{lem:composition_of_marked_invertible}.
	The second point is then proved just as in Corollary \ref{cor:naive_saturation_fibrant_replacement}.
\end{proof}

\noindent
Before we move on to the proof of semi-strictification, we briefly discuss how, from a merge-$n$\nbd category, one obtains an algebraic model in the more traditional sense, that is, an algebra of \emph{globular} units and compositions on the underlying $\omega$\nbd graph of globular cells.
For now, we will not attempt to formally (let alone functorially) produce an algebraic model \emph{\`a la} Batanin--Leinster, although we consider this an obvious direction for future developments; we will simply describe one way to define globular composition operations from the combination of units and composition of round diagrams.

Given an inflate-complex $X$, a pasting diagram $u$ in $X$, and $n \in \mathbb{N}$, we let
\[
	\Rnd_n u \eqdef u \cp{n} \un(\bd{n}{+}u)
\]
and, by downward recursion on $k < n$,
\[
	\Rnd^n_n u \eqdef u, \quad \quad \Rnd^n_k u \eqdef \Rnd_{n-1}(\Rnd^{n-1}_k u).
\]
It is straightforward to show that, if $u$ is an $n$\nbd dimensional pasting diagram and $k \leq n$ is such that $\bd{k}{-}u$, $\bd{k}{+}u$ are round, then $\Rnd^n_k u$ is an $n$\nbd dimensional round diagram.
Moreover, this operation satisfies the equations
\begin{equation} \label{eq:rounding_relations}
	\bd{m}{\a}\Rnd^n_k u = \begin{cases}
		\Rnd^n_k u 
		& \text{if $n \leq m$,} \\
		\Rnd^m_k \bd{m}{\a}u
		& \text{if $k \leq m < n$,} \\
		\bd{m}{\a}u
		& \text{if $m < k$}.
	\end{cases}
\end{equation}
Now, suppose that $X$ is a merge-inflate-complex, let $n \in \mathbb{N}$, $k < n$, and let $u, v\colon \globe{n} \to X$ be globular cells such that $\bd{k}{+}u = \bd{k}{-}v$.
Then, we let
\[
	u *_k v \eqdef \glcom{\Rnd_{k+1}^n (u \cp{k} v)}\colon \globe{n} \to X.
\]
Using the equations (\ref{eq:rounding_relations}), we can check that the operation $- *_k -$ satisfies the same compatibility with boundaries as $- \cp{k} -$, that is,
\[
	\bd{m}{\a}(u *_k v) = \begin{cases}
		\bd{m}{\a} u *_k \bd{m}{\a}v
		& \text{if $k < m$}, \\
		\bd{m}{-} u
		& \text{if $m = k$, $\a = -$,} \\
		\bd{m}{+} v
		& \text{if $m = k$, $\a = +$,} \\
		\bd{m}{\a}u = \bd{m}{\a}v
		& \text{if $m < k$.}
	\end{cases}
\]
If $(X, A)$ is a merge-$n$\nbd category, using the theory of marked-equivalences of contexts along the same lines as \cite[Section 5]{chanavat2024equivalences}, we can show that these operations satisfy associativity and unitality up to marked-equivalence, in the appropriate sense.
However, the only equations that hold \emph{strictly} are the associativity equations for the operation $- *_{n-1} -$ on globular $n$\nbd cells, so at the level of globular composition, we achieve a very mild strictification. 

From a practical and computational point of view, however, we would argue that globular composition is relatively useless, and that the kind of generalised strict associativity and interchange that we achieve is far more useful: as we know, pasting diagrams even form a \emph{strict} $\omega$\nbd category, and there is no need at all to pass to globular composites until one is finished pasting together diagrams of more general shapes.
It suffices to take a look at the 1\nbd categorical commutative diagrams that appear in this or any other article---these are a truncated form of 2\nbd dimensional pasting diagrams, in which either a unique 2\nbd cells exists between boundaries or does not---or, more in general, 2\nbd categorical pasting diagrams, to observe that ones that are pastings of globe-shaped diagrams are vanishingly rare; instead, generalised pastings of commutative squares, triangles, sometimes more general polygons are abundant, and more often than not, they are already round.
All such diagrams can be constructed as pasting diagrams in a merge-$n$\nbd category, and the generalised associativity and interchange of round compositions ensure that the order in which different round portions of these diagrams are composed is irrelevant.

\subsection{Semi-strictification} \label{sec:simpson}

\noindent
Everything is now in place to prove our main theorem, by a similar argument as the one employed in \cite[Section 6.2]{henry2018regular}.

\begin{lem} \label{lem:free_merge-complex_on_inftyn_is_essential_inftyn}
	Let $(X, A)$ be an essential \inftyn\nbd category.
	Then $\FmMerg(X, A)$ is an essential merge-$n$\nbd category.
\end{lem}
\begin{proof}
	By the results of the previous section, from the inflate-complex structure on $(X, A)$, we obtain a merge-inflate-complex structure on $\FmMerg(X, A)$.
	Let $e\colon U \to \FMerg X$ be a marked cell of type $u' \celto v'$ in $\FmMerg(X, A)$.
	Then $e$ is uniquely of the form $[h, s]$ for a marked round diagram $h\colon V \to X$ and a subdivision $s\colon U \sd V$.
	By essential completeness of $(X, A)$, $h\colon u \rdto v$ is marked-invertible, so we can find a weak inverse $h^*\colon v \celto u$ of shape $V^*$ and witnesses of $z_L\colon h \cp{} h^* \mcelto \un u$, $z_R\colon h^* \cp{} h \mcelto \un v$, $z'_L\colon \un u \mcelto h \cp{} h^*$, $z'_R\colon \un v \mcelto h^* \cp{} h$, whose shapes are
	\[
		Z_L \eqdef (V \cp{} V^*) \celto (\arr \pcyl{\bd{}{}(\bd{}{-}V)} \bd{}{-}V), \quad
		Z_R \eqdef (V^* \cp{} V) \celto (\arr \pcyl{\bd{}{}(\bd{}{+}V)} \bd{}{+}V),
	\]
	and their duals $Z^*_L$ and $Z^*_R$.
	Let $U^* \eqdef \bd{}{+}U \celto \bd{}{-}U$.
	We have subdivisions
	\begin{align*}
		s_L\colon \left((U \cp{} U^*) \celto (\arr \pcyl{\bd{}{}(\bd{}{-}U)} \bd{}{-}U)\right)
		& \sd \left((V \cp{} V^*) \celto (\arr \pcyl{\bd{}{}(\bd{}{-}V)} \bd{}{-}V)\right), \\
		s_R\colon \left((U^* \cp{} U) \celto (\arr \pcyl{\bd{}{}(\bd{}{+}U)} \bd{}{+}U)\right)
		& \sd \left((V^* \cp{} V) \celto (\arr \pcyl{\bd{}{}(\bd{}{+}V)} \bd{}{+}V)\right)
	\end{align*}
	restricting to $s$ on $U \submol \bd{}{-}Z_L, \bd{}{-}Z_R$, and to the unique subdivisions fitting in the factorisations
	\[\begin{tikzcd}
	{\arr \pcyl{\bd{}{}(\bd{}{\a}U)} \bd{}{\a}U } & {\arr \pcyl{\bd{}{}(\bd{}{\a}V)} \bd{}{\a}V} \\
	{\bd{}{\a}U} & {\bd{}{\a}V}
	\arrow[arloop->, dashed, from=1-1, to=1-2]
	\arrow["{\tau_{\bd{}{}(\bd{}{\a}U)}}", two heads, from=1-1, to=2-1]
	\arrow["{\tau_{\bd{}{}(\bd{}{\a}V)}}", two heads, from=1-2, to=2-2]
	\arrow[arloop->, "{\restr{s}{\bd{}{\a}U}}", from=2-1, to=2-2]
\end{tikzcd}\]
	on $\bd{}{+}Z_L$ and $\bd{}{+}Z_R$ for $\a = -$ and $\a = +$, respectively.
	Then, by construction, $[z_L, s_L]$ has type $[h, s]\cp{}e^* \mcelto \un u'$ while $[z_R, s_R]$ has type $e^* \cp{} [h, s] \mcelto \un v'$ for some cell $e^*$ decomposing to $h^*$.
	From subdivisions $s'_L$, dual to $s_L$, and $s'_R$, dual to $s_R$, we also obtain witnesses $[z'_L, s'_L]\colon \un u' \mcelto [h, s] \cp{} e^*$ and $[z'_R, s'_R]\colon \un v' \mcelto e^* \cp{} [h, s]$.
	This proves that $e$ is marked-invertible.

	Next, suppose that $e = [h, s]$ is marked-invertible, with $e$, $h$, $s$ as before, except $h$ need not be marked.
	By passing to mergers, we may assume that its weak inverse $e^*$, as well as witnesses of marked-invertibility $w_L\colon e \cp{} e^* \mcelto \un u'$, $w_R\colon e^* \cp{} e \mcelto \un v'$, $w'_L\colon \un u' \mcelto e \cp{} e^*$, $w'_R\colon \un v' \mcelto e^* \cp{} e$ are all single cells.
	Then $e^*$ is uniquely of the form $[h^*, s']$ for some round diagram $h^*$ and subdivision $s'$, and
	\[
		w_L = [z_L, s_L], \quad w_R = [z_R, s_R], \quad w'_L = [z'_L, s'_L], \quad w'_R = [z'_R, s'_R]
	\]
	for some subdivisions $s_L$, $s_R$, $s'_L$, $s'_R$ and marked round diagrams $z_L$, $z_R$, $z'_L$, $z'_R$ in $(X, A)$ which exhibit $h \cp{} h^* \mrdto \un u$, $h^* \cp{} h \mrdto \un v$, $\un u \mrdto h \cp{} h^*$, and $\un v \mrdto h^* \cp{} h$, respectively.
	Then $h$ is marked-invertible.
	By essential completeness of $(X, A)$, there exists a marked weak composite $\mrg{h}$ of $h$, exhibited by marked cells $k\colon h \mcelto \mrg{h}$ and $k'\colon \mrg{h} \mcelto h$.
	There are unique subdivisions
	\[
		t\colon (U \celto \mrg{U}) \sd (V \celto \mrg{V}), \quad \quad
		t'\colon (\mrg{U} \celto U) \sd (\mrg{V} \celto V)
	\]
	restricting to $s$ on the input and the output boundary, respectively.
	Then $[k, t]$ and $[k', t']$ are marked cells in $\FmMerg(X, A)$ exhibiting marked-equivalence of $e$ with a marked cell decomposing to $\mrg{h}$.
	Finally, if all cells in dimension $> n$ are marked in $(X, A)$, it is clear from the definition that all cells in dimension $> n$ are marked in $\FmMerg(X, A)$.
	This completes the proof.
\end{proof}

\begin{rmk}
	Note that, even if $(X, A)$ is an \inftyn\nbd category, there is no guarantee that $\FmMerg(X, A)$ is more than an essential \inftyn\nbd category, since a marked-invertible round diagram in $(X, A)$ may have non-invertible top-dimensional cells---think of a section-retraction pair in dimension $1$---so its ``free'' merger in $\FmMerg(X, A)$ will be marked-invertible, but not marked.
\end{rmk}

\noindent
By Lemma \ref{lem:free_merge-complex_on_inftyn_is_essential_inftyn}, paired with Proposition 
\ref{prop:naive_saturation_merge_n_category}, we know that, if $(X, A)$ is an essential \inftyn\nbd category, then
\[
	\FmMerg(X, A) \acof (\FMerg X, \satur{\clcom{A}})
\]
is a fibrant replacement in $\MMwn$.
We let
\[
	\sigma_{(X, A)}\colon (X, A) \to \UmMerg(\FMerg X, \satur{\clcom{A}})
\]
denote its transpose morphism in $\mdCpx$; its underlying morphism of directed complexes is the same as the underlying morphism of the unit $\eta^\Merg_{(X, A)}$ of the adjunction $\FmMerg \dashv \UmMerg$, sending a cell $u\colon U \to X$ to $[u, \idd{U}]$.

Moreover, if $f\colon (X, A) \to (Y, B)$ is a morphism of essential \inftyn\nbd categories, then $\FMerg f$, being a morphism of the underlying essential \inftyn\nbd categories, sends marked-invertible cells to marked-invertible cells by Proposition 
\ref{prop:morphisms_preserve_marked_invertibility}, so it also determines a semi-strict functor
\[
	\FMerg f\colon (\FMerg X, \satur{\clcom{A}}) \to (\FMerg Y, \satur{\clcom{B}})
\]
between the fibrant replacements.
By naturality of $\eta^\Merg$, we also have a ``naturality square''
\[\begin{tikzcd}[column sep=large]
	{(X, A)} & {(Y, B)} \\
	{\UmMerg(\FMerg X, \satur{\clcom{A}})} & {\UmMerg(\FMerg Y, \satur{\clcom{B}}).}
	\arrow["f", from=1-1, to=1-2]
	\arrow[hook, "{\sigma_{(X, A)}}", from=1-1, to=2-1]
	\arrow[hook, "{\sigma_{(Y, B)}}", from=1-2, to=2-2]
	\arrow["{\UMerg \FMerg f}", from=2-1, to=2-2]
\end{tikzcd}\]

\begin{thm} \label{thm:embedding_into_merge_n_category}
	Let $(X, A)$ be an \inftyn\nbd category.
	Then $\sigma_{(X, A)}$ is an equivalence of \inftyn\nbd categories.
\end{thm}
\begin{proof}
	It follows from Lemma \ref{lem:free_merge-complex_on_inftyn_is_essential_inftyn}, Proposition \ref{prop:naive_saturation_merge_n_category}, and Proposition \ref{prop:merge_n_categories_are_inftyn_categories} that the domain and codomain of $\sigma_{(X, A)}$ are both \inftyn\nbd categories, so by Theorem \ref{thm:characterisation_of_equivalences} it suffices to show that $\sigma_{(X, A)}$ is an essential acyclic fibration.
	Because there are no non-trivial subdivisions of 0\nbd cells, $\sigma_{(X, A)}$ is surjective on 0\nbd cells.
	Next, suppose that $u$, $v$ are parallel round diagrams of shapes $U$, $V$ in $X$, and let $a\colon [u, \idd{U}] \celto [v, \idd{V}]$ be a cell in $\FMerg X$.
	Then $a$ is uniquely of the form $[a', s]$ for some round diagram $a'$ of shape $W$ in $X$ and subdivision $s\colon (U \celto V) \sd W$ restricting to the identity on the boundary.
	Since $(X, A)$ has weak composites, there exists a cell $\mrg{a'}$ parallel to $a'$ and a marked cell $k\colon a' \mcelto \mrg{a'}$.
	Let $t$ be the unique subdivision
	\[
		\left((U \celto V) \celto (U \celto V)\right) \sd (W \celto \mrg{W})
	\]
	restricting to $s$ on the input boundary; then $t$ restricts to the identity on the output boundary.
	It follows that $[k, t]\colon [a', s] \mcelto [\mrg{a'}, \idd{\mrg{W}}]$ exhibits the marked-equivalence of $a$ with a cell in the image of $\sigma_{(X, A)}$.
	Finally, if $a$ is marked, then, as in the proof of Lemma \ref{lem:free_merge-complex_on_inftyn_is_essential_inftyn}, one shows that $a'$ must be marked-invertible, in which case $\mrg{a'}$ is also marked-invertible.
	By completeness of $(X, A)$, it follows that $\mrg{a'}$ is marked.
\end{proof}

\begin{rmk}
	Since $\sigma_{(X, A)}$ is evidently a monomorphism, hence a cofibration, by 
	Lemma \ref{lem:acyclic_cofib_is_cofib_equivalence} it is in fact an acyclic cofibration.
\end{rmk}

\begin{comm}
	\emph{Sensu stricto}, Theorem \ref{thm:embedding_into_merge_n_category} is the core ``semi-strictification theorem'', in the sense that it is, formally, the higher-dimensional analogue of Mac Lane's strictification theorem for bicategories: it shows that ``freely adding composites'' embeds an \inftyn\nbd category into a merge-$n$\nbd category via an equivalence.
	However, while this realises an embedding of \inftyn\nbd categories into merge-$n$\nbd categories, it does not yet show that the embedding is conservative in the sense that equivalences of merge-$n$\nbd categories are essentially the same as equivalences of \inftyn\nbd categories; for this, we need the results of the rest of this section.
\end{comm}

\begin{lem} \label{lem:counit_acyclic_fibration}
	Let $(X, A)$ be an $\MMwn$-fibrant marked merge-complex.
	Then the counit $\varepsilon^\Merg$ of the adjunction $\FmMerg \dashv \UmMerg$ evaluated at $(X, A)$ factors as
\[\begin{tikzcd}
	{\FmMerg\UmMerg(X, A)} & {(\FMerg\UMerg X, \satur{\clcom{A}})} \\
	& {(X, A)}
	\arrow["\sim", hook, from=1-1, to=1-2]
	\arrow["{\varepsilon^\Merg_{(X, A)}}", curve={height=12pt}, from=1-1, to=2-2]
	\arrow["\wr", two heads, from=1-2, to=2-2]
\end{tikzcd}\]
	where the horizontal morphism is an acyclic cofibration and the vertical morphism is an acyclic fibration.
\end{lem}
\begin{proof}
	By Theorem \ref{thm:characterisation_of_fibrants} and Theorem \ref{thm:model_structures_on_marked_merge}, $\UmMerg(X, A)$ admits a structure of inflate-complex making it an \inftyn\nbd category.
	If we endow it with this structure, by Lemma \ref{lem:free_merge-complex_on_inftyn_is_essential_inftyn}, $\FmMerg\UmMerg(X, A)$ becomes an essential merge-$n$\nbd category, and by Proposition 
	\ref{prop:naive_saturation_merge_n_category}, the marking $\FmMerg\UmMerg(X, A) \acof (\FMerg\UMerg X, \satur{\clcom{A}})$ whose underlying morphism is the identity on $\FMerg\UMerg X$ is a fibrant replacement, hence an acyclic cofibration.
	Now, a cell $u\colon U \to \FMerg\UMerg X$ is uniquely represented as a pair $[u', s]$ of a round diagram $u'\colon U' \to X$ and a subdivision $s\colon U \sd U'$, and the underlying morphism of $\varepsilon^\Merg_{(X, A)}$ sends it to the cell $u's$ of $X$.
	To show that this determines a morphism of marked merge-complexes $(\FMerg\UMerg X, \satur{\clcom{A}}) \to (X, A)$, it suffices to show that, if $[u', s]$ is marked-invertible in $(\FMerg\UMerg X, \clcom{A})$, then $u's$ is marked in $(X, A)$.
	Since both $(\FMerg\UMerg X, \clcom{A})$ and $(X, A)$ have underlying essential \inftyn\nbd categories, if $[u', s]$ is marked-invertible, then by Proposition \ref{prop:morphisms_preserve_marked_invertibility} $u's$ is marked-invertible, and by completeness of $(X, A)$ it is marked.

	This proves that the factorisation exists, so it remains to show that the vertical morphism is an acyclic fibration.
	Consider a commutative square
	\[\begin{tikzcd}
	{\flatm{\bd{}{}U}} & {(\FMerg\UMerg X, \satur{\clcom{A}})} \\
	{\flatm{U}} & {(X, A).}
	\arrow["{(u, v)}", from=1-1, to=1-2]
	\arrow["{\flatm{\bdmap_U}}", hook, from=1-1, to=2-1]
	\arrow[two heads, from=1-2, to=2-2]
	\arrow["a", from=2-1, to=2-2]
\end{tikzcd}\]
	Then $(u, v)$ classifies two parallel round diagrams $[u', s]$ and $[v', t]$ in $\FMerg\UMerg X$, where $s$ and $t$ are subdivisions matching on the boundaries, and $a\colon u's \celto v't$.
	Then, letting $\bd{}{-}V \eqdef s(\bd{}{-}U)$, $\bd{}{+}V \eqdef t(\bd{}{+}U)$, and $V \eqdef \bd{}{-}V \celto \bd{}{+}V$, by $\Gamma_\S$\nbd continuity we can factorise the square through the pushout
	\[\begin{tikzcd}
	{\flatm{\bd{}{}U}} & {\flatm{\bd{}{}V}} & {(\FMerg\UMerg X, \satur{\clcom{A}})} \\
	{\flatm{U}} & {\flatm{V}} & {(X, A)}
	\arrow[arloop->, "{(s, t)}", from=1-1, to=1-2]
	\arrow["{\flatm{\bdmap_U}}", hook, from=1-1, to=2-1]
	\arrow["{(u', v')}", from=1-2, to=1-3]
	\arrow["{\flatm{\bdmap_V}}", hook, from=1-2, to=2-2]
	\arrow[two heads, from=1-3, to=2-3]
	\arrow[arloop->, from=2-1, to=2-2]
	\arrow["\lrcorner"{anchor=center, pos=0.125, rotate=180}, draw=none, from=2-2, to=1-1]
	\arrow["{a'}", dashed, from=2-2, to=2-3]
\end{tikzcd}\]
	where now $a'\colon u' \celto v'$ can be directly lifted to $[a', \idd{V}]$.
	Next, consider a commutative square
	\[\begin{tikzcd}
	{\flatm{U}} & {(\FMerg\UMerg X, \satur{\clcom{A}})} \\
	{\mrk{U}} & {(X, A)}
	\arrow["{[u, s]}", from=1-1, to=1-2]
	\arrow["{\m_U}", from=1-1, to=2-1]
	\arrow[two heads, from=1-2, to=2-2]
	\arrow["us", from=2-1, to=2-2]
\end{tikzcd}\]
	where $s\colon U \sd U'$ is a subdivision and $u\colon U' \to X$ a round diagram in $X$.
	Since $us$ is marked in $(X, A)$, it is marked-invertible, and because $(X, A)$ satisfies the conditions of Lemma \ref{lem:composition_of_marked_invertible}, it follows that $u$ is marked-invertible.
	Then, by the same result, $[u, s] = [u, \idd{U'}]s$ is also marked-invertible in $(\FMerg\UMerg X, \satur{\clcom{A}})$.
	By completeness, $[u, s]$ is marked, which concludes the proof.
\end{proof}

\begin{thm} \label{thm:quillen_equivalence}
	The Quillen adjunction $(\FmMerg, \UmMerg)$ is a Quillen equivalence between $\Mwn$ and $\MMwn$.
\end{thm}
\begin{proof}
	Let $(X, A)$ be any marked directed complex, which is automatically cofibrant, and let $(X, A) \acof (X', A')$ be a fibrant replacement.
	By Theorem \ref{thm:characterisation_of_fibrants}, we may endow $(X', A')$ with a structure of inflate-complex making it an \inftyn\nbd category.
	Then $\FmMerg(X, A) \acof \FmMerg(X', A')$ is an acyclic cofibration in $\MMwn$, and by Lemma \ref{lem:free_merge-complex_on_inftyn_is_essential_inftyn}, the codomain is an essential merge-$n$\nbd category.
	By Proposition \ref{prop:naive_saturation_merge_n_category}, we may post-compose it with the acyclic cofibration $\FmMerg(X', A') \acof (\FMerg X', \satur{\clcom{A'}})$ to obtain a fibrant replacement $\varphi\colon \FmMerg(X, A) \acof (\FMerg X', \satur{\clcom{A'}})$.
	Then, in the commutative diagram
	\[\begin{tikzcd}
	{(X, A)} & {(X', A')} \\
	{\UmMerg\FmMerg(X, A)} & {\UmMerg\FmMerg(X', A')} & {\UmMerg(\FMerg X', \satur{\clcom{A'}})}
	\arrow["\sim", hook, from=1-1, to=1-2]
	\arrow[hook, "{\eta^\Merg_{(X, A)}}", from=1-1, to=2-1]
	\arrow[hook, "{\eta^\Merg_{(X', A')}}", from=1-2, to=2-2]
	\arrow["{\sigma_{(X', A')}}", curve={height=-12pt}, from=1-2, to=2-3]
	\arrow[from=2-1, to=2-2]
	\arrow["{\UmMerg(\varphi)}", curve={height=30pt}, from=2-1, to=2-3]
	\arrow[from=2-2, to=2-3]
\end{tikzcd}\]
	the path going down and then right is the transpose of $\varphi$, while the path going right and then diagonally down is an equivalence by Theorem \ref{thm:embedding_into_merge_n_category}.
	This proves one of the conditions of Proposition \ref{prop:quillen_equivalence_criterion}.
	Since there is no need for cofibrant replacement in $\Mwn$, the second condition follows from Lemma \ref{lem:counit_acyclic_fibration} by 2-out-of-3 for equivalences.
	This completes the proof.
\end{proof}

\begin{cor} \label{cor:equivalences_of_merge-n-categories}
	Let $f\colon (X, A) \to (Y, B)$ be a semi-strict functor of merge-$n$\nbd categories.
	The following are equivalent:
	\begin{enumerate}[label=(\alph*)]
		\item $f$ is an equivalence in $\MMwn$;
		\item $\UMerg f$ is an equivalence of \inftyn\nbd categories.
	\end{enumerate}
\end{cor}
\begin{proof}
	Right functors in a Quillen equivalence reflect equivalences between fibrants, see \cite[Proposition 2.4.5]{henry2020weak}.
\end{proof}

\noindent
We conclude our article with a statement summarising what we achieved.

\begin{thm}[Semi-strictification] \label{thm:main_thm}
	Let $f\colon (X, A) \to (Y, B)$ be a functor of \inftyn\nbd categories.
	Then there exists a square
	\[\begin{tikzcd}[column sep=large]
	{(X, A)} & {(Y, B)} \\
	{\UmMerg(\FMerg X, \satur{\clcom{A}})} & {\UmMerg(\FMerg Y, \satur{\clcom{B}}).}
	\arrow["f", from=1-1, to=1-2]
	\arrow[hook, "{\sigma_{(X, A)}}", from=1-1, to=2-1]
	\arrow[hook, "{\sigma_{(Y, B)}}", from=1-2, to=2-2]
	\arrow["{\UMerg\FMerg f}", from=2-1, to=2-2]
\end{tikzcd}\]
	where
	\begin{enumerate}
		\item $\FMerg f\colon (\FMerg X, \satur{\clcom{A}}) \to (\FMerg Y, \satur{\clcom{B}})$ is a semi-strict functor of merge-$n$\nbd categories,
		\item $\sigma_{(X, A)}$ and $\sigma_{(Y, B)}$ are equivalences of \inftyn\nbd categories, in particular acyclic cofibrations.
	\end{enumerate}
	Moreover, the adjoint pair $\FmMerg \dashv \UmMerg$ determines a Quillen equivalence between weak model categories $\Mwn$ and $\MMwn$ such that
	\begin{enumerate}
		\item the category of \inftyn\nbd categories and functors is equivalent to the category of fibrant objects in $\Mwn$,
		\item every morphism of fibrant objects in $\MMwn$ is a semi-strict functor of merge-$n$\nbd categories up to acyclic fibrations over its domain and codomain.
	\end{enumerate}
	In particular, a semi-strict functor of merge-$n$\nbd categories is an equivalence if and only if its underlying functor of \inftyn\nbd categories is an equivalence.
\end{thm}

\bibliographystyle{alpha}
\small \bibliography{main}

\end{document}